\documentclass[a4paper, 11pt]{article}

\providecommand{\CE}{{\cal E}}

\providecommand{\bbR}{\mathbb{R}}

\providecommand*{\N}[1]{\left\|{#1}\right\|} 

\usepackage{amssymb,amsmath,amscd,amsfonts,amsthm,bbm,mathrsfs,yhmath}
\usepackage{hyperref}
\usepackage{url}
\usepackage[utf8]{inputenc} 
\usepackage{cleveref}
\usepackage{algorithm}
\usepackage{algpseudocode}
\usepackage{graphicx}
\usepackage{color}
\usepackage{listings}
\usepackage{physics}
\usepackage{nicefrac}
\usepackage{subcaption}
\usepackage{geometry}
\usepackage{multirow}

\usepackage[affil-it]{authblk}

\usepackage[table]{xcolor}
\usepackage{fullpage}
\usepackage[table]{xcolor}
\newcommand{\normmsq}[1]{\left\| #1 \right\|_2^2}
\newcommand{\normm}[1]{\left\| #1 \right\|_2}
\newcommand{\inner}[2]{\left< #1 , #2 \right>}
\newcommand{\cV}{{\mathscr{V}}^*}
\newcommand{\dist}[1]{\operatorname{dist}(#1)}

\newcommand{\Ep}[1]{\mathbb{E}\left(#1\right)}
\newcommand{\cP}{\mathcal{P}_R}
\newcommand{\cv}{\mathcal{V}}

\newcommand{\cT}{\mathcal{T}}
\newcommand{\cC}{\mathcal{C}}

\newcommand{\cc}{C_{\mathrm{large}}}
\newcommand{\R}{\tilde{R}}
\newcommand{\rr}{\dot{r}}

\newcommand{\Cone}{\vartheta}
\newcommand{\Ctwo}{C_1}
\newcommand{\Cthree}{C_2}
\newcommand{\Cfour}{C_3}
\newcommand{\Cfive}{C_4}
\newcommand{\Csix}{C_5}
\newcommand{\Cseven}{C_6}
\newcommand{\Ceight}{C_7}
\newcommand{\Cnine}{C_8}
\newcommand{\Cten}{C_9}
\newcommand{\Celeven}{C_{10}}
\newcommand{\Ctwelve}{C_{11}}
\newcommand{\Cthirteen}{C_{12}}
\newcommand{\Cfourteen}{C_{13}}
\newcommand{\Cfifteen}{C_{14}}
\newcommand{\Csixteen}{C_{15}}
\newcommand{\Cseventeen}{C_{16}}
\newcommand{\Ceightteen}{C_{17}}
\newcommand{\Cnineteen}{C_{18}}
\newcommand{\Ctwenty}{C_{19}}
\newcommand{\Ctwentyone}{C_{20}}
\newcommand{\Ctwentytwo}{C_{21}}
\newcommand{\Ctwentythree}{C_{22}}
\newcommand{\Ctwentyfour}{C_{23}}
\newcommand{\Ctwentyfive}{C_{24}}
\newcommand{\Ctwentysix}{C_{25}}

\newcommand{\Cctr}{C_{\mathrm{a}}}
\newcommand{\Csmall}{C_{\mathrm{b}}}
\newcommand{\Cexp}{C_{\mathrm{exp}}}
\newcommand{\Cerr}{C_{\mathrm{err}}}
\newcommand{\Cd}{C_{\mathrm{sup}}}

\newcommand{\sml}{\ell}
\newcommand{\mm}{Z}

\usepackage[colorinlistoftodos,bordercolor=orange,backgroundcolor=orange!20,linecolor=orange,textsize=scriptsize]{todonotes}

\newif\ifrevised
\newcommand{\revised}[1]{%
	\ifrevised
	{\color{blue} #1} \color{black} %
	\else
	#1%
	\fi}
\revisedfalse

\newcommand{\overbar}[1]{\makebox[0pt]{$\phantom{#1}\mkern 1.5mu\overline{\mkern-1.5mu\phantom{#1}\mkern-1.5mu}\mkern 1.5mu$}#1}
\renewcommand{\underbar}[1]{\makebox[0pt]{$\phantom{#1}\mkern 1.5mu\underline{\mkern-1.5mu\phantom{#1}\mkern-1.5mu}\mkern 1.5mu$}#1}

\newcommand{\globmin}{v^*}
\renewcommand{\CE}{f}

\newcommand{\empmeasure}[1]{\widehat\rho_{#1}^N}

\newcommand{\conspoint}[1]{v_{\alpha}({#1})}

\renewcommand{\underbar}[1]{\makebox[0pt]{$\phantom{#1}\mkern 1.5mu\underline{\mkern-1.5mu\phantom{#1}\mkern-1.5mu}\mkern 1.5mu$}#1}

\newcommand{\zonea}{\mathcal{Z}_1}
\newcommand{\zoneb}{\mathcal{Z}_2}

\newcounter{results}[section]

\theoremstyle{plain}
\newtheorem{theorem}[results]{Theorem}

\newtheorem{lemma}[results]{Lemma}
\newtheorem{proposition}[results]{Proposition}

\newtheorem{assumption}[results]{Assumption}

\theoremstyle{remark}
\newtheorem{remark}[results]{Remark}

\theoremstyle{definition}

\title{
A PDE Framework of Consensus-Based Optimization for Objectives with Multiple Global Minimizers
}

\date{March, 5 2024}

\author[1,2,3]{Massimo Fornasier\thanks{Email: \texttt{massimo.fornasier@cit.tum.de} }}
\author[4,5]{Lukang Sun\thanks{Email: \texttt{lukang.sun@kaust.edu.sa} }}

\affil[1]{Technical University of Munich, School of Computation, Information and Technology, Department of Mathematics, Munich, Germany}
\affil[2]{Munich Center for Machine Learning, Munich, Germany}
\affil[3]{Munich Data Science Institute, Germany}
\affil[4]{King Abdullah University of Science and Technology, Thuwal, Saudi Arabia}
\affil[5]{KAUST AI Initiative, Thuwal, Saudi Arabia}

\begin{document}
	\maketitle
	\begin{abstract}
Consensus-based optimization (CBO) is an agent-based derivative-free method for non-smooth global optimization that has been introduced in 2017 \cite{B1-pinnau2017consensus}, leveraging a surprising
interplay between stochastic exploration and Laplace’s principle. In addition to its versatility and effectiveness in handling high-dimensional, \revised{non-convex,} and \revised{non-smooth} optimization problems, this approach lends itself well to theoretical analysis. Indeed, its dynamics is governed by a degenerate nonlinear Fokker--Planck equation, whose \revised{large time} behavior explains the convergence of the method. Recent results \cite{carrillo2018analytical,B1-fornasier2021global} provide guarantees of convergence under the restrictive assumption of unique global minimizer for the objective function. In this work, we \revised{propose} a novel and simple variation of CBO to tackle \revised{non-convex optimization problems} with multiple \revised{global} minimizers. {Despite the simplicity of this new model, its analysis is particularly challenging because of its nonlinearity and nonlocal nature.}
We prove existence of solutions of the corresponding nonlinear Fokker--Planck equation and we show exponential concentration in time to the set of minimizers made of multiple smooth, convex, and compact components. {Our proofs require combining several ingredients, such as delicate geometrical arguments, new  variants of a  quantitative Laplace principle,  {\it ad hoc} regularizations and approximations, and  regularity theory for parabolic equations.} Ultimately, this result suggests that the corresponding CBO algorithm, formulated as an Euler--Maruyama discretization of the underlying empirical stochastic process, tends to converge to \revised{multiple} global minimizers.
	\end{abstract}

	

\tableofcontents

	\section{Introduction}	
{ 
Computing efficiently minimal points of a potentially non-convex and non-smooth objective functions is a highly relevant problem across diverse applications spanning applied mathematics, science, technology, engineering, and machine learning. The global optimization task can be succinctly defined as:

\begin{equation}
\label{eq:B1-optimization_problem}
(\arg)\min_{v\in \Omega} f(v)\revised{.}
\end{equation}

Here, we aim to find the minimal values $f(v)$ with the variable $v$ within the state space $\Omega$, a subset of $\mathbb{R}^{d}$, where $d$ may in some case be notably large. Below we are interested in unconstrained optimizations only, for which $\Omega = \mathbb{R}^{d}$. While this problem is straightforward to pose, solving it can be extremely daunting when dealing with general \revised{non-convex} objective functions, as there are no clear directional or geometrical principles to guide the search for global minima. The difficulty of this task is further compounded when we attempt to address it using methods reliant solely on local information, such as gradient descent techniques. This is due to the abundance of local minima resulting from the \revised{non-convex} nature of the problem. To address these limitations of local higher-order methods, various approaches have been explored in computer science and engineering. In recent years, inspiration has been drawn from the concept of ``\revised{swarm} intelligence", which is rooted in physical or biological models and involves interacting particle or agent systems. Among these methods, gradient-free techniques, which rely solely on evaluations of the objective function, have proven to be sensible strategies in scenarios with low regularity or when the computation of gradients is prohibitively expensive. 

\paragraph{Consensus-Based Optimization and its analysis.}
A family of methods have recently been introduced under the terminology of Consensus-Based Optimization (CBO) methods~\cite{B1-pinnau2017consensus,carrillo2018analytical}, reviewed in~\cite{B1-totzeck2021trends}. These methods aim at fusing the cooling strategy of Simulated Annealing \cite{B1-Kirkpatrick1983,B1-SA86,B1-SA87} towards Gibbs equilibria with the space exploration by multiple particles/explorers as in Particle Swarm Optimization \cite{B1-kennedy1995particle} by taking advantage of a consensus mechanism as in Cucker--Smale \cite{B1-fornasier2010CS} and opinion formation models, in which an average orientation or opinion is obtained from the individual observations. 
The equations defining the iterates $V^i_k$ of the CBO algorithm read ($i=1,\dots,N$ labelling the particles, and $k=0,\dots,K$ denoting the iterates)
\begin{align}
\label{eq:B1-CBO1}
V^i_{k+1} &= V^i_{k} - \Delta t \lambda(V^i_k-v_{\alpha}(\widehat \rho_k^N)) + \sqrt{\Delta t}\sigma \|V^i_k-v_{\alpha}(\widehat  \rho_k^N)\|_2  B_k^i, 
\end{align}
where
\begin{align}
v_{\alpha}(\widehat \rho_k^N) = \frac{1 }{\sum_{i=1}^N \omega_\alpha^f(V_k^i)} \sum_{i=1}^N V_k^i \omega_\alpha^f(V_k^i), 
\end{align}
and $f$ is the objective function to be minimized, $v_{\alpha}(\widehat \rho_k^N)$ is the so-called consensus point, $\omega_\alpha^f(v)=e^{-\alpha {f}(v)}$ is  the Gibbs weight, and $B^i_k$ is an independent standard Gaussian random vector. The initial \revised{particles} $V_0^i$ are drawn i.i.d. at random according to a given probability distribution $\rho_0$.
The algorithm \eqref{eq:B1-CBO1} has a very simple formulation (it can be implemented in couple of lines of code) with complexity $\mathcal O(N)$ and is \revised{easily parallelizable, see, e.g., \cite{benfenati22,JMLR:v25:23-0764}.} Moreover, it is  based \revised{solely on} pointwise evaluations of the objective function $f$, hence, it does not require any higher order information. \\
The main mechanism of CBO is a combination of a relaxation drift towards the weighted average of the particles and their stochastic perturbation (exploitation-exploration mechanism). It is important to note that the original formulation of CBO, as described in equation \eqref{eq:B1-CBO1}, was designed to function under the restrictive assumption that the objective function $f$ has a {\it unique} minimizer.
In this seemingly restrictive setting, the first rigorous proof of {\it local} convergence to the unique global minimizer of the mean-field dynamics was achieved quite recently in \cite{carrillo2018analytical}. 
To remark the relevance of this contribution let us mention that, at that time, the latter work showed for the first time the convergence of an agent-based stochastic scheme for global optimization with mild assumptions on the regularity of the objective function.
The proof assumes (in the subsequent work \cite{B1-fornasier2021global} this assumption is rigorously justified, see below) that the dynamics for $\Delta t \to 0$ and $N \to \infty$ in \eqref{eq:B1-CBO1} can be described by its mean-field approximation
\begin{equation}\label{eq:B1-CBO2}
		d \overline{V}_t
		=-\lambda\left(\overline{V}_t-v_\alpha(\rho_t)\right)dt+\sigma \|\overline{V}_t-v_\alpha(\rho_t)\|_2 d B_t.
	\end{equation}
 where  $v_\alpha(\rho)=\frac{\int v \omega^{f}_\alpha(v) d\rho(v)}{\int \omega^f_\alpha(v) d\rho(v)}$ and $\rho_t= \operatorname{Law}(\overline{V}_t)$, which fulfills the nonlinear Fokker--Planck equation
 \begin{align} \label{eq:B1-fokker_planck}
	\partial_t\rho_t
	= \lambda\operatorname{div} \big(\!\left(v - v_\alpha(\rho_t)\right)\rho_t\big)
	+ \frac{\sigma^2}{2}\Delta\big(\|v-v_\alpha(\rho_t)\|_2^2\rho_t\big),
\end{align}
\revised{here $v\in\mathbb{R}^d$ is the variable of the equation.}
The large time behavior of \eqref{eq:B1-fokker_planck} that explains the convergence of \eqref{eq:B1-CBO1} is a particularly challenging problem, because \eqref{eq:B1-fokker_planck} does not fulfill a recognizable gradient flow structure and it is not prone to the usual entropy dissipation methods for large time asymptotics. By employing  an ingenious application of Laplace's principle, an {\it ad hoc}  variance analysis leveraging the PDE \eqref{eq:B1-fokker_planck}, and  provided well-preparedness of the initial datum $\rho_0$ to be concentrated in the \revised{neighbourhood} of $v^*$, the authors of \cite{carrillo2018analytical} ensure that $v_\alpha(\rho_t) \approx v^* =\ \arg\min_{v\in \mathbb R^d}f(v)$ and the dynamics ought to approach the global minimizer $v^*$ at large time; namely, they prove that $\operatorname{Var}(\rho_t) =\int \| v- E(\rho_t)\|_2^2 d\rho_t(v) \to 0$ and $E(\rho_t)= \int v d\rho_t(v) \approx v^*$ for $t\to \infty$.
However, in the work \cite{carrillo2018analytical} the authors did not investigate the convergence of the fully discrete finite particle dynamics \eqref{eq:B1-CBO1} to the mean-field approximation \eqref{eq:B1-CBO2} for $\Delta t \to 0$ and $N \to \infty$, hence missing to obtain a comprehensive result of convergence for \eqref{eq:B1-CBO1}.
The combination of standard results (see \cite{platen1999introduction})  of the Euler--Maruyama numerical approximation of the iterations \eqref{eq:B1-CBO1} for $\Delta t \to 0$ to solutions $V_t^i$, $i=1,\dots,N$ of the empirical system
\begin{equation}
d {V}^i_t
		=-\lambda\left(V^i_t-v_\alpha(\hat \rho_t^N )\right)dt+\sigma \|V^i_t-v_\alpha(\widehat \rho^N_t)\|_2 d B_t^i, \quad v_\alpha(\widehat \rho_t^N ) =  \frac{1 }{\sum_{i=1}^N \omega_\alpha^f(V_t^i)} \sum_{i=1}^N V_t^i \omega_\alpha^f(V_t^i),
\end{equation}
a quantitative version of Laplace's principle, which we report below in Lemma \ref{lem:qlp}, and a quantitative mean-field limit (see also \cite[Theorem 2.6]{gerber2023propagation})
\begin{equation}\label{eq:qmfl}
   \sup_{t \in [0,T]}\sup_{i=1,\dots,N} \mathbb E\left [ \| \overline V^i_t -V^i_t\|_2^2\right ]  \leq CN^{-1}, 
\end{equation}
\revised{where $\overline V^i_t$ are $N$ particles satisfying \eqref{eq:B1-CBO2} with $\overline V^i_0=V_0^i, i=1,\ldots,N$,}
allowed for a comprehensive  result of  {\it global} convergence to the {\it unique}  global minimizer in finite particle dynamics at discrete time, as obtained first in \cite{B1-fornasier2021global}. 
The result synthetically reads as follows: 
for any $M>0$, consider the set of bounded processes \begin{equation}\label{eq:bndp}
\Omega_M= \left \{\sup_{t\in[0,T]} \frac{1}{N}\sum_{i=1}^N \max\left\{\N{V_t^i}_2^4,\N{\overbar{V}_t^i}_2^4\right\} \leq M\right \},
\end{equation} whose probability $\mathbb P(\Omega_M) \geq 1- \frac{C}{M} = \delta$ (see \cite[Lemma 3.10]{B1-fornasier2021global}), for $C>0$ independent of $N$. After $K$ iterations of \eqref{eq:B1-CBO1} to achieve a suitable $T^* = K \Delta t$ (to be defined below, see \eqref{eq:finerr}), the expected error conditional to bounded processes is estimated by
\begin{equation}\label{B1-ownresult} \small
		\begin{aligned}
			\Ep{\normmsq{\frac{1}{N} \sum_{i=1}^N V_{K}^i-\globmin}\Bigg |\Omega_M}&\leq 3\left(\underbrace{\Ep{\normmsq{\frac{1}{N} \sum_{i=1}^N \left(V_{K}^i-V_{T^*}^i\right)}\Bigg |\Omega_M}}_{=: I}+{ \underbrace{\Ep{\normmsq{\frac{1}{N} \sum_{i=1}^N \left(V_{T^*}^i-\overline V_{T^*}^i\right)}}}_{=: II} } \right.\\
			&\left.+{\underbrace{\Ep{\normmsq{\frac{1}{N} \sum_{i=1}^N \overline V_{T^*}^i-\globmin}
   }}_{=: III} } \right)
   \leq C_{\mathrm{NA}}\Delta t+\frac{C_{\mathrm{MFA}}}{N}+\epsilon,
		\end{aligned}
	\end{equation}
	\revised{where $\epsilon$ is the error bound for the term $III$, $C_{\mathrm{NA}},C_{\mathrm{MFA}}$ are constants depending on $\lambda,\sigma,d,\alpha,f,T^*$.}
Then a simple application of Markow inequality and Bayes rule (see \cite[{Proof of Theorem 3.8}]{B1-fornasier2021global}) implies
\begin{equation}\label{B1-ownresult2}
  \mathbb P\Bigg( \normmsq{\frac{1}{N} \sum_{i=1}^N V_{K}^i-\globmin} \leq \varepsilon\Bigg) \geq 1 - \left [ \varepsilon^{-1} (C_{\mathrm{NA}}\Delta t+C_{\mathrm{MFA}} N^{-1}+\epsilon) -\delta \right ],  
\end{equation}
and, with this, one obtains convergence in probability of the CBO iterations \eqref{eq:B1-CBO1}\footnote{The necessity of conditioning on bounded processes is due to the fact that the convergence of the Euler--Maruyama scheme holds for bounded processes \cite{platen1999introduction}.}.
\revised{Here $\varepsilon>0$ is the accuracy of the final approximation and $\epsilon>0$ comes from \eqref{eq:finerr} below.}  \\ 
The error $I$ describes the discrepancy from the Euler--Maruyama approximation \eqref{eq:B1-CBO1} to the solution $V_t^i$ of the corresponding empirical SDE for $\Delta t \to 0$ \cite{platen1999introduction}. The error $II$ describes the quantitative mean-field approximation from \eqref{eq:qmfl}, and $III$ the large time behavior of the mean-field solution \eqref{eq:B1-CBO2}.
Notably, in contrast to \cite{carrillo2018analytical}, where a local variance analysis was employed, for error $III$ the authors of \cite{B1-fornasier2021global} conducted a study of the dynamics of the 2-Wasserstein distance  $W_2(\rho_t,\delta_{v^*})$ along the solution $\rho_t$.
Namely, the global convergence is resulting from the quantification of the approximation $v_\alpha(\rho_t) \approx v^*$ at all times for $\alpha>0$ large enough, and the consequent observation that the Wasserstein distance $W_2^2(\rho_t, \delta_{v^*})$ is a natural Lyapunov functional for \eqref{eq:B1-fokker_planck} at finite time, with exponential decay
\begin{equation}\label{eq:contr}
W_2^2(\rho_t, \delta_{v^*}) \leq W_2^2(\rho_0, \delta_{v^*}) e^{-(1-\theta)(2\lambda - d \sigma^2) t}, 
\end{equation}
for a suitable $\theta \in (0,1)$ for $t \in [0,T^*]$ and 
\begin{equation}\label{eq:finerr}
    W_2^2(\rho_{T^*}, \delta_{v^*}) \leq \epsilon.
\end{equation}
What makes this latter result  remarkable is that it unveiled a surprising phenomenon: CBO has the capability of convexifying a wide array of optimization problems as the number $N$ of optimizing agents approaches infinity, effectively transforming the original \revised{non-convex} optimization problem \eqref{eq:B1-optimization_problem} into the canonical - convex - task of minimizing the squared (transport) distance to the global minimizer \eqref{eq:contr} (see \cite[Figure 1]{B1-fornasier2021global} where expected trajectories of particles remarkably form straight rays to the global minimizer).
 Moreover, the result improved over prior analyses by placing minimal constraints on the method's initialization $\rho_0$ (in particular no concentration in the \revised{neighbourhood} of $v^*$ is needed) and by encompassing objectives that are only locally Lipschitz continuous. 
  While the analysis benefited from developing a quantitative mean-field limit with a Monte Carlo rate of convergence in the number $N$ of particles, it is worth noting that the constants in the estimate may exhibit exponential dependence on the dimension, which is to encompass also optimization problems \revised{intrinsically} affected by the curse of \revised{dimensionality}.  Other convergence results applied directly to the fully discrete microscopic system have been obtained in \cite{B1-ha2020convergenceHD,B1-ha2021convergence,doi:10.1142/S0218202522500245}, although they do not provide \revised{quantitative rates to global minimizers. A recent paper \cite{bellavia2024discreteconsensusbasedglobaloptimization} combines the techniques of continuous and discrete in time approaches, obtaining new results of global convergence.}\\
\paragraph{Relevance of Consensus-Based Optimization in applications.}
Besides these  theoretical results, CBO has been used successfully with state-of-the-art performances in several high-dimensional benchmark applications and we recall here a few of them: \cite{B1-fornasier2020consensus_sphere_convergence} applies CBO for a phase retrieval problem, robust subspace detection, and the robust computation of eigenfaces; \cite{B1-riedl2022leveraging} solves a compressed sensing task; \cite{B1-carrillo2019consensus,B1-fornasier2021anisotropic,B1-riedl2022leveraging} train shallow neural networks; \cite{B1-trillos2023FedCBO} devises FedCBO to solve clustered federated learning problems while ensuring maximal data privacy. The extensions below ensure that CBO can be applied also to an even more diverse spectrum of problems.

\paragraph{Extensions of Consensus-Based Optimization.} Variations of the CBO  methods have also been put forward to address global optimization problems with constraints. In cases of equality-constrained problems, it is handled by ensuring that the dynamics stay within the feasible manifold at all times, as done in \cite{B1-fornasier2020consensus_sphere_convergence,B1-fornasier2021anisotropic}. Additional related works on CBO for constrained optimizations are discussed in \cite{B1-kim2020stochastic, B1-ha2021emergent}, which tackle problems on the Stiefel manifold, and in the works of \cite{B1-carrillo2021consensus, B1-borghi2021constrained}, where constrained optimization is reformulated as a penalized problem \cite{B1-bostan2013asymptotic}. The philosophy of utilizing interacting swarms of particles to tackle various significant problems in science and engineering has given rise to further adaptations and extensions of the original CBO algorithm for optimization. These include methods  for addressing multi-objective optimization problems \cite{B1-borghi2022consensus, B1-borghi2022adaptive}, saddle point problems \cite{B1-qiu2022Saddlepoints}, 
or sampling from specific distributions \cite{B1-carrillo2022sampling}. In a similar vein, the original CBO method itself has undergone several modifications to enable more complex dynamics. These enhancements involve particles with memory \cite{B1-totzeck2020consensus, B1-grassi2021mean, B1-riedl2022leveraging}, the integration of momentum \cite{B1-chen2020consensus}, using anisotropic or jump noise processes \cite{B1-carrillo2019consensus,B1-fornasier2021anisotropic,B1-fornasier2022globalanisotropic,B1-kalise2022consensus}
the inclusion of gradient information \cite{B1-riedl2022leveraging}, and the utilization of on-the-fly extracted higher-order differential information through inferred gradients based on point evaluations of the objective function \cite{B1-schillings2022Ensemble}. 
Notably, it has been discovered that the well-known particle swarm optimization method (PSO) \cite{B1-kennedy1995particle} can be reformulated as CBO with momentum, and CBO can be seen as its vanishing inertia limit \cite{grassi2020particle, B1-cipriani2021zero}. This understanding has allowed for the adaptation of CBO's analytical techniques to rigorously prove the convergence of PSO \cite{B1-qiu2022PSOconvergence}. Some of the previous adaptations of CBO yet lack a comprehensive analysis of convergence and there is plenty of work to do to close some of the relevant gaps. 
\\

At this juncture, we have made the decision not to offer a further review of all the very recent developments in this field. Nevertheless, we believe that the description we gave so far stands as compelling evidence of the field's substantial growth and relevance.

\paragraph{Objectives with multiple minimizers and contribution of this work.} 
All the results, applications, and extensions reported above assume explicitly or implicitly that the objective function $f$ to be optimized possesses a unique minimizer. Indeed, when $f$ has multiple global minimizers, the original CBO
dynamic \eqref{eq:B1-CBO2} may have pathological behaviors as it is not designed to converge to multiple minimizers. The pathological cases may be rare, however they can occur, for example: let $d=1, f(v):=(v-1)^2 \zeta_{v\geq 0}+(v+1)^2 \zeta_{v\leq 0},\rho_0 =\mathcal{N}(0,1)$, where $\zeta$ is the characteristic function, then we know $1$ and $-1$ are the two global minimizers of $f$ and $v_{\alpha}(\revised{\rho_t})=0$ for any $t\geq 0$, this means if $\overline V_t$ starts from $0$, then it will always stay at $0$, however $0$ is not a global minimizer of $f$. Moreover, given the fact that in \eqref{eq:B1-CBO2} only one consensus point is allowed, if the particles converge to a \revised{minimizer,} then they necessarily converge to the same minimizer, missing to explore other minimizers that possibly exist and may be of interest.
In order to cope with optimizations with multiple minimizers recently a variation of CBO was proposed in \cite{bungert2022polarized} to allow for multiple consensus points, whose equation reads
\begin{equation}\label{eq:n1}
		dV^{\kappa}_t=-\lambda\left(V^{\kappa}_t-v_{\alpha}(\rho^{\kappa}_t,V^{\kappa}_t)\right)dt+\sigma\left(\normm{V^{\kappa}_t-v_{\alpha}(\rho^{\kappa}_t,V_t^\kappa)}+\kappa\right)dB_t,
	\end{equation}
	where $\rho_t^{\kappa}$ is the distribution of $V_t^{\kappa}$, $\kappa$ may be a small \revised{non-negative} constant~(in \cite{bungert2022polarized}, $\kappa=0$), and \begin{equation}\label{eq:n2}
		v_{\alpha}(\rho_t^{\kappa},V_t^{\kappa}):=\frac{\int w e^{-\alpha A_f(w,V^{\kappa}_t)}d\rho^\kappa_t\revised{(w)}}{\int e^{-\alpha A_f(w,V_t^{\kappa})}d\rho^{\kappa}_t\revised{(w)}},
	\end{equation} 
 whose law $\rho_t^{\kappa} =\operatorname{Law}(V^{\kappa}_t)$  fulfills the related nonlinear Fokker--Planck equation
	\begin{equation}\label{eq:n3}
		{\partial_t}\rho_t^{\kappa}=\lambda\operatorname{div}\left(\left(v-v_{\alpha}(\rho^{\kappa}_t,v)\right)\rho_t^{\kappa}\right)+\frac{\sigma^2}{2}\Delta\left(({\normm{v-v_{\alpha}(\rho^{\kappa}_t,v)}+\kappa})^2\rho_t^{\kappa}\right).
	\end{equation}
 Here $A_f(w,v):=f(w)+\frac{\normmsq{v-w}}{\Cone}$ and $\Cone>0$ is a problem dependent constant and we call this model the {\it polarized CBO}. 
 \revised{The CBO algorithm developed in \cite{bungert2022polarized} is the first in the literature to provide simultaneous computation of multiple minimizers for a broad class of functions.} 
 In \cite{bungert2022polarized}, under the assumption that \eqref{eq:n1} has a solution, the authors provide a proof of concentration for $\kappa=0$ of $\rho_t^{0}$  towards the set of global minimizers as $t \to \infty$ for the case where $f$ is convex and numerical evidences of convergence to multiple global minimizers in case of \revised{non-convex} $f$, yet without a rigorous proof for the latter case. The principle difficulty for obtaining a proof of global convergence for this model of $A_f(w,v)$ for \revised{non-convex} functions is readily provided: if $\Cone$ is fixed and not chosen properly, the polarized CBO dynamic can concentrate around points that are not close to the global minimizers of $f$. For example, in the \revised{one} dimensional case, set $\Cone=1$ and $f(w)=w^2$ for $w\leq 1$, $f(w)=\mathrm{min}\{1,(w-10)^2-1\}$ for $w\in [1,10]$ and $f(w)=(w-10)^2-1$ for $w\geq 10$. It is easy to verify that the only global minimizer of $f$ is $10$; for $v\in [-1,1]$, we have $\arg\min A_f(\cdot,v):=f(\cdot)+\normmsq{v-\cdot}=\frac{v}{2}\in [-1,1]$, so with the Laplace principle (see, e.g., Lemma \ref{lem:qlp} in Appendix) the polarized CBO dynamic can concentrate on points within the interval of $[-1,1]$ with non-negligible probability and these points are not close to $10$ which is the global minimizer of $f$. 
 { We show in Figure \ref{fig:my} a simple numerical experiment that confirms these analytic considerations.} \\
 In order to overcome this limitation of the polarized CBO and to allow for an adaptive scaling, in this work we propose the modified choice
  \begin{equation}\label{eq:extCBO}
  A_f(w,v):=A(w,v):=\frac{1}{\varkappa}f(w)(f(v)+\Cone)+\normmsq{v-w},
\end{equation}
  where \revised{$\varkappa, \Cone>0$} is a problem dependent constant. \revised{For  $A_f(\cdot,v)$ as in \eqref{eq:extCBO}, the coefficient $f(v)+\Cone$ now is a position dependent value and under proper assumptions on $f$, this adaptive scaling guarantees $\arg\min A_f(\cdot,v)$ to be close to some global minimizer of $f$. Thus, by the Laplace's principle, the dynamic \eqref{eq:n1} will drive particles to concentrate towards the set of global minimizers of $f$, check Figure~\ref{fig:my} for a visual comparison between the polarized CBO in \cite{bungert2022polarized} and our model.} Let us stress that the choice \eqref{eq:extCBO} is remarkably simple, yet effective to provide convergence to multiple minimizers in an adaptive way. Indeed, with this choice, we can provide theoretical guarantees that the dynamic~\eqref{eq:n1} has a solution and will  concentrate around the global minimizers of $f$ as $t\to\infty$ under mild assumptions on $f$, which now do allow $f$ to be \revised{non-convex}. The set of assumptions is collected in Section \ref{sec:ass} below. (Below we omit $f$ in $A_f(w,v)$ and we write $A(w,v)$ for ease of notation.) \\

{\begin{figure}
	\centering
	\title{Comparison between our dynamic and the polarized CBO dynamic}
	\includegraphics[width=1\linewidth]{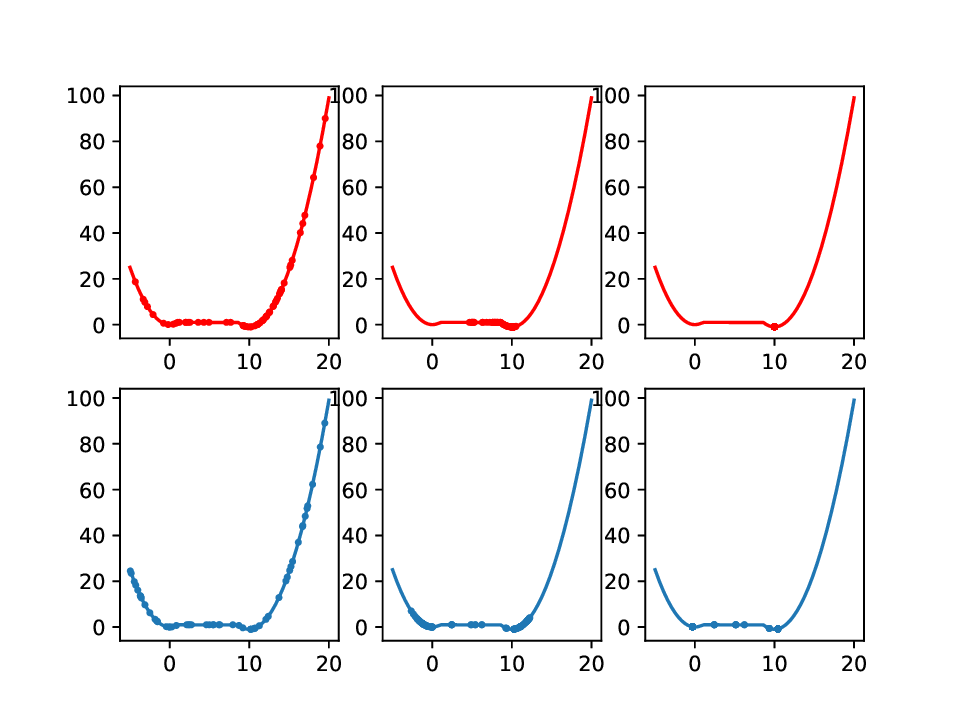}
	\caption{\revised{ We  compare the implementation \eqref{eq:alg} of our model in form of empirical realization (red color) with the one of polarized CBO in \cite{bungert2022polarized} for the minimization of $f(w)=w^2$ for $w\leq 1$, $f(w)=\mathrm{min}\{1,(w-10)^2-1\}$ for $w\in [1,10]$ and $f(w)=(w-10)^2-1$ for $w\geq 10$. 
 The function $f$ does satisfy the assumptions in this paper.
 For each dynamic, from left to right, we use three pictures to illustrate how particles evolve when the iteration number $k=10,20,40$. We can see from the tests that our dynamic eventually concentrates around the unique global minimizer of $f$, while some particles in the polarized CBO remain stuck at points in a non-minimal plateau. 
  We considered the function $A_f(w,v):=f(w)+\frac{\normmsq{v-w}}{\Cone}$ of the polarized CBO. We set $A(w,v)=20f(w)(f(v)+\Cone)+\norm{v-w}_2^2$ in our dynamic. In both dynamics, we set $\lambda=\Cone=1,\sigma=0.2,\alpha=10, N=100,\Delta t=0.3,\kappa=0$ and the initial particles are i.i.d. sampled from $\mathcal{N}(5,10)$.  We also tested for other parameters $\alpha=20,50,100,200$ for the polarized CBO dynamic, but the results do not differ significantly.}}
	\label{fig:my}
\end{figure}}

 \noindent The main result of this paper reads:
  
  \begin{theorem}\label{thm:main0} Let us fix $\kappa>0$ and $A(w,v)$ as in \eqref{eq:extCBO} with $\Cone\geq \sml\sup_{v,w\in\cV}\normm{v-w}^p-\inf_{v\in\mathbb{R}^d}f(v)$, where \revised{$\sml,p>0$ depending on $f$} and $\cV$ is the set of global minimizers of $f$. Assume $\cV$ fulfills $\cV=\cup_{i=1}^{\revised{\mm}} \cV_i$, where $\revised{\mm}\in \mathbb N$ and each $\cV_i$ is compact, convex, and has smooth boundary. Denote  
  \begin{equation}
        \cv_t:= \cv_t^\kappa:=\int {\dist{w,\cV}^2} d\rho_t^\kappa(w).
  \end{equation}
Under Assumption \ref{asp:11}, Assumption \ref{asp:22}, and Assumption \ref{asp:44}, for any $T^*>0$, by choosing \revised{$\varkappa$ small enough and} $\alpha>0$ large enough, the SDE~\eqref{eq:n1} has a strong solution till $T^*$ and its law $\rho^{\kappa}_t$ satisfies the related Fokker--Planck equation \eqref{eq:n3} in the weak sense. Moreover, for any error bound $\epsilon>0$, choose
		$T^*:=\frac{2}{\Cexp}\log(\frac{\cv_0\vee (2\epsilon)}{2\epsilon})$, where $\Cexp>0$ is from \eqref{ieq:144} (as a rule of thumb for $\Cexp>0$ it is sufficient that $\lambda>3 \sigma^2 d$). Then  we can choose $\kappa,\varkappa>0$ small enough and $\alpha>0$  large enough so that 
  \begin{equation}
\mathcal{V}_{t}\leq (\Cexp\epsilon t+\mathcal{V}_0)e^{-\Cexp t} \text{  for all } t \in [0,T^*],
  \end{equation}
  and
    \begin{equation}
  \mathcal{V}_{T^*}\leq2\epsilon.
  \end{equation}
  \end{theorem}
\noindent {A few remarks are in order:}
\\

{Theorem 1.1 establishes existence of solutions. Uniqueness remain\revised{s} an open issue, see Remark \ref{rem:uniq}. The result provides also the exponential concentration of the law $\rho_t^\kappa$ to the set of minimizers of $f$, under mild assumptions on the objective function.} As $\cv_0$ is not known, the horizon $T^*$ to achieve the prescribed error bound is not readily given, yet, as we observe in Remark \ref{rem:estT}, an estimate for $\cv_0$ and $T^*$ can be provided. \\

From a technical point of view, the proof of Theorem \ref{thm:main0} is not a minor incremental variation of the results in \cite{B1-fornasier2021global} or in \cite{bungert2022polarized}.
In fact, the analysis requires combining  delicate geometrical arguments, new  variations of the quantitative Laplace principle,  {\it ad hoc} regularizations and approximations, and  regularity theory for parabolic equations, all ingredients that were not used/needed in \cite{B1-fornasier2021global,bungert2022polarized}.
Moreover, in \cite{bungert2022polarized} the well-posedness of the model is not addressed.

\paragraph{Paper Structure.}
	The main results of this work are separated into two cases: 
	\begin{enumerate}
		\item $\cV$ is compact, convex and has smooth boundary;
		\item $\cV=\cup_{i=1}^{\revised{\mm}} \cV_i$, where each $\cV_i$ is compact, convex and has smooth boundary.
	\end{enumerate}
	The first case is certainly simpler and is included in the second case, however it encapsulates the key insights for the analysis, so once the readers understand this simpler case, it will not be as  hard for them to follow the more technical derivation of the second case. We will discuss the two cases separately in Section~\ref{sec:1} and Section~\ref{sec:2}. 
 In the last section, Section~\ref{sec:4}, we make the conclusion. The closing Appendix collects two versions of the quantitative Laplace principle, which are useful auxiliary results, and the concluding estimates for the convergence of the algorithm.
 
\paragraph{Notations.}
	Let us introduce the notations used in this paper: $\cV$ is the set of global minimizers of $f$ and we will always assume $\cV$ is \textit{compact} in this paper, $\underbar{f}:=\inf_{v\in\mathbb{R}^d}f(v)$,  ${V^*(w) \in \arg\min_{v\in \cV}\normm{v-w}}, \dist{w,\cV}:=\normm{w-V^*(w)}, \operatorname{diam}(\cV):=\sup_{v,w\in\cV}\normm{v-w}$, where $\normm{\cdot}$ denotes the Euclidean norm of a vector and $\cv_t^\kappa:=\int\normmsq{w-V^*(w)}d\rho_t^\kappa(w)$. We will denote \begin{equation}\label{eq:defA}
		A(w,v):=\frac{1}{\varkappa}f(w)(f(v)+\Cone)+\normmsq{v-w},\quad\text{for some constant $\Cone$}.
	\end{equation} We will define the test function $\phi_r^{\tau}$ as 
	\begin{equation}\label{eq:first}
		\phi_r^{\tau}(v):=\begin{cases}
			1+(\tau-1)\normm{\frac{v}{r}}^{\tau}-\tau\normm{\frac{v}{r}}^{\tau-1}& \normm{v}\leq r\\
			0& \revised{\mathrm {else}}
		\end{cases},\quad \tau\geq 3 \text{ and $\tau$ is an integer},
	\end{equation}
	it is easy to verify $\phi_r^{\tau}\in \mathcal{C}_c^1(\mathbb{R}^d,[0,1])$.

\section{Assumptions}\label{sec:ass}
In this paper, we assume $\rho_0\in\mathcal{C}_b^2(\mathbb{R}^d)$. Next, we introduce the assumptions on $f$ used in this work.
\revised{\begin{assumption}\label{asp:11} There exist some constants $\sml>0$ and $p\in [1,2]$, such that
    \begin{equation}
        f(w)-\underbar{f}\geq \sml\dist{w,\cV}^p.
    \end{equation}
\end{assumption}}
	
\begin{assumption}\label{asp:22}There exists some constant $L$, such that 
    \begin{itemize}
        \item For any $w\in\mathbb{R}^d$, we have 
        \begin{equation} \label{eq:gr}
            f(w)-\underbar{f}\leq L(1+\normmsq{w});
        \end{equation}
        \item For any $v,w\in\mathbb{R}^d$, we have
        \begin{equation}\label{eq:lip}
            |{f(v)-f(w)}|\leq L(\normm{w}+\normm{v} \revised{+1})\normm{v-w}.
        \end{equation}
    \end{itemize}
\end{assumption}
\begin{assumption}\label{asp:33}
    \revised{The} set $\cV,$ is compact, convex and has smooth boundary.
\end{assumption}
\begin{assumption}
		\label{asp:44}\revised{The set $\cV,$ is the union of convex, compact sets with smooth boundaries, that is,} $\cV=\cup_{i=1}^{\revised{\mm}} \cV_i$, and each $\cV_i$ is convex, compact set with smooth boundary.
	\end{assumption}
 {Let us comment \revised{on} the assumptions above. \revised{Assumption \ref{asp:11} requires an at least polynomial growth of degree $p \in [1,2]$ from the set of minimizers.} 
 Assumption \ref{asp:22} imposes an at most quadratic growth  and a controlled local Lipschitz continuity of the function $f$. 
Actually \eqref{eq:lip} implies \eqref{eq:gr}, but we keep them both for simplicity of presentation and the use of the same constant $L>0$.
In our analysis below, we consider first the case $f$ fulfilling Assumption \ref{asp:33} and then the one of $f$ satisfying Assumption \ref{asp:44}, which is more general and include\revised{s} the former as $\revised{\mm}=1$. {Let us mention that under assumption Assumption \ref{asp:33}  $V^*(v)$ is uniquely determined for all $v \in \mathbb R^d$, while under Assumption \ref{asp:44} $V^*(v)$ is uniquely determined for a.e. $v \in \mathbb R^d$.}
 \begin{remark}\label{rem:0in}
     In the derivation of the rest of the paper, we will assume {without loss of any generality} $0\in\cV$, since this can help avoid introducing more notations. 
 \end{remark}

	\section{Simple case analysis}\label{sec:1}
In this section, we  assume $\cV$ is convex, { more precisely we consider the case of $f$ fulfilling Assumption \ref{asp:33}}. We  set $\kappa=0$~(the analysis of $\kappa>0$ is similar) in \eqref{eq:n1} and \eqref{eq:n3}, and for simplicity, we  omit  $\kappa$ in the notations $V_t^\kappa,\rho_t^\kappa$ in this section, that is 
	\begin{equation}\label{eq:SDEsimple}
		dV_t=-\lambda(V_t-v_{\alpha}(\rho_t,V_t))dt+\sigma\normm{V_t-v_{\alpha}(\rho_t,V_t)}dB_t,
	\end{equation}
	where $\rho_t$ is the distribition of $V_t$ and \begin{equation}
		v_{\alpha}(\rho_t,V_t):=\frac{\int w e^{-\alpha A(w,V_t)}d\rho_t\revised{(w)}}{\int e^{-\alpha A(w,V_t)}d\rho_t\revised{(w)}},
	\end{equation}
	and its related Fokker--Planck equation is
	\begin{equation}\label{eq:FPsimple}
		{\partial_ t}\rho_t=\lambda\operatorname{div}\left(\left(v-v_{\alpha}(\rho_t,v)\right)\rho_t\right)+\frac{\sigma^2}{2}\Delta\left(\normmsq{v-v_{\alpha}(\rho_t,v)}\rho_t\right).
	\end{equation}
	 In this section, we will first show that \eqref{eq:SDEsimple} and \eqref{eq:FPsimple} have a solution in Theorem \ref{prop:22}, then the main convergence result of this section will be presented in Theorem \ref{prop:33}.
	
	\subsection{Existence of solutions in the simple case}
	
	The next lemma shows that $V^*(v)$ \revised{is close to the global minimizer} of the function $A(\cdot,v)$.
	\begin{lemma}\label{lem:111}
		Under Assumption \ref{asp:11} and Assumption \ref{asp:33} , let $\Cone\geq -\underbar{f}$, then we have
		\begin{equation}\label{eq:4}
			A(w,v)-A(V^*(v),v)\geq \normmsq{w-V^*(v)}\revised{-2(\frac{2\varkappa}{\sml^2})^{\frac{1}{p-1}}}, \quad \text{ for all } w,v\in\mathbb{R}^d.
		\end{equation}
	\end{lemma}
	\begin{proof}
		For simplicity, we  denote $a:=V^*(v)$.
		Let $\left\{e_i\right\}_{i=1}^d$ be \revised{an orthonormal basis} of $\mathbb{R}^d$, we will assume $v-V^*(v)=te_d$ for some $t\geq 0$ and denote $w=\sum_{i=1}^dw_ie_i$ for any $w\in\mathbb{R}^d$, so we know $\inner{v-V^*(v)}{w}=(v_d-V_d^*(v))w_d=tw_d$.
		When $w_d\geq 0$, we have
		\begin{equation}\label{ieq:29}
			\begin{aligned}
				&A(a+w,v)-A(a,v)\\
				&=\frac{1}{\varkappa}(f(a+w)-f(a))(f(v)+\Cone)+\normmsq{a+w-v}-\normmsq{a-v}\\
				&\geq \revised{\frac{1}{\varkappa}(\sml\dist{a+w,\cV}^p)(\sml\normm{v-a}^p+\Cone+\underbar{f})+\normmsq{w}-2|a_d-v_d|w_d}\\
				&\geq \revised{\normmsq{w}+\frac{\sml^2}{\varkappa}\normm{v-a}^pw_d^p-2|a_d-v_d|w_d}\\
				&\geq \revised{\normmsq{w}-2(\frac{2\varkappa}{\sml^2})^{\frac{1}{p-1}},}
			\end{aligned}
		\end{equation}
  in the above, the second inequality is due to $\dist{a+w,\cV}\geq w_d$ which can be derived by the convexity of $\cV$, the choice of $e_d$ and $\Cone+\underbar{f}\geq 0$; the last inequality is due to the following argument: \revised{when $\normm{a-v}w_d\leq ({2\varkappa}/{\sml^2})^{1/(p-1)}$, we have \begin{equation}
      \frac{\sml^2}{\varkappa}\normm{v-a}^pw_d^p-2\normm{a-v}w_d\geq -2\normm{a-v}w_d\geq -2(\frac{2\varkappa}{\sml^2})^{\frac{1}{p-1}};
  \end{equation} when $\normm{a-v}w_d> ({2\varkappa}/{\sml^2})^{1/(p-1)},$ 
		we have \begin{equation}
		    \frac{\sml^2}{\varkappa}\normm{v-a}^pw_d^p-2\normm{a-v}w_d = \normm{a-v}w_d \left (\frac{\sml^2}{\varkappa}\normm{v-a}^{p-1}w_d^{p-1}-2 \right ) > 0.
		\end{equation}}

		When $w_d<0$, since we already have $f(a+w)-f(a)\geq 0$, so we obtain
		\begin{equation}\label{eq:p3p3}
			\begin{aligned}
				A(a+w,v)-A(a,v)&\geq  \normmsq{a+w-v}-\normmsq{a-v}\\
				&=\normmsq{w}-2|a_d-v_d|w_d\\
				&\geq \normmsq{w}.
			\end{aligned}
		\end{equation}
		So all in all, for any $w\in \mathbb{R}^d$, we have
		\begin{equation}
			A(w,v)-A(a,v)\geq \normmsq{w-a}-2(\frac{2\varkappa}{\sml^2})^{\frac{1}{p-1}}.
		\end{equation}
	\end{proof}
By the above lemma, the global minimizer of $A(\cdot,v)$ is \revised{in $B_{r}(V^*(v))$ with $r=\sqrt{2(2\varkappa/\sml^2)^{1/(p-1)}}$, so to have $r$ small, we need to set $\varkappa$ small depending on $\sml$.}

 The next lemma shows the { Lipschitz continuity} of \revised{the} map $v_{\alpha}(\rho,\cdot)$.
	\begin{lemma}\label{lem:2}
		Under Assumption \ref{asp:22} and \revised{assuming that} $\int \normm{w}^4d\rho(w)\leq K$, we have 
		\begin{equation}
			\normm{v_{\alpha}(\rho,v)-v_{\alpha}(\rho,w)}\leq \Cfour(\alpha,R,K,L,\varkappa)\normm{v-w},
		\end{equation}
		for any $v,w\in B_R(0)$, where $\Cfour(\alpha,R,K,L,\varkappa)$ means a constant \revised{that} only depends on $\alpha,R,K,L,\varkappa$.
	\end{lemma}
	\begin{proof}
		Without loss of generality, we assume $\underbar{f}=0$~(since replacing $f(u)$ by $f(u)-\underbar{f}$ in $A(u,v)$ won't change $v_{\alpha}(\rho,v)$) and $A(u,v):=\frac{1}{\varkappa}f(u)(f(v)+\Cone)+\normmsq{v-u}$ \revised{to be} non-negative. We have 
		\begin{equation}\label{eq:1414}
			\begin{aligned}
				&\normm{v_{\alpha}(\rho,v)-v_{\alpha}(\rho,w)}\\
				&=\normm{\frac{\int u e^{-\alpha A(u,v)}d\rho(u)}{\int e^{-\alpha A(u,v)}d\rho(u)}-\frac{\int u e^{-\alpha A(u,w)}d\rho(u)}{\int e^{-\alpha A(u,w)}d\rho(u)}}\\
				&\leq \frac{1}{{\int e^{-\alpha A(u,v)}d\rho(u)}}\normm{{\int u e^{-\alpha A(u,v)}d\rho(u)}-{\int u e^{-\alpha A(u,w)}d\rho(u)}}\\
				&\quad +\frac{1}{{\int e^{-\alpha A(u,v)}d\rho(u)}{\int e^{-\alpha A(u,w)}d\rho(u)}}\\
				&\revised{\quad\times\normm{{\int u e^{-\alpha A(u,w)}d\rho(u)}\left[{\int e^{-\alpha A(u,v)}d\rho(u)}-{\int e^{-\alpha A(u,w)}d\rho(u)}\right]}}.
			\end{aligned}
		\end{equation}
		Since 
		\begin{equation}\label{eq:1515}
			\frac{1}{{\int e^{-\alpha A(u,v)}d\rho(u)}}\leq e^{\alpha\int A(u,v)d\rho(u)}\leq \Ctwo(\alpha,R,K,\revised{L},\varkappa)
		\end{equation}
		and 
		\begin{equation}\label{eq:1616}
			\begin{aligned}
				|{e^{-\alpha A(u,v)}-e^{-\alpha A(u,w)}}|&\leq \max\left\{e^{-\alpha A(u,v)},e^{-\alpha A(u,\revised{w})}\right\}\alpha |{A(u,v)-A(u,w)}|\\
				&\leq \Cthree(\alpha,R,L,\varkappa)(1+\normm{u}^{2})\normm{v-w},
			\end{aligned}
		\end{equation}
		in \eqref{eq:1616}, the first inequality is due to $|e^x-e^y|\leq \max\{e^x,e^y\}\revised{|x-y|},$
		\revised{ the second one is due to \begin{equation}
		\begin{aligned}
		|A(u,v)-A(u,w)|&= \frac{1}{\varkappa}f(u)\left(f(v)-f(w)\right)+\inner{2u-v-w}{w-v}\\
		&\leq \frac{L^2}{\varkappa}(1+\normmsq{u})(\normm{v}+\normm{w}+1)\normm{v-w}+\normm{2u-v-w}\normm{w-v}\\
		&\leq \Cthree(\alpha,R,L,\varkappa)(1+\normmsq{u})\normm{v-w}.
		\end{aligned}
		\end{equation}}
  
So, combining \eqref{eq:1414}, \eqref{eq:1515}, \eqref{eq:1616}, and the assumption that $\int\normm{w}^4d\rho(w)\leq K$, we finally have
		\begin{equation}
			\normm{v_{\alpha}(\rho,v)-v_{\alpha}(\rho,w)}\leq \Cfour(\alpha,R,K,L,\varkappa)\normm{v-w}.
		\end{equation}
		
	\end{proof}

	To prove the existence of a strong solution of \revised{the} SDE \eqref{eq:SDEsimple}, {we are going to use an approximation procedure based on suitable projections onto compact sets. To do this, we first define the projection map onto a centered ball of radius $R>0$} as
	\begin{equation}
		\cP\left(v\right):=\begin{cases}
			v, & \text{if $\normm{v}\leq R,$}\\
			R\frac{v}{\normm{v}}, & \text{if $\normm{v}>R.$}
		\end{cases}
	\end{equation}
	With this projection map, we can prove the following proposition.
	\begin{proposition}\label{prop:11}Under Assumption \ref{asp:22}, 
		the following SDE
		\begin{equation}\label{eq:1919}
			dV^R_t=-\lambda(V^R_t-v_{\alpha}(\rho^R_t,\cP(V^R_t))dt+\sigma\normmsq{V^R_t-v_{\alpha}(\rho^R_t,\cP(V^R_t)}dB_t,\quad R\in (0,\infty),
 		\end{equation}
		where $\rho_t^R$ is the distribution of $V_t^R$ and $\rho_0^R=\rho_0\in\mathcal{P}_4(\mathbb{R}^d)$, has a unique strong solution and $\rho_t^R$ satisfies the following Fokker--Planck equation in the weak sense,
		\begin{equation}\label{eq:projFP}
			{\partial_ t}\rho_t^R=\lambda\operatorname{div}\left(\left(v-v_{\alpha}(\rho_t^R,\cP(v))\right)\rho_t^R\right)+\frac{\sigma^2}{2}\Delta\left(\normm{v-v_{\alpha}(\rho_t^R,\cP(v))}\rho_t^R\right).
		\end{equation}
	\end{proposition}
	
	\begin{proof}
		The proof is an adaptation from the proof of \cite[Theorem 3.1]{carrillo2018analytical} and{ it is based on the application of Leray--Schauder fixed point theorem}; we only provide a sketch here. We first define the mapping 
		\begin{equation}\label{eq:2121}
			\cT_R:\cC([0,T^*]\times\overline{B_R(0)},\mathbb{R}^d)\to \cC([0,T^*]\times\overline{B_R(0)},\mathbb{R}^d);\quad \revised{u\mapsto v_{\alpha}(\rho_{\cdot}^{R,u},\cdot)},~v\in B_R(0),
		\end{equation}
		where $\rho_t^{R,u}$ is the distribution of $V_t^{R,u}$, which satisfies the following SDE
		\begin{equation}
			dV^{R,u}_t=-\lambda(V^R_t-u(t,\cP(V_t^{R,u})))dt+\sigma\normm{V^R_t-u(t,\cP(V_t^{R,u}))}dB_t.
		\end{equation}
		{ Below we intend to show that  $\cT_R$ has a fixed point, hence the existence of solutions for \eqref{eq:1919} and \eqref{eq:projFP}.
		For that, we \revised{first show} that $\cT_R$ is a compact mapping.} 
		By standard arguments~(see \cite[Chapter 7]{arnold1974stochastic}),  for any $p\geq 1$, we have
		\begin{equation}
			\Ep{\normm{V_t^{R,u}}^p}
			\leq \Cfive(1+\Ep{\normm{V_0^{R,u}}^p}),
		\end{equation}
		for some constant $\Cfive>0$ \revised{depending} on $p,t,\lambda,\sigma,d,\norm{u}_{\infty}$. By It\^o isometry,  for any $t,s\in [0,T^*]$, we have
		\begin{equation}
			\Ep{\normmsq{V_t^{R,u}-V_s^{R,u}}}\leq \Csix\normm{t-s},
		\end{equation}
		here $\Csix$ depends on $\norm{u}_{\infty},\lambda,\sigma,T^*$, then apply \cite[Lemma 3.2]{carrillo2018analytical}, we have
		\begin{equation}\label{eq:h29}
			\normm{v_{\alpha}(\rho_s^{R,u},\cP(v))-v_{\alpha}(\rho_t^{R,u},\cP(v))}\leq \Cseven W_2(\rho_t^{R,u},\rho_s^{R,u})\leq \Cseven\Csix\normm{t-s}^{\frac{1}{2}},
		\end{equation}
		for some {constant $\Cseven>0$ that depends} on $R,\alpha, L,\varkappa, \Ep{\normm{V_t^{R,u}}^4}$ and $\Ep{\normm{V_s^{R,u}}^4}$, hence we { obtain} the H\"older continuity of $t\to v_{\alpha}(\rho_t^{R,u},\cP(v))$ with exponent $\frac{1}{2}$. Due to Lemma \ref{lem:2},  $v\mapsto v_{\alpha}(\rho_t^{R,u},\cP(v))$ is Lipschitz continuous~(and thus is also H\"older continuous with exponent $\frac{1}{2}$), and so the compactness of $\cT_R$ can be provided by the compact embedding  $\mathcal{C}^{0,1 / 2}\left([0, T^*]\times\overline{B_R(0)}, \mathbb{R}^d\right) \hookrightarrow$ $\mathcal{C}\left([0, T^*]\times\overline{B_R(0)}, \mathbb{R}^d\right).$
		
		Then, we  show that the set $\Gamma:=\left\{u\in \mathcal{C}([0,T^*]\times \overline{B_R(0)},\mathbb{R}^d): \revised{u=r \cT_Ru,r \in [0,1]}\right\}$ is bounded, that is, for each $u$ in this set, we have $\norm{u}_{\infty}\leq C$, for some uniform constant $0<C<\infty$; { then by the Leray--Schauder fixed point theorem, we can conclude the existence part of this proof}. Towards this,  when $v\in B_R(0)$, \revised{the} function $A(\cdot,v):=\frac{1}{\varkappa}f(\cdot)(f(v)+\Cone)+\normmsq{\cdot-v}$ satisfies 
		\begin{itemize}
			\item $A(u,v)-A(V^*(v),v)\geq \normmsq{u-V^*(v)}\geq \frac{1}{10}\normmsq{u}, \quad \text{for}\quad \normm{u}\geq 100\operatorname{diam}(\cV),$
			\item $ A(u,v)-A(V^*(v),v)\leq \frac{1}{\varkappa}[f(u)-f(V^*(u))][f(v)+\Cone]+\normmsq{v-u}-\normmsq{V^*(v)-v} \leq \Ceight(R,L,\varkappa)(1+\normmsq{u}),~ \text{since}~\normm{v}\leq R.$
		\end{itemize}
  In the above, the first condition is due to $0\in\cV$ and Lemma \ref{lem:111}, the second condition is due to Assumption \ref{asp:22}.
		Then by \cite[Lemma 3.3]{carrillo2018analytical}, for the function $A(\cdot,v),v\in B_R(0)$, we have
		\begin{equation}\label{eq:2626}
			\normmsq{v_{\alpha}(\rho_t^R,\cP(v))}\leq \Cnine\Ep{\normmsq{V_t^R}}+\Cten,
		\end{equation}
		with $\Cnine= 20\Ceight(R,L,\varkappa)(1+\frac{10}{\alpha\operatorname{diam}(\cV)^2}),\Cten=10^4\operatorname{diam}(\cV)^2+\Cnine$. Following the third step in the proof of \cite[Theorem 3.2]{carrillo2018analytical},  the boundedness of \revised{the} set $\Gamma$ is proved.
		
		Lastly, the uniqueness is a standard procedure, the reader can refer to the fourth step in the proof of \cite[Theorem 3.1]{carrillo2018analytical}.
	\end{proof}
	{ As we have proved that the SDE \eqref{eq:1919} has a solution by the last proposition, in  what follows we can focus on the   analysis of its large time behavior, by analyzing the asymptotics of $\rho_t^R$.}
	\begin{lemma}\label{lem:66}Under  Assumption \ref{asp:22} and Assumption \ref{asp:33}, we have 
		\begin{equation}\label{eq:2727}
			\begin{aligned}
				\frac{d}{dt}\int\normmsq{v-V^*(v)}d\rho^R_t(v)
				&\leq-(\lambda-2\sigma^2d)\int\normmsq{v-V^*(v)}d\rho^R_t(v)\\
				&\quad +(\lambda+2\sigma^2d)\int\normmsq{V^*(v)-v_{\alpha}(\rho_t^R,\cP(v))}d\rho^R_t(v).
			\end{aligned}
		\end{equation}
	\end{lemma}
	\begin{proof}
		For ease of notation, we will omit $R$ and $\cP$ in the proof, also  in the proof $v, V^*(v)$ are row vectors.
		Apply It\^o formula to the function $\normmsq{v-V^*(v)}$, we have
		\begin{equation}
			\begin{aligned}
				\frac{d}{dt}\int\normmsq{v-V^*(v)}d\rho_t(v)
				&={-\lambda\int \inner{\nabla \normmsq{v-V^*(v)}}{v-v_{\alpha}(\rho_t,v)}d\rho_t(v)}\\
				&\quad+{\frac{\sigma^2}{2}\int\Delta\normmsq{v-V^*(v)}\normmsq{v-v_{\alpha}(\rho_t,v)}d\rho_t(v)}\\
                & =-2\lambda\int (v-v_{\alpha}(\rho_t,v))(I-\nabla V^*(v))(v-V^*(v))^{\top}d\rho_t(v)\\ &\quad+{\frac{\sigma^2}{2}\int\Delta\normmsq{v-V^*(v)}\normmsq{v-v_{\alpha}(\rho_t,v)}d\rho_t(v)}\\
				&=-2\lambda\underbrace{\int (v-V^*(v))(I-\nabla V^*(v))(v-V^*(v))^{\top}d\rho_t(v)}_{I}\\ &\quad+\frac{\sigma^2}{2}\underbrace{\int\Delta\normmsq{v-V^*(v)}\normmsq{v-v_{\alpha}(\rho_t,v)}d\rho_t(v)}_{II}\\
				&\quad-2\lambda\int (V^*(v)-v_{\alpha}(\rho_t,v))(I-\nabla V^*(v)(v-V^*(v))d\rho_t(v).\\				
			\end{aligned}
		\end{equation}
		
		For term $I$, we have it equal to  
		\begin{equation}
			\begin{aligned}
				&-\int (v-V^*(v))(I-\nabla V^*(v))(v-V^*(v))^{\top}d\rho_t(v)\\
				&=-\int \normmsq{v-V^*(v)}d\rho_t(v),
			\end{aligned}
		\end{equation}
		{ the equality can be ensured as follows}:  for $t\in (-1,0)$, we always have
		\begin{equation}
			V^*(v-t(v-V^*(v)))=V^*(v),
		\end{equation}
		and so 
		\begin{equation}\label{eq:17}
			\frac{d}{dt}V^*(v-t(v-V^*(v)))\mid_{t=0}=-(v-V^*(v))\nabla V^*(v)=0.
		\end{equation}
		
		For term $II$, we need to calculate $\Delta\normmsq{v-V^*(v)}$: since $v$ and $V^*(v)$ are row vectors, $\nabla V^*(v)=\left(\nabla V_1^*(v),\nabla V_2^*(v),\ldots,\nabla V_d^*(v)\right)$ is { a matrix  of dimensions $d \times d$ and each $\nabla V_i^*(v)$ has dimension $d \times 1$}. Denote $\left\{e_i\right\}_{i=1}^d$ as \revised{a} basis of $\mathbb{R}^d$. If $\nabla V^*(v)$ is symmetric { (and we  prove it a few lines below)}, then we have $(v-V^*(v))\nabla V^*(v)=\nabla V^*(v)(v-V^*(v))^{\top}=0$ for any $v\in (\cV)^{c}$ by \eqref{eq:17}, and so
		\begin{equation}
			\begin{aligned}
				\Delta\normmsq{v-V^*(v)}&=\operatorname{tr}\left(\nabla^2 \normmsq{v-V^*(v)}\right)\\
				&=2\operatorname{tr}\left(\nabla \left[(I-\nabla V^*(v))(v-V^*(v))^{\top}\right]\right)\\
				&=2\operatorname{tr}\left(\nabla \left[v-V^*(v)\right]\right)\\
				&=2\operatorname{tr}\left(I-\nabla V^*(v)\right)\\
				&\leq 2d
			\end{aligned}
		\end{equation}
		since $\cV$ is convex and has smooth boundary and so $I-\nabla V^*(v)\preccurlyeq I$, see Remark~\ref{rmk:10}. \revised{In the following, we will show} $\nabla V^*(v)$ is symmetric~({we only need to verify the symmetry of its matrix with respect to any orthogonal basis}). { Let us assume that $k \leq d$ is the dimension of $\cV$, and $k-1$ is the dimension of its smooth boundary $\partial \cV$;} then, due to the convexity of $\cV$, there must be a $k$ dimensional linear subspace, denoted as $K$, such that $\cV\subset K$, since in Remark \ref{rem:0in} we have assumed $0\in \cV$. Without loss of generality, we will denote the basis of $K$ as $\{e_i\}_{i=1}^k$, its perpendicular basis denoted as $\{e_i\}_{i=k+1}^d$ and $V^*(v)=\sum_{i=1}^dV^*_i(v)e_i=\sum_{i=1}^kV^*_i(v)e_i$. For any direction $e\bot K$, we know $V^*(v+e)=V^*(v)=V^*(\mathrm{Proj}_{K}(v))$, where $\mathrm{Proj}_{K}(v)$ denotes the projection of $v$ onto $K$,  and thus $\partial_i V^*_j(v)=0$ for any $i\in [k+1,\ldots,d],j\in [1,\ldots,d]$. Under this basis,  we know $V^*_j(v)=0$ for any $v\in\mathbb{R}^d,j\in [k+1,\ldots,d]$ and thus $\partial_iV^*_j(v)=0$ for any 
  $i\in [1,\ldots,d],j\in[k+1,\ldots,d]$. 
  
 \revised{In the following, we need to verify} $\left(\nabla V^*(v)\right)_{k\times k}$ is symmetric for any $v\in K$. Assume $v_0\in K$ and $v_0\not\in \cV$~
  (if $v_0\in \cV$, then $V^*(v_0)=v_0$ and $\left(\nabla V^*(v)\right)_{k\times k}=I_{k\times k}$), and without loss of generality we can re-choose the basis $\{e_i\}_{i=1}^k$ of $K$ such that $v_0-V^*(v_0)=ae_k$ with $a>0$. From the convexity of $\cV$, we know
		$V_k^*(v_0+\epsilon e_i)\leq V_k^*(v_0)$, for $\epsilon$ small and $i=1,\ldots,k-1$, so we have
		$\partial_i V^*_k(v_0)=0$, $i=1,\ldots,k-1$. Since $V^*(v_0+te_k)=V^*(v_0)$ for $t\geq 0$, we have $\partial_k V^*(v_0)=0$. 
  
  Again, we need to verify $(\nabla V^*(v_0))_{(k-1)\times(k-1)}$ is symmetric. For simplicity, denote $b_0=V^*(v_0)$ and $\left(b_s\right)_{s\geq 0}$ as a smooth curve on $\partial\cV$ from $b_0$ and parameterized by time $s$, denote $a_s$ as the unit outward normal vector of $\partial \cV$ at point $b_s$ and then we have $v_0=b_0+ta_0$ for some $t\geq 0$ and
		\begin{equation}
			\begin{aligned}
				\lim_{s\to 0}\frac{b_s-b_0}{s}&=\lim_{s\to 0}\frac{V^*(b_s+ta_s)-V^*(b_0+ta_0)}{s}\quad //\text{since $b_s=V^*(b_s+ta_s)$}\\
				&=\lim_{s\to 0}\frac{V^*(b_0+ta_0+(b_s-b_0+t(a_s-a_0)))-V^*(b_0+ta_0)}{s}\\
                    &=(\nabla V^*(b_0+ta_0))_{(k-1)\times(k-1)}(I_{(k-1)\times(k-1)}+t \nabla N(b_0))\lim_{s\to 0}\frac{b_s-b_0}{s}\\
				&=(\nabla V^*(v_0))_{(k-1)\times(k-1)}(I_{(k-1)\times(k-1)}+t \nabla N(b_0))\lim_{s\to 0}\frac{b_s-b_0}{s},
			\end{aligned}
		\end{equation}
		where $N$ is the Gauss map~(that is $N(b_0)$ is the \revised{unit outward normal vector to} the surface $\partial \cV$ at point $b_0$). Since $\lim_{s\to 0}\frac{b_s-b_0}{s}$ can be any direction in the space spanned by $\left\{e_1,\ldots,e_{k-1}\right\}$, we thus have
		\begin{equation}
			(\nabla V^*(v_0))_{(k-1)\times(k-1)}=(I_{(k-1)\times(k-1)}+t \nabla N(b_0))^{-1}.
		\end{equation}
		As $\nabla N(v_0)$ is symmetric~({see, e.g., 
  \cite{do2016differential}, where in it, $d=3$; however, the proof of the self-adjointness of the gradient of the Gauss map can be applied to higher dimensions}), we have \revised{that} $(\nabla V^*(v_0))_{(k-1)\times(k-1)}$ is also symmetric. 
  All in all, we have shown $\nabla V^*(v)$ is symmetric for any $v\in\mathbb{R}^d$.
		
		
		At last, by combining term $I$ and term $II$, we have 
		\begin{equation}
			\begin{aligned}
				\frac{d}{dt}\int\normmsq{v-V^*(v)}d\rho_t(v)
				&=-2\lambda\int (v-V^*(v))(I-\nabla V^*(v))(v-V^*(v))^{\top}d\rho_t(v) \\
				&\quad+\sigma^2\int \operatorname{tr}(I-\nabla V^*(v))\normmsq{v-v_{\alpha}(\rho_t,v)}d\rho_t(v)\\
				&\quad-2\lambda\int (V^*(v)-v_{\alpha}(\rho_t,v))(I-\nabla V^*(v)(v-V^*(v))d\rho_t(v),
			\end{aligned}
		\end{equation}
	then we use the Cauchy-Schwartz inequality in the last term to get
	\begin{equation}
	\begin{aligned}
			&\frac{d}{dt}\int\normmsq{v-V^*(v)}d\rho_t(v)\\
			&\leq -2\lambda\int\normmsq{v-V^*(v)}d\rho_t(v)+2\sigma^2d\int\normmsq{v-V^*(v)}d\rho_t(v)\\
	&\quad+2\sigma^2d\int\normmsq{V^*(v)-v_{\alpha}(\rho_t,v)}d\rho_t(v)\\
	&\quad+2\lambda\sqrt{\int\normmsq{v-V^*(v)}d\rho_t(v)}\sqrt{\int \normmsq{V^*(v)-v_{\alpha}(\rho_t,v)}d\rho_t(v)}\\
	&\leq-(\lambda-2\sigma^2d)\int\normmsq{v-V^*(v)}d\rho_t(v)+(\lambda+2\sigma^2d)\int\normmsq{V^*(v)-v_{\alpha}(\rho_t,v)}d\rho_t(v).
	\end{aligned}
	\end{equation}
	\end{proof}
	\begin{remark}\label{rmk:10}
		Since $\cV$ is convex and has smooth boundary, we have $0\preccurlyeq\nabla V^*\preccurlyeq I$: remember in the proof, under suitable coordinates, { we have that the eigenvalues of $(\nabla V^*(v))_{(k-1)\times(k-1)}=\left(I+\normm{v-V^*(v)}\nabla N(V^*(v))\right)^{-1}$ belong to $[0,1]$, if $v\in (\cV)^c$}; 
   if $v\in\cV$, we have $v=V^*(v)$ and so $\left(\nabla V^*(v)\right)_{k\times k}=I_{k\times k}$.
	\end{remark}

	\revised{Next}, we will upper bound the second term in the right hand side of \eqref{eq:2727}. We will show that for any $T^*>0$, by choosing small enough $\varkappa$ and large enough $\alpha$, we have 
	\begin{equation}
		\normmsq{V^*(v)-v_{\alpha}(\rho_t^R,\cP(v))}\leq \Cctr\normmsq{v-V^*(v)}+\Csmall,\quad \forall t\in [0,T^*],
	\end{equation}
	where $\Cctr,\Csmall$ are constants independent of $R$.

	\begin{lemma}\label{lem:44}	
		Under Assumption \ref{asp:11}, Assumption \ref{asp:22} and Assumption \ref{asp:33}, for any $T^*>0$, by \revised{choosing $\varkappa$ small enough depending on $\lambda,\sigma,d,\sml$ and $\alpha$ large enough depending on $\rho_0,\lambda,\sigma,d,L,\varkappa,T^*,$\\
  $\operatorname{diam}(\cV)$, but independent of $R$,} we can have \begin{equation}\label{eq:4040}
			\normmsq{V^*(v)-v_{\alpha}(\rho_t^R,\cP(v))}\leq \Cctr\normmsq{v-V^*(v)}+\Csmall,\quad\forall t\in [0,T^*]
		\end{equation}
		where $\Cctr,\Csmall$ are constants independent of $R$.
	\end{lemma}
 
	\begin{proof}
		Under Assumption \ref{asp:22}, by a direct calculation, it is easy to show that 
		\begin{equation}\label{eq:404040}
			\sup_{u\in B_r(V^*(v))}|A(u,v)-A(V^*(v),v)|\leq \Celeven(1+\normmsq{v-V^*(v)})r,\text{where $r\in(0,1),$}
		\end{equation}
		for some constant $\Celeven=\Celeven(L,\varkappa,\operatorname{diam}(\cV))$. By Lemma \ref{lem:111} and the quantitative Laplace principle Lemma \ref{lem:qlp}, we have 
\begin{equation}\label{iieq:175}
    \begin{aligned}
        \normm{v_{\alpha}(\rho_t^R,\cP(v))-V^*(v)}\leq (q+A_r(v))^{{1}/{2}}+\frac{e^{-\alpha(q-2(\frac{2\kappa}{\sml^2})^{1/(p-1)})}}{\rho_t^R(B_r(V^*(v)))}\int\normm{w-V^*(v)}d\rho_t^R(w),
    \end{aligned}
\end{equation}
for any $r>0$ and $q>2(2\varkappa/\sml^2)^{1/(p-1)},$ here $A_r:=\sup_{u\in B_r(V^*(v))}A(u,v)-A(V^*(v),v)$. Insert \eqref{eq:404040} into \eqref{iieq:175}, we have
\begin{equation}\label{iieq:177}
    \begin{aligned}
         \normmsq{v_{\alpha}(\rho_t^R,\cP(v))-V^*(v)}&\leq 2q+4\Celeven r(1+\normmsq{v-V^*(v)})\\
         &\quad+2\left(\frac{e^{-\alpha(q-2(\frac{2\kappa}{\sml^2})^{1/(p-1)})}}{\rho_t^R(B_r(V^*(v)))}\right)^2\int\normmsq{w-V^*(v)}d\rho_t^R(w),
    \end{aligned}
 \end{equation}
  which is \begin{equation}\label{eq:4141}
			\begin{aligned}
				&\normmsq{v_{\alpha}(\rho_t^R,\cP(v))-V^*(v)}\\&\leq \Ctwelve\left[r\normmsq{v-V^*(v)}+r+(E^R_{t,r,p,q,\sml}(\alpha,\varkappa))^2\int \normmsq{v-V^*(v)}d\rho_t^R(v)\right]+2q,
			\end{aligned}
		\end{equation}
		where $\Ctwelve$ only depends on $L,\operatorname{diam}(\cV),\varkappa$, and $E^R_{t,r,p,q,\sml}(\alpha,\varkappa):=\frac{e^{-\alpha(q-2({2\varkappa}/{\sml^2})^{1/(p-1)})}}{\inf_{w\in\cV}\int_{B_r(w)}\phi_r^{\tau}(u-w)\rho_t^R(u)}$.
		By Lemma \ref{lem:66} and the above analysis, we have 
			\begin{equation}
			\begin{aligned}
				&\frac{d}{dt}\int\normmsq{v-V^*(v)}d\rho_t^R(v)\\
				&\leq-(\lambda-2\sigma^2d)\int\normmsq{v-V^*(v)}d\rho^R_t(v)+(\lambda+2\sigma^2d)\int\normmsq{V^*(v)-v_{\alpha}(\rho_t^R,\cP(v))}d\rho_t^R(v)\\
				&\leq-[\lambda-2\sigma^2d-\Ctwelve r(\lambda+2\sigma^2d)]\int\normmsq{v-V^*(v)}d\rho_t^R(v)\\
				&\quad+(\lambda+2\sigma^2d)\Ctwelve\left[r+(E^R_{t,r,p,q,\sml}(\alpha,\varkappa))^2\int \normmsq{v-V^*(v)}d\rho_t^R(v)\right]+2(\lambda+2\sigma^2d)q\\
				&\leq -\Cexp\int\normmsq{v-V^*(v)}d\rho^R_t(v)+\Cerr\left[(E^R_{t,r,p,q,\sml}(\alpha,\varkappa))^2\int \normmsq{v-V^*(v)}d\rho^R_t(v)+r\right]\\
    &\quad+2(\lambda+2\sigma^2d)q,
			\end{aligned}
		\end{equation}
		where $\Cexp:=\lambda-2\sigma^2d-\Ctwelve r(\lambda+2\sigma^2d)$ depends on $\lambda,\sigma,d,L,\varkappa,\operatorname{diam}(\cV)$~(\revised{we need to choose $r,\lambda,\sigma$ to make sure $\Cexp>0$; for example, let $\lambda>0$ and set $\sigma,r$ such that $\lambda=4\sigma^2d, \Ctwelve r(\lambda+2\sigma^2d)=1/4\lambda$, then $\Cexp=1/4\lambda>0$}) and $\Cerr:=(\lambda+2\sigma^2d)\Ctwelve$ depends on $\lambda,\sigma,d,L,\varkappa,\operatorname{diam}(\cV)$. For simplicity, we will denote $ \cv_t^R:=\int\normmsq{v-V^*(v)}d\rho_t^R(v)$ and
  \begin{equation}
        H^R_{t,r,p,q,\sml}(\alpha,\varkappa,\mathcal{V}_t^R):=\Cerr\left[(E^R_{t,r,p,q,\sml}(\alpha,\varkappa))^2\cv_t^R+r\right]+2(\lambda+2\sigma^2d)q,
\end{equation} then \revised{with this} new notation, we have 
\begin{equation}
			\frac{d}{dt}\mathcal{V}_t^R\leq -\Cexp\mathcal{V}_t^R+H^R_{t,r,p,q,\sml}(\alpha,\varkappa,\mathcal{V}_t^R).
		\end{equation}
	
	For any $T^*>0$, set $\cc = 10T^*+\mathcal{V}_0^R$, define $T^{o}:=\sup\left\{t\geq 0,\mathcal{V}^R_{t'}\leq \cc,\forall t'\in [0,t]\right\}$; since in \eqref{eq:2626} we have proved
	\begin{equation}
		\normmsq{v_{\alpha}(\rho_t^R,\cP(v))}\leq\Cnine\mathcal{V}^R_t+\Cten,
	\end{equation}
	with $\Cnine=20\Ceight(\operatorname{diam}(\cV)+r,L,\varkappa)(1+\frac{10}{\alpha\operatorname{diam}(\cV)^2}), \Cten=10^4\operatorname{diam}(\cV)^2+\Cnine$, for any $v\in \cV+B_r(0)$, so in Lemma \ref{lem:5l}, the constant
 \begin{equation}
     B:=\sup_{v\in\cV}\sup_{t\in[0, {T^*\wedge T^o}],w\in B_r(v)}\normm{v_{\alpha}(\rho_t,w)-v}\leq \Cthirteen(\Cnine,\Cten,\revised{\cc})
 \end{equation}is uniformly upper bounded    
 and so, by Lemma \ref{lem:5l}, we have $\inf_{w\in\cV}\int_{B_r(w)}\phi_r^{\tau}(u-w)\rho_t^R(u)\geq \Cfourteen>0,$ where $\Cfourteen$ depends on $\rho_0,r,\lambda,\sigma,d,T^*,\operatorname{diam}(\cV),L,\varkappa$. This means we can choose $\varkappa,q~(>2(2\varkappa/\sml^2)^{1/(p-1)})$ small enough~(for example we choose $q=3(2\varkappa/\sml^2)^{1/(p-1)}$) depending on $\lambda,\sigma,d,\sml$; $r$ small enough which depends on $\lambda,\sigma,d,L,\operatorname{diam}(\cV),\varkappa$; $\alpha$ large enough which depends on $\rho_0,r,q,\varkappa,\sigma,\lambda,d,T^*,\operatorname{diam}(\cV),L$,  such that
	\begin{equation}
		H^R_{t,r,p,q,\sml}(\alpha,\varkappa,\mathcal{V}_t^R)\leq 10,
	\end{equation}
	for any $t\in [0,T^*\wedge T^o]$. 
	
	Next, we  show $T^o\geq T^*$, then on $[0,T^*]$, we have 
	$H^R_{t,r,p,q,\sml}(\alpha,\varkappa,\mathcal{V}_t^R)\leq 10$  and so 
	\begin{equation}
		\normmsq{v_{\alpha}(\rho_t^R,\cP(v))-V^*(v)}\leq \Cctr\normmsq{v-V^*(v)}+\Csmall.
	\end{equation}
	We prove this by contradiction: if $T^o<T^*$, we have for any $t\in [0,T^o]$
	\begin{equation}
		\frac{d}{dt}	\mathcal{V}_t^R\leq 10,
	\end{equation}
	and so
	\begin{equation}
		\mathcal{V}^R_t-\mathcal{V}^R_0\leq 10 t,
	\end{equation}
	then
	\begin{equation}\label{eq:4949}
		\cc-\mathcal{V}_0^R= \mathcal{V}_{T^o}^R-\mathcal{V}_0^R\leq 10 T^o.
	\end{equation}
	However, by definition $	\cc=10T^*+\mathcal{V}_0^R$, then with \eqref{eq:4949}, we have 
	\begin{equation}
		\cc=10T^*+\mathcal{V}^R_0\leq 10T^o+\mathcal{V}_0^R,
	\end{equation}
	which means $T^*\leq T^o$ and so this is a contradiction.
	
	All in all, we have shown that for any $T^*$, we can choose $\varkappa,q,r$ small enough and $\alpha$ large enough in \eqref{eq:4141} such that 
	
	\begin{equation}
		\normmsq{v_{\alpha}(\rho_t^R,\cP(v))-V^*(v)}\leq \Cctr\normmsq{v-V^*(v)}+\Csmall,
	\end{equation}
	where $\Cctr,\Csmall$ are independent of $R$, \revised{and also independent of $\lambda,\sigma,d,L,\varkappa,\operatorname{diam}(\cV),T^*,\sml$.}
	\end{proof}

With Lemma \ref{lem:44}, we will prove \begin{equation} 
	\Ep{\revised{|\bold{V}^R_t|^4}}\leq \Cd,
\end{equation}
where here and in the following $\bold{V}^R_t:=\sup_{s\in[0,t]}\normm{V^R_s}$, $\Cd$ is a constant.}
	
	\begin{lemma}\label{lem:3} 	Under Assumption \ref{asp:11}, Assumption \ref{asp:22} and Assumption \ref{asp:33}, for any $T^*>0$, we can choose \revised{$\varkappa$ small enough depending on $\lambda,\sigma,d,\sml$ and $\alpha$ large enough depending on $\lambda,\sigma,d,L,\varkappa,\operatorname{diam}(\cV),T^*$, but} independent of $R$, such that  
		\begin{equation}
			\Ep{\revised{|\bold{V}^R_t|^4}}\leq \Cd,\quad t\in [0,T^*],
		\end{equation}
		where $\Cd$ only depends on $\lambda,\sigma,d,\rho_0,T^*$. 
	\end{lemma}
	\begin{proof}
		By Lemma \ref{lem:44}, we have
		\begin{equation}\label{eq:90}
			\normmsq{v_{\alpha}(\rho^R_t,\cP(v))-V^*(v)}\leq \Cctr\normmsq{v-V^*(v)}+\Csmall.
		\end{equation}

  \revised{For any $h\in [0,t]$, by using the integral formulation of the stochastic process $V_s^R$ and standard inequalities, we have
  \begin{equation}
      \begin{aligned}
          {\normm{V_{h}^R}^4}&\leq 27\left\{{\normm{V_0^R}^4}+\lambda^4{\normm{\int_0^{h}V_s^R-v_{\alpha}(\rho_s^R,\cP(V_s^R))ds}^4}\right.\\
          &\quad\left.+\sigma^4{\normm{\int_0^{h}\normm{V_s^R-v_{\alpha}(\rho_s^R,\cP(V_s^R))}dB_s}^4}\right\}\\
          &\leq 27\left\{{\normm{V_0^R}^4}+\lambda^4{\sup_{h'\in [0,t]}\normm{\int_0^{h'}V_s^R-v_{\alpha}(\rho_s^R,\cP(V_s^R))ds}^4}\right.\\
          &\quad\left.+\sigma^4{\sup_{h'\in[0,t]}\normm{\int_0^{h'}\normm{V_s^R-v_{\alpha}(\rho_s^R,\cP(V_s^R))}dB_s}^4}\right\},
      \end{aligned}
  \end{equation}
  then
  \begin{equation}
      \begin{aligned}
           \sup_{h\in[0,t]}\normm{V_{h}^R}^4&\leq 27\left\{
           {\normm{V_0^R}^4}+\lambda^4{\sup_{h\in [0,t]}\normm{\int_0^hV_s^R-v_{\alpha}(\rho_s^R,\cP(V_s^R))ds}^4}\right.\\
          &\quad\left.+\sigma^4{\sup_{h\in[0,t]}\normm{\int_0^{h}\normm{V_s^R-v_{\alpha}(\rho_s^R,\cP(V_s^R))}dB_s}^4}\right\},
      \end{aligned}
  \end{equation}
  where we have changed the dumb index $h'$ into $h$ in the right hand side,
  take expectation from both sides, we have
  }		
  \begin{equation}
			\begin{aligned}
				\Ep{\revised{|\bold{V}^R_t|^4}}&\revised{\leq} 27\left\{\Ep{\normm{{V_0^R}}^4}+\lambda^4\Ep{\sup_{\revised{h}\in[0,t]}\normm{\int_{0}^{\revised{h}} V_s^R-v_{\alpha}(\rho_s^R,\cP(V_s^R))ds}^4}\right.\\
				&\left.\quad+\sigma^4\Ep{\sup_{\revised{h}\in[0,t]}\normm{\int_{0}^{\revised{h}} \normm{V_s^R-v_{\alpha}(\rho_s^R,\cP(V_s^R))}dB_s}^4}\right\}.
			\end{aligned}
		\end{equation}
		For the last term, by applying the Burkholder--Davis--Gundy inequality, we have
		\begin{equation}
			\Ep{\sup_{\revised{h}\in[0,t]}\normm{\int_{0}^{\revised{h}} \normm{V_s^R-v_{\alpha}(\rho_s^R,\cP(V_s^R))}dB_s}^4}\revised{\leq} \revised{C_{\mathrm{d}}} \Ep{{\left |\int_{0}^t \normmsq{V_s^R-v_{\alpha}(\rho_s^R,\cP(V_s^R))}ds \right |^2}},
		\end{equation}
  \revised{where here, we denote $C_{\mathrm{d}}$ as a constant that may only depend on $d$.} 
		Since $\normm{V^*(v)}$ is bounded ({in view of the compactness of $\cV$}), so using \eqref{eq:90} and \revised{the triangle} inequality, we have
		\begin{equation}
			\Ep{\revised{|\bold{V}^R_t|^4}}\revised{\leq} 
			\Cfifteen\left[1+\int_0^t\revised{\Ep{|\bold{V}^R_s|^4}} ds\right],
		\end{equation}
		where $\Cfifteen$ depends on $\lambda,\sigma,d,\rho_0$. Finally by \revised{Gr\"onwall's inequality}, we have
		\begin{equation}
			\Ep{\revised{|\bold{V}^R_t|^4}}\leq \Cd,
		\end{equation}
		where $\Cd$ only depends on $\lambda,\sigma,d,\rho_0,T^*$.
	\end{proof}
	The next proposition shows the SDE~\eqref{eq:SDEsimple} has a strong solution, this is proved by a limiting procedure.
	\begin{theorem}\label{prop:22}Under Assumption \ref{asp:11}, Assumption \ref{asp:22} and Assumption \ref{asp:33}, for any $T^*>0$, by choosing $\varkappa>0$ small enough and $\alpha>0$ { large} enough, then \revised{the} SDE \eqref{eq:SDEsimple} has a strong solution, { whose law $\rho_t$} satisfies the related Fokker--Planck equation \eqref{eq:FPsimple} in the weak sense.
	\end{theorem}
 \begin{remark}\label{rem:uniq}
{     In the limiting process in the proof below, when providing existence of solutions, we loose compactness. For this reason,  uniqueness of solution for \eqref{eq:SDEsimple} and its related Fokker--Planck equation remains an open issue, which does not seem to follow from standard stability arguments, even \revised{when} using some localization and approximation. The source of the problem is the estimate \eqref{eq:2626} where the constant $\Cnine$ depends on $\|v\|_2^2$.} 
 \end{remark}
	\begin{proof}
		Combine Lemma \ref{lem:2} and Lemma \ref{lem:3}, we have $\{v_{\alpha}(\rho_t^{\Bar{R}},\cdot)\}_{\Bar{R}>R,t\in[0,T^*]}$ is uniformly Lipschitz in the space variable for $v,w\in B_R(0)$ and uniformly bounded by inequality \eqref{eq:2626} and Lemma \ref{lem:3}. Thus for any fixed $t\in [0,T^*]$, {by \revised{the} Ascoli--Arzelà theorem} we can choose a subsequence $\{v_{\alpha}(\rho_t^{R_i},v)\}_{i\in \mathbb{N}}$ by a diagonal selection procedure, such that $R_i\geq 2^i$ and the limiting function denoted as  $v_{t,\infty}(\cdot)$~(of course it has the same Lipschitz constant as $v_{\alpha}(\rho_t^{\Bar{R}},\cdot)$ inside $B_R(0)$ for $\Bar{R}\geq R$) satisfies
		\begin{equation}\label{eq:5959}
			\normm{v_{\alpha}(\rho_t^{R_k},v)-v_{t,\infty}(v)}\leq \frac{1}{\exp(100T^*(\lambda^2T^*+\sigma^2d)(1+\Cfour(\alpha,2^i,\Cd,L,\varkappa)))2^i}, 
		\end{equation}
		for $\forall v\in B_{2^i}(0) ,k\geq i,$ where $\Cfour(\alpha,2^i,\Cd,L,\varkappa)$ comes from Lemma \ref{lem:2} and $\Cd$ comes from Lemma \ref{lem:3}~({The convenience of the cumbersome bound on the right-hand side  \eqref{eq:5959} will be clear below, see \revised{the} derivation of \eqref{eq:6767}}).
		
		By \cite[Corollary 6.3.1]{arnold1974stochastic}, the following SDE has a strong solution
		\begin{equation}
			dV^{\infty}_t=-\lambda(V^{\infty}_t-v_{t,\infty}(V^{\infty}_t))dt+\sigma\normm{V^{\infty}_t-v_{t,\infty}(V^{\infty}_t)}dB_t,
		\end{equation}
		and so its related Fokker--Planck equation has a weak solution $\rho_t^\infty$ for test functions $\phi\in C_c^{\infty}([0,T]\times \mathbb{R}^d)$.
		Next, we will show that 
		\begin{equation}
			{ v_\alpha(\rho_\cdot^\infty, \cdot)= \cT_{\infty}v_{\cdot,\infty}}=v_{\cdot,\infty},
		\end{equation}
		where \revised{the} map $\cT_{\infty}$ comes from  \eqref{eq:2121} by { formally} letting $R=\infty$,
		then it will be natural that \revised{the} SDE~\eqref{eq:SDEsimple} and its related Fokker--Planck equation {have respective}	solutions.
		
		Define $Z_t^i:=V^{\infty}_t-V_t^{R_i}$, then it satisfies
		\begin{equation}
			\begin{aligned}
				Z_t^i&=Z_0^i-\lambda \int_0^tZ_s^ids-\lambda\int_0^t(v_{s,\infty}(V^{\infty}_s)-v_{\alpha}(\rho_s^{R_i},\mathcal{P}_{R_i}(V_s^{R_i})))ds\\
				&\quad+\sigma\int_0^t\left[\normm{V^{\infty}_s-v_{s,\infty}(V^{\infty}_s)}-\normm{V_s^{R_i}-v_{\alpha}(\rho_s^{R_i},\mathcal{P}_{R_i}(V_s^{R_i}))}\right]dB_s,
			\end{aligned}
		\end{equation}
		define 
		\begin{equation}\label{eq:6666}
			I_M(t)=\begin{cases}
				1& \text{if $\max\{\normm{V_{\tau}},\normm{V_{\tau}^{R_i}}\}\leq M$ for all $\tau\in [0,t]$,}\\
				0&\text{otherwise,}
			\end{cases}
		\end{equation}
		then it is adapted to the natural filtration and has the property $I_M (t) = I_M (t)I_M (\tau )$ for all $\tau\in[0, t]$.
		With \revised{Jensen's inequality}, $\rho_0^{\infty}=\rho_0^{R_i}=\rho_0$, and It\^o isometry, we can derive 
		\begin{equation}
			\begin{aligned}
				\Ep{\normmsq{Z_t^i}I_{2^i}(t)}&\leq 100(\lambda^2t+\sigma^2d)\left[\int_0^t\Ep{\normmsq{Z_s^i}I_{2^i}(s)}ds\right.\\
				&\quad\left.+\int_0^t\Ep{\normmsq{v_{\alpha}(\rho_s^{R_i},\mathcal{P}_{R_i}(V_s^{R_i}))-v_{s,\infty}(V^{\infty}_s)}I_{2^i}(s)}ds\right]\\
			\end{aligned}
		\end{equation}
	 and
	\begin{equation}
		\begin{aligned}
			& \left\|v_\alpha\left(\rho_s^{R_i}, \mathcal{P}_{R_i}\left(V_s^{R_i}\right)\right)-v_{s, \infty}\left(V_s\right)\right\|_{\revised{2}} I_{2^i}(s) \\
			& =\left\|v_\alpha\left(\rho_s^{R_i}, V_s^{R_i}\right)-v_{s, \infty}\left(V_s^{\infty}\right)\right\|_{\revised{2}} I_{2^i}(s) \\
			& \leq\left\|v_\alpha\left(\rho_s^{R_i}, V_s^{R_i}\right)-v_{s, \infty}\left(V_s^{R_i}\right)\right\|_{\revised{2}} I_{2^i}(s)+\left\|v_{s, \infty}\left(V_s^{R_i}\right)-v_{s, \infty}\left(V_s^{\infty}\right)\right\|_{\revised{2}} I_{2^i}(s) \\
			& \leq \frac{1}{\exp(100T^*(\lambda^2T^*+\sigma^2d)(1+\Cfour(\alpha,2^i,\Cd,L,\varkappa)))2^i}+\Cfour(\alpha,2^i,\Cd,L,\varkappa)\left\|Z_s^i\right\|_{\revised{2}} I_{2^i}(s),
		\end{aligned}
	\end{equation}
	the last inequality is due to \eqref{eq:5959} and Lemma \ref{lem:2}. And so, by combining the last two inequalities, we have
	\begin{equation}
		\begin{aligned}
			\Ep{\normmsq{Z_t^i}I_{2^i}(t)}&\leq 100(\lambda^2T^*+\sigma^2d)(1+C(\alpha,2^i,D,L,\varkappa))\int_0^t\Ep{\normmsq{Z_s^i}I_{2^i}(s)}ds\\
			&\quad+200(\lambda^2T^*+\sigma^2d)T^*\frac{1}{\exp(100T^*(\lambda^2T^*+\sigma^2d)(1+\Cfour(\alpha,2^i,\Cd,L,\varkappa)))2^i},
		\end{aligned}
	\end{equation}
	hence, by \revised{Gr\"onwall's inequality}, we have
	\begin{equation}\label{eq:6767}
		\Ep{\normmsq{Z_t^i}I_{2^i}(t)}\leq \frac{\Csixteen(\lambda,\sigma,d,T^*)}{2^i}.
	\end{equation}
	By combining the above result and Lemma \ref{lem:3}, we can show $\lim_{i\to\infty}W_2^2(\rho_t^{R_i},\rho_t^{\infty})=0$, see Lemma \ref{lem:444} for the details. Finally, for any $v\in \mathbb{R}^d$, there is some $j$ such that $v\in B_{2^j}(0)$, and we have
	\begin{equation}\label{wasscntr}
		\begin{aligned}
			&\lim_{i\to\infty}\normm{\mathcal{T}_{R_{i}}v_{\alpha}(\rho_t^{{R_{i}}},v)-
				\mathcal{T}_{\infty}v_{t,\infty}(v)}\\
			&=\lim_{i\to\infty}\normm{v_{\alpha}
				(\rho_t^{R_{i}},v)-v_{\alpha}(\rho^{\infty}_t,v)}\\
			&\leq\lim_{i\to\infty} \Cseventeen(\alpha,\Cd,L,\varkappa)W_2(\rho_t^{R_{i}},\rho^{\infty}_t)=0,
		\end{aligned}
	\end{equation}
the first equality is due to \revised{the facts that} $v_{\alpha}(\rho_t^{{R_{i}}},v)$ is invariant under \revised{the} map $\cT_{R_i}$ and $v_{\alpha}(\rho^{\infty}_t,v):=	\mathcal{T}_{\infty}v_{t,\infty}(v)$, the last inequality is due to \cite[Lemma 3.2]{carrillo2018analytical}. So
\begin{equation}
	\mathcal{T}_{\infty}v_{t,\infty}(v)=\lim_{i\to\infty}\mathcal{T}_{R_{i}}v_{\alpha}(\rho_t^{R_{i}},v)=\lim_{i\to\infty}v_{\alpha}(\rho_t^{R_{i}},v)=v_{t,\infty}(v),
\end{equation}
and thus we finished the proof.
\end{proof}


	\begin{remark}\label{rmk:22}
		For the process $\bold{V}_t^{\infty}$, where $\bold{V}_t^{\infty}:=\sup_{s\in[0,t]}\normm{V^{\infty}_s}$,  
		we also have $\Ep{|\bold{V}_t^{\infty}|^4}\leq D$, since 
		\begin{equation}
			\normm{v_{t,\infty}(v)-v_{\alpha}(\rho_t^{R_i},v)}\leq \frac{1}{\exp(100T^*(\lambda^2T^*+\sigma^2d)(1+\Cfour(\alpha,2^i,\Cd,L,\varkappa)))2^i},
		\end{equation}
		and so
		\begin{equation}
			\normm{v_{t,\infty}(v)-V^*(v)}\leq \Cctr\normm{v-V^*(v)}+\Csmall,
		\end{equation}
		for some \revised{constants,} for convenience we will also denote them as $\Cctr,\Csmall$, and so the uniform bound can be derived by a similar analysis as in the proof of Lemma \ref{lem:3}.


	\end{remark}
	
	\begin{lemma}\label{lem:444}
		We have 
		\begin{equation}
			\lim_{i\to\infty}W_2^2(\rho_t^{R_i},\rho_t^{\infty})=0.
		\end{equation}
	\end{lemma}
	\begin{proof} We have
		\begin{equation}
		\begin{aligned} W_2^2(\rho_t^{R_i},\rho_t^{\infty})&\leq\Ep{\normmsq{Z_t^i}I_{2^i}(t)}+\Ep{\normmsq{V_t^{R_i}-V_t^{\infty}}{\zeta}_{\text{$\bold{V}_t^{R_i}\leq 2^i$ and $\bold{V}_t^{\infty}> 2^i$}}}\\
		&\quad+\Ep{\normmsq{V_t^{R_i}-V_t^{\infty}}\zeta_{\text{$\bold{V}_t^{R_i}> 2^i$ and $\bold{V}_t^{\infty}\leq 2^i$}}}\\
		&\quad+\Ep{\normmsq{V_t^{R_i}-V_t^{\infty}}\zeta_{\text{$\bold{V}_t^{R_i}> 2^i$ and $\bold{V}_t^{\infty}> 2^i$}}},
		\end{aligned}
		\end{equation}
		where $I_{2^i}(t)$ is from \eqref{eq:6666}, and {$\zeta$ is the characteristic function of the underlying set}.
		For the first term we have it converging to $0$ as $i\to\infty$ by \eqref{eq:6767}. For the second term, we have
		\begin{equation}
		\begin{aligned}
		\Ep{\normmsq{V_t^{R_i}-V_t^{\infty}}\zeta_{\text{$\bold{V}_t^{R_i}\leq 2^i$ and $\bold{V}_t^{\infty}> 2^i$}}}&\leq 2\Ep{(\normmsq{V_t^{\infty}}+4^i)\zeta_{\bold{V}_t^{\infty}>2^i}}\\
		&\leq 2\Ep{(\normmsq{\bold{V}_t^{\infty}}+4^i)\zeta_{\bold{V}_t^{\infty}>2^i}}\\
		&\leq 4\Ep{\normmsq{\bold{V}_t^{\infty}}\zeta_{\bold{V}_t^{\infty}>2^i}}\\
		&\leq \frac{\Cd}{4^{i-1}}\to 0\quad \text{as $i\to\infty$},
		\end{aligned}
		\end{equation}
		since we have $\Ep{\normm{\bold{V}_t^{\infty}}^4}\leq \Cd$ by Remark \ref{rmk:22} and so
		\begin{equation}
		\begin{aligned}
		\Ep{\normmsq{\bold{V}_t^{\infty}}\zeta_{\bold{V}_t^{\infty}>2^i}}&=\frac{\Ep{\normmsq{\bold{V}_t^{\infty}}4^i\zeta_{\bold{V}_t^{\infty}>2^i}}}{4^i}\\
		&\leq \frac{\Ep{\normm{\bold{V}_t^{\infty}}^4\zeta_{\bold{V}_t^{\infty}>2^i}}}{4^i}\\
		&\leq \frac{\Ep{\normm{\bold{V}_t^{\infty}}^4}}{4^i}\\
		&\leq\frac{\Cd}{4^i}.
		\end{aligned}
		\end{equation}
		Similarly, we have the third and fourth terms converge to $0$ as $i\to\infty$.
	\end{proof}
	\subsection{Large time asymptotics}\label{sec:LTA1}
 \revised{In this section, we will show the SDE \eqref{eq:SDEsimple} will concentrate around the set of global minimizers as $t\to\infty$, which is formally stated in Theorem~\ref{prop:33}. To prove it, we first show an auxilariy differential inequality, namely~\eqref{eq:8181}.}
 
	As in the proof of Lemma \ref{lem:44}, under Assumption \ref{asp:22}, we have
	\begin{equation}
		\sup_{u\in B_r(V^*(v))}\normm{A(u,v)-A(V^*(v),v)}\leq \Celeven(1+\normmsq{v})r, \text{ where $r\in (0,1)$},
	\end{equation}
 and 
	\begin{equation}\label{eq:7979}
		\begin{aligned}
			\normmsq{v_{\alpha}(\rho_t,v)-V^*(v)}&\leq \Ctwelve\left[r\normmsq{v-V^*(v)}+r+E^2_{t,r,p,q,\sml}(\alpha,\varkappa)\int \normmsq{v-V^*(v)}d\rho_t(v)\right]+2q,
		\end{aligned}
	\end{equation}
	where the function $E_{t,r,p,q,\sml}(\alpha,\varkappa)$ is given by 
 \begin{equation}
     E_{t,r,p,q,\sml}(\alpha,\varkappa):=\frac{e^{-\alpha (q-2(2\varkappa/\sml^2)^{1/(p-1)})}}{\inf_{w\in\cV}\int_{B_r(w)}\phi_r^{\tau}(u-w)\rho_t(u)}.
 \end{equation}
	\revised{In the last section, we have proved the SDE~\eqref{eq:SDEsimple} has a strong solution, so we can apply It\^o's formula to the function $\normmsq{v-V^*(v)}$ and the process $V_t$,} then using the proof of Lemma \ref{lem:66}, we have 
\begin{equation}\label{eq:95}
		\begin{aligned}
			&\frac{d}{dt}\int\normmsq{v-V^*(v)}d\rho_t(v)\\
			&\leq-(\lambda-2\sigma^2d)\int\normmsq{v-V^*(v)}d\rho_t(v)+(\lambda+2\sigma^2d)\int\normmsq{V^*(v)-v_{\alpha}(\rho_t,v)}d\rho_t(v).
  	\end{aligned}
	\end{equation} 
Using \eqref{eq:7979}, we further have
	\begin{equation}\label{eq:8080}
		\begin{aligned}
			&\frac{d}{dt}\int\normmsq{v-V^*(v)}d\rho_t(v)\\
			&\leq-(\lambda-2\sigma^2d)\int\normmsq{v-V^*(v)}d\rho_t(v)+(\lambda+2\sigma^2d)\int\normmsq{V^*(v)-v_{\alpha}(\rho_t,v)}d\rho_t(v)\\
			&\leq-[\lambda-2\sigma^2d-\Ctwelve r(\lambda+2\sigma^2d)]\int\normmsq{v-V^*(v)}d\rho_t(v)\\
			&\quad+(\lambda+2\sigma^2d)\Ctwelve\left[r+E^2_{t,r,p,q,\sml}(\alpha,\varkappa)\int \normmsq{v-V^*(v)}d\rho_t(v)\right]+2(\lambda+\sigma^2d)q\\
			&\leq -\Cexp\int\normmsq{v-V^*(v)}d\rho_t(v)+\Cerr\left[E^2_{t,r,p,q,\sml}(\alpha,\varkappa)\int \normmsq{v-V^*(v)}d\rho_t(v)+r\right]\\
   &\quad+2(\lambda+\sigma^2d)q,
		\end{aligned}
	\end{equation}
	where $\Cexp:=\lambda-2\sigma^2d-\Ctwelve r(\lambda+2\sigma^2d)$ depends on $\lambda,\sigma,d,L,\varkappa,\operatorname{diam}(\cV)$~(\revised{as before, we need to choose $\lambda,\sigma, r$ such $\Cexp>0$}) and $\Cerr:=(\lambda+2\sigma^2d)\Ctwelve$ depends on $\lambda,\sigma,d,L,\varkappa,\operatorname{diam}(\cV)$. For simplicity, we denote $H_{t,r,p,q,\sml}(\alpha,\varkappa,\mathcal{V}_t):=\Cerr\left[E^2_{t,r,p,q,\sml}(\alpha,\varkappa)\cv_t+r\right]+2(\lambda+\sigma^2d)q,$\\
 $ \cv_t:=\int\normmsq{v-V^*(v)}d\rho_t(v)$, then under these new notation, we have 
	\begin{equation}\label{eq:8181}
		\frac{d}{dt}\mathcal{V}_t\leq -\Cexp\mathcal{V}_t+H_{t,r,p,q,\sml}(\alpha,\varkappa,\mathcal{V}_t).
	\end{equation}
	
	The following result can be adapted from \cite[Proposition 4.6]{B1-fornasier2021global} or from \cite[Lemma 16]{fornasier2023consensus1} and we will omit the proof~(the proof is exactly the same, only need to change $v_{\alpha}(\rho_t)$ in their proofs).
	\begin{lemma} \label{lem:5l}Assume \eqref{eq:1919}~($R$ can be $\infty$) has a strong solution for any $t\in [0,T^*]$. Denote $B(w):= \sup_{t\in[0,T^*],w'\in B_r(w)}\normm{v_{\alpha}(\rho^R_t,\cP(w'))-w}$, then we have 
  \begin{equation}\label{eq:8855}
			\int_{B_r(w)}\phi_r^{\tau}(u-w)d\rho_t^R(u)\geq e^{-q't}\int_{B_r(w)}\phi_r^{\tau}(u-w)d\rho_0^R(u),\quad \forall t\in[0,T^*], \forall w\in \mathbb{R}^d,
		\end{equation}
  where $q'$ only depends on $r,\lambda,\sigma,B(w),d,\tau$~(for example, we can choose $\tau=d+2$). And thus we have 
		\begin{equation}\label{eq:88555}
			\inf_{w\in \cV}\int_{B_r(w)}\phi_r^{\tau}(u-w)d\rho_t^R(u)\geq e^{-q't}\inf_{w\in \cV}\int_{B_r(w)}\phi_r^{\tau}(u-w)d\rho_0^R(u),\quad \forall t\in[0,T^*],
		\end{equation}
   where now $q'$ only depends on $r,\lambda,\sigma,B:=\sup_{w\in\cV}B(w),d,\tau$~(for example, we can choose $\tau=d+2$).
		
	\end{lemma}
	\begin{remark}
		\label{rmk:7}
		The above lemma can also be applied to SDE \eqref{eq:n1} and \eqref{eq:92929292} with $\kappa>0$; actually $\kappa>0$ is a good term to lower bound the left hand side of inequality \eqref{eq:8855}\revised{, since with positive $\kappa$, there will always be particles diffused into $B_r(w)$ such that $\rho_t^{R}(B_r(w))>0$. }
	\end{remark}
	With the above preparation, we can now show the main convergence result of this section.
	
	\begin{theorem}\label{prop:33}Under Assumption \ref{asp:11}, Assumption \ref{asp:22} and Assumption \ref{asp:33}, for any error bound $\epsilon>0$, denote 
		$T^*:=\frac{2}{\Cexp}\log(\frac{\cv_0\vee (2\epsilon)}{2\epsilon})$, where $\Cexp$ is from \eqref{eq:8080}. Then for the SDE \eqref{eq:SDEsimple}, we can choose $\varkappa$ small enough and $\alpha$ large enough and with this $\varkappa,\alpha$, we have 
  \begin{equation}\label{eq:decay1}
\mathcal{V}_{t}\leq (\Cexp\epsilon t+\mathcal{V}_0)e^{-\Cexp t} \text{  for all } t \in [0,T^*],
  \end{equation}
  and
    \begin{equation}
  \mathcal{V}_{T^*}\leq2\epsilon.
  \end{equation}
	\end{theorem}
	{ \begin{remark}\label{rem:estT} While \eqref{eq:decay1} ensures that $\mathcal{V}_{t}$ decays exponentially until the target time $T^*$ and accuracy $2 \epsilon$ are reached, our argument does not allow to ensure that after the horizon time $T^*$ the functional $\mathcal{V}_{t}$ will continue to decay. We also cannot give $T^*$ precisely, because the value of $\mathcal{V}_0$ may be unknown a priori. However, as we expect from \eqref{eq:7979} that $\normm{v_\alpha(\rho_t,v)- V^*(v)}\leq \Cctr\normm{v}+\Csmall$, where $\Cctr,\Csmall$ are two very small numbers by choosing $\varkappa,q,r$ small enough and $\alpha$ large enough, for any $v$, $\mathcal{V}_0$ can be estimated by    
 \begin{equation}\label{eq:estV0}
 \mathcal{V}_0 = \int \|v - V^*(v)\|_2^2 d \rho_0(v) \approx \int \|v - v_\alpha(\rho_t,v)\|_2^2 d \rho_0(v),
 \end{equation}
 for 
 $\alpha>0$ large. The approximation in \eqref{eq:estV0} is more accurate for $\rho_0$  concentrated, possibly of compact support, for $\alpha>0$ large, and for $t>0$ large. In fact, the larger $t>0$, the more $\rho_t$ results concentrated on $\cV$, and $E_{t,r,p,q,\sml}$ may get smaller; note also that the bound \eqref{eq:8855} is quite pessimistic.  For $r>0$ small, hence $\alpha>0$ larger, the constant $\Cexp\approx \lambda-2\sigma^2d$. With estimates for $\cv_0$ and $\Cexp$, the horizon time $T^* = \frac{2}{\Cexp}\log(\frac{\cv_0\vee (2\epsilon)}{2\epsilon})$ remains approximately determined as well. 
 \end{remark}}
	\begin{proof}(of Theorem \ref{prop:33}.)
		By Lemma \ref{lem:44} with $R\to\infty$ or Remark \ref{rmk:22}, we have 
		\begin{equation}
			 \sup_{t\in[0,T^*]}\sup_{v\in \cV+B_r(0)}\normm{V^*(v)-v_{\alpha}(\rho_t,v)}<\infty,
		\end{equation}
		and so by Lemma \ref{lem:5l}, we have
		\begin{equation}
			\inf_{t\in [0,T^*]}{\inf_{w\in\cV}\int_{B_r(w)}\phi_r^{\tau}(u-w)\rho_t(u)}>0.
		\end{equation}
		By Lemma \ref{lem:3} with $R\to\infty$, we have
		\begin{equation}
			\sup_{t\in [0,T^*]}\mathcal{V}_t<\infty,
		\end{equation}
		by \eqref{eq:8181}, we have
		\begin{equation}
			\frac{d}{dt}\mathcal{V}_t\leq -\Cexp\mathcal{V}_t+H_{t,r,p,q,\sml}(\alpha,\mathcal{V}_t),
		\end{equation}
		where by definition $\mathcal{V}_t=\int\normmsq{v-V^*(v)}d\rho_t(v)$, $H_{t,r,p,q,\sml}(\alpha,\varkappa,\mathcal{V}_t)=\Cerr\left[r+(E_{t,r,p,q,\sml}(\alpha))^2\mathcal{V}_t\right]+2(\lambda+2\sigma^2d)q$ and $E_{t,r,p,q,\sml}(\alpha,\varkappa)=\frac{e^{-\alpha (1-2(2\varkappa/\sml^2)^{1/(p-1)})}}{\inf_{w\in\cV}\int_{B_r(w)}\phi_r^{\tau}(u-w)\rho_t(u)},$ so we can choose $\varkappa,q,r$ small enough and $\alpha$ large enough such that
		\begin{equation}
			H_{t,r,p,q,\sml}(\alpha,\mathcal{V}_t)\leq {\Cexp\epsilon},
		\end{equation}
		then on $[0,T^*]$, we have
		\begin{equation}
			\frac{d}{dt}\mathcal{V}_t\leq-\Cexp\mathcal{V}_t+\Cexp\epsilon
		\end{equation}
		then by \revised{Gr\"onwall's inequality}, we have
  \begin{equation}\label{eq:in106}
\mathcal{V}_{t}\leq (\Cexp\epsilon t+\mathcal{V}_0)e^{-\Cexp t} \text{  for all } t \in [0,T^*],
  \end{equation}
  and
    \begin{equation}\label{eq:in107}
  \mathcal{V}_{T^*}\leq2\epsilon.
  \end{equation}
  \revised{The last inequality is derived by inserting the definition of $T^*$ into \eqref{eq:in106}, then we have \begin{equation}
      (\Cexp\epsilon T^*+\mathcal{V}_0)e^{-\Cexp T^*}=(2\epsilon\log(\frac{\mathcal{V}_0\vee (2\epsilon)}{2\epsilon})+\mathcal{V}_0)\left(\frac{2\epsilon}{\mathcal{V}_0\vee (2\epsilon)}\right)^2,
  \end{equation}
  	so to derive \eqref{eq:in107}, it is enough to show 
  	\begin{equation}
  	    (2\epsilon\log(\frac{\mathcal{V}_0\vee (2\epsilon)}{2\epsilon})+\mathcal{V}_0)\left(\frac{2\epsilon}{\mathcal{V}_0\vee (2\epsilon)}\right)^2\leq 2\epsilon
  	\end{equation}
  	or equivalently
  	\begin{equation}
  	    \log(\frac{\mathcal{V}_0\vee (2\epsilon)}{2\epsilon})+\frac{\mathcal{V}_0}{2\epsilon}\leq \left(\frac{\mathcal{V}_0\vee (2\epsilon)}{2\epsilon}\right)^2,
  	\end{equation}
  	this is true, since $\log(x)+x\leq 2x-1\leq x^2,\forall x\geq 1.$
  }
	\end{proof}
	
	\section{The case  $\cV=\cup_{i=1}^{\revised{\mm}} \cV_i$}\label{sec:2}
In this section we discuss the case when $\cV=\cV_1\cup\cV_2$ for simplicity, however the analysis can be naturally applied to the case when $\cV=\cup_{i=1}^{\revised{\mm}}\cV_i$ with $\revised{\mm}>2$. In this section $\kappa>0$, and this choice of $\kappa$ can help the dynamic\revised{s}~\eqref{eq:n1} escape from $\partial\zonea$,  \revised{for the function example in the paragraph above \eqref{eq:n1}, if $\kappa=0$ and $\rho_0^0$ is symmetric about $v=0$, then if $V_t^0$ starts from $0$, which is $\partial\zonea$ in this example, it will always stay at $0$, however $0$ is not a global minimizer of $f$, if $\kappa>0$, then $V_t^{\kappa}$ which starts from $0$ has chance to escape from $0$ and converge to $-1$ or $1$.}

	In this case, $\mathbb{R}^d$ will be separated into two zones, the points in zone $\zonea$ is closer to $\cV_1$ and points in zone $\zoneb$ is closer to $\cV_2$~\revised{(let's say each zone is open).} 
 Define $V^*_1(v):=\arg\min_{w\in \cV_1}\normm{v-w}$ and $V^*_2(v)=\arg\min_{w\in \cV_2}\normm{v-w}$, then we know $V^*(v)=V^*_1(v)$ for $v\in\zonea$ and $V^*(v)=V^*_2(v)$ for $v\in\zoneb$. 

We analyze first the SDE~\eqref{eq:n1} and its related Fokker--Planck equation~\eqref{eq:n3}. The existence of a solution to \eqref{eq:n1} and \eqref{eq:n3} is shown in Theorem~\ref{prop:2222}, the proof is similar to the simple case and is also based on a limiting process. The main convergence result of this section is presented in Theorem~\ref{prop:66}.
 
	\subsection{Existence of solutions for the general case}
	
 Similar to \eqref{prop:11}, we have the following proposition.
	\begin{proposition}\label{prop:333}Under Assumption \ref{asp:22}, 
		the following SDE
		\begin{equation}\label{eq:92929292}
			dV^{\kappa,R}_t=-\lambda(V^{\kappa,R}_t-v_{\alpha}(\rho^{\kappa,R}_t,\cP(V^{\kappa,R}_t))dt+\sigma\left(\normm{V^{\kappa,R}_t-v_{\alpha}(\rho^{\kappa,R}_t,\cP(V^{\kappa,R}_t)}+\kappa\right)dB_t,
		\end{equation}
		where $ R\in (0,\infty),\kappa\geq 0$, $\rho_t^{\kappa,R}$ is the distribution of $V_t^{\kappa,R}$ and $\rho_0^{\kappa,R}=\rho_0\in\mathcal{P}_4(\mathbb{R}^d)$, has a strong solution and $\rho_t^{\kappa,R}$ satisfies the related Fokker--Planck equation in the weak sense.
	\end{proposition}
	\begin{proof}
		The proof is the same as the proof of Proposition \ref{prop:11}, just notice the additional $\kappa$.
	\end{proof}
	\begin{lemma}\label{lem:6666} Under Assumption \ref{asp:22} and Assumption \ref{asp:44}, when $R\in (0,\infty),\kappa> 0,\rho_0^{\kappa,R}=\rho_0\in \mathcal{C}^2_b(\mathbb{R}^d)$, we have 
		\begin{equation}
			\begin{aligned}
				\frac{d}{dt}\int\normmsq{v-V^*(v)}d\rho^{\kappa,R}_t(v)
				&\leq-(\lambda-2\sigma^2d)\int\normmsq{v-V^*(v)}d\rho_t^{\kappa,R}(v)\\
				&\quad+(\lambda+2\sigma^2d)\int\left(\normm{V^*(v)-v_{\alpha}(\rho_t^{\kappa,R},\cP(v))}+\kappa\right)^2d\rho_t^{\kappa,R}(v).
			\end{aligned}
		\end{equation}
	\end{lemma}
\begin{proof}

    In this case, $V^*(v)$ is smooth inside $\zonea,\zoneb$, but { may not be  smooth enough on the whole space} $\mathbb{R}^d$. For this reason, in the course of the proof  we would need to discuss { \begin{equation}
\int_{\zonea}\normmsq{v-V^*(v)} \operatorname{diff}(\rho^{\kappa,R}_t)(v) dv, \text { and } \int_{\zoneb}\normmsq{v-V^*(v)}\operatorname{diff}(\rho^{\kappa,R}_t)(v) dv,        
    \end{equation} respectively, for suitable differential operators $\operatorname{diff}$. In particular, we would need $\rho_t^{\kappa,R}$ to be  $\mathcal{C}^2$ in the space variable, and thus the possibility of applying the Green's identity. Unfortunately, here we can not immediately ensure $\rho_t^{\kappa,R}\in\mathcal{C}^2$ in the space variable, so instead of approaching directly the calculation of the two terms, we consider first a regularization and then a limiting process.} Therefore, we first study the following regularized Fokker--Planck equation:
    \newcommand{\na}{\chi}
    \begin{alignat}{2}
{\partial_t}\rho_t^{\kappa,R,i}&=\lambda\operatorname{div}\left(\na_{i,t}(v)\rho_t^{\kappa,R,i}\right)+\frac{\sigma^2}{2}\Delta\left(({\normm{\na_{i,t}(v)}+\kappa})^2\rho_t^{\kappa,R,i}\right)\label{eq:923},\\
\rho_0^{\kappa,R,i}&=\rho_0\in \mathcal{C}^2_b(\mathbb{R}^d),
    \end{alignat}
    here
\begin{equation}
        \na_{i,t}(v)=\Phi_{i}(v-\varphi_{2^{-i}}* v_{\alpha}(\rho^{\kappa,R}_t,\cP(v))),
    \end{equation}
    where $*$ \revised{denotes} the convolution, $\varphi_\epsilon$ is the typical mollifier and $\Phi_i:\mathbb{R}^d\to\mathbb{R}^d$ is a smooth map with $\Phi_i(v)\in B_{2^i}(0), $ for all $ v\in\mathbb{R}^d$, $\Phi_i(v)=v,$ for all $ v\in B_{2^{i-1}}(0)$ and has bounded $\mathcal{C}^3$ norm and $\|\nabla \Phi_i\|_{\mathrm{op}}\leq 1$~(it is not hard to construct such $\Phi_i$, here we omit the construction of $\Phi_i$), thus we have $\lim_{i\to\infty}\na_{i,t}(v)=v-v_{\alpha}(\rho_t^{\kappa,R},\cP(v))$, for all  $v\in\mathbb{R}^d,t\in[0,T^*]$. Since now $\na_{i,t}$ has bounded $\mathcal{C}^3$ norm and $\na_{i,t}(v)$ is H\"older continuous in terms of $t$ due to the $1/2$-H\"older continuity of $v_{\alpha}(\rho_t^{\kappa,R},\cP(v))$ with respect to $t$~(by a similar analysis as in \eqref{eq:h29}),  by an application of  \cite[Theorem 6.6.1]{bogachev2022fokker}, we have $\rho_t^{\kappa,R,i}\in\mathcal{C}^{2,1}(\mathbb{R}^d\times [0,T^*])$~(which \revised{denotes} it has continuous derivatives up to second order in the space variable and continuous first order derivative in the time variable).

    In the rest of \revised{the} proof, for simplicity, we will omit $\kappa,R$,
unless otherwise stated. By the Fokker--Planck equation~\eqref{eq:923}, since now $\rho_t^{i}\in\mathcal{C}^{2,1}(\mathbb{R}^d\times [0,T^*])$, we have
		\begin{equation}\label{eq:9494}
			\begin{aligned}
				&\frac{d}{dt}\int\normmsq{v-V^*(v)}d\rho^i_t(v)\\
				&=\underbrace{\lambda\int\normmsq{v-V^*(v)}\operatorname{div}(\na_{i,t}(v)\rho^i_t(v))dv}_{\text{drift}}+\underbrace{\frac{\sigma^2}{2}\int\normmsq{v-V^*(v)}\Delta((\normm{\na_{i,t}(v)}+\kappa)^2\rho^i_t(v))dv}_{\text{diffusion}}.
			\end{aligned}
		\end{equation}
\textbf{Drift Term:}
		In $\zonea$, we have $V^*(v)=V^*_1(v)$ and
		\begin{equation}\label{eq:9595}
			\begin{aligned}
				&\lambda\int_{\zonea}\normmsq{v-V^*_1(v)}\operatorname{div}(\na_{i,t}(v)\rho_t^i(v))dv\\
				&=-\lambda\int_{\zonea}\inner{\nabla \normmsq{v-V_1^*(v)}}{\na_{i,t}(v)}d\rho_t^i(v)\\
				&\quad +\int_{\partial \zonea}\inner{\normmsq{v-V^*_1(v)}\na_{i,t}(v)\rho_t^i(v)}{n_1}dS,
			\end{aligned}
		\end{equation}
		where $n_1$ is the outer unit normal vector of $\partial \zonea$ at \revised{a} point $v$. Similarly, we have on $\zoneb$, $V^*(v)=V^*_2(v)$ and 
		\begin{equation}\label{eq:9696}
			\begin{aligned}
				&\lambda\int_{\zoneb}\normmsq{v-V^*_2(v)}\operatorname{div}(\na_{i,t}(v)\rho_t^i(v))dv\\
				&=-\lambda\int_{\zoneb}\inner{\nabla \normmsq{v-V_2^*(v)}}{\na_{i,t}(v)}d\rho_t^i(v)\\
				&\quad +\int_{\partial \zoneb}\inner{\normmsq{v-V^*_2(v)}\na_{i,t}(v)\rho_t^i(v)}{n_2}dS,
			\end{aligned}
		\end{equation}
		where $n_2$ is the outer unit normal vector of $\partial \zoneb$ at \revised{a} point $v$, we have $n_1(v)=-n_2(v)$ for \revised{a} point $v$ on $\partial \zonea~(=\partial \zoneb)$.
		Combining \eqref{eq:9595} and \eqref{eq:9696}, we have
		\begin{equation}\label{eq:288}
			\begin{aligned}
				&\lambda\int_{\mathbb{R}^d}\normmsq{v-V^*(v)}\operatorname{div}(\na_{i,t}(v)\rho_t^i(v))dv\\
				&=\lambda\left[\int_{\zonea}\normmsq{v-V^*_1(v)}\operatorname{div}(\na_{i,t}(v)\rho_t^i(v))dv+\int_{\zoneb}\normmsq{v-V^*_2(v)}\operatorname{div}(\na_{i,t}(v)\rho_t^i(v))dv\right]\\
				&=-\lambda\left[\int_{\zonea}\inner{\nabla \normmsq{v-V_1^*(v)}}{\na_{i,t}(v)}d\rho_t^i(v)+\int_{\zoneb}\inner{\nabla \normmsq{v-V_2^*(v)}}{\na_{i,t}(v)}d\rho_t^i(v)\right]\\
				&=-\lambda\int_{\mathbb{R}^d}\inner{\nabla \normmsq{v-V^*(v)}}{\na_{i,t}(v)}d\rho_t^i(v),
			\end{aligned}
		\end{equation}
		the two boundary terms cancel, since on $\partial \zonea$, we have $\normmsq{v-V_1^*(v)}=\normmsq{v-V_2^*(v)}$ by the definition of $\zonea$ and $\zoneb$, and $n_1=-n_2$.  For the last line in \eqref{eq:288}, we can deal with it just like the case where $\cV$ is  convex, compact and has smooth boundary~(we can ignore the points $v$ on the $\partial \zonea$, since $\partial \zonea$ has dimension less than $d$ and so $\rho_t^i(\partial \zonea)=0$ since $\rho_t^i$ is absolutely continuous).

  \textbf{Diffusion Term:}
		In $\zonea$, we have $V^*(v)=V^*_1(v)$  and
		\begin{equation}\label{eq:9797}
			\begin{aligned}
				&\frac{\sigma^2}{2}\int_{\zonea}\normmsq{v-V^*_1(v)}\Delta((\normm{\na_{i,t}(v)}+\kappa)^2\rho_t^i(v))dv\\
				&=\frac{\sigma^2}{2}\int_{\zonea}\Delta(\normmsq{v-V^*_1(v)})(\normm{\na_{i,t}(v)}+\kappa)^2d\rho_t^i(v)\\
				&\quad+\frac{\sigma^2}{2}\int_{\partial \zonea}\normmsq{v-V^*_1(v)}\frac{\partial((\normm{\na_{i,t}(v)}+\kappa)^2\rho_t^i(v))}{\partial n_1}dS\\
				&\quad -\frac{\sigma^2}{2}\int_{\partial \zonea}\frac{\partial \normmsq{v-V_1^*(v)}}{\partial n_1}(\normm{\na_{i,t}(v)}+\kappa)^2\rho_t^i(v)dS,
			\end{aligned}
		\end{equation}
		in the above we used \revised{Green's} identity and $n_1$ is the unit outer normal vector of 
		$\partial \zonea$ at a point $v$; similarly, on $\zoneb$, we have $V^*(v)=V^*_2(v)$ and 
		\begin{equation}\label{eq:9898}
			\begin{aligned}
				&\frac{\sigma^2}{2}\int_{\zoneb}\normmsq{v-V^*_2(v)}\Delta((\normm{\na_{i,t}(v)}+\kappa)^2\rho_t^i(v))dv\\
				&=\frac{\sigma^2}{2}\int_{\zoneb}\Delta(\normmsq{v-V^*_2(v)})(\normm{\na_{i,t}(v)}+\kappa)^2d\rho_t^i(v)\\
				&\quad+\frac{\sigma^2}{2}\int_{\partial \zoneb}\normmsq{v-V^*_2(v)}\frac{\partial((\normm{\na_{i,t}(v)}+\kappa)^2\rho_t^i(v))}{\partial n_2}dS\\
				&\quad -\frac{\sigma^2}{2}\int_{\partial \zoneb}\frac{\partial \normmsq{v-V_2^*(v)}}{\partial n_2}(\normm{\na_{i,t}(v)}+\kappa)^2\rho_t^i(v)dS,
			\end{aligned}
		\end{equation}
		where $n_2=-n_1$ is the unit outer normal vector of $\partial \zoneb$ at a point $v$. Combining \eqref{eq:9797} and \eqref{eq:9898}, we have 
		\begin{equation}\label{eq:311}
			\begin{aligned}
				&\frac{\sigma^2}{2}\int_{\mathbb{R}^d}\normmsq{v-V^*(v)}\Delta((\normm{\na_{i,t}(v)}+\kappa)^2\rho_t^i(v))dv\\
				&=\frac{\sigma^2}{2}\left[ \int_{\zonea}\Delta(\normmsq{v-V^*_1(v)})(\normm{\na_{i,t}(v)}+\kappa)^2d\rho_t^i(v)\right.\\
				&\quad\left.+\int_{\zoneb}\Delta(\normmsq{v-V^*_2(v)})(\normm{\na_{i,t}(v)}+\kappa)^2d\rho_t^i(v)\right]\\
				&\quad+\frac{\sigma^2}{2}\left[\int_{\partial \zonea}\normmsq{v-V^*_1(v)}\frac{\partial((\normm{\na_{i,t}(v)}+\kappa)^2\rho_t^i(v))}{\partial n_1}dS\right.\\
				&\quad\left.+\int_{\partial \zoneb}\normmsq{v-V^*_2(v)}\frac{\partial((\normm{\na_{i,t}(v)}+\kappa)^2\rho_t^i(v))}{\partial n_2}dS\right]\\
				&\quad-\frac{\sigma^2}{2}\left[\int_{\partial \zonea}\frac{\partial \normmsq{v-V_1^*(v)}}{\partial n_1}(\normm{\na_{i,t}(v)}+\kappa)^2\rho_t^i(v)dS\right.\\
				&\quad\left.-\int_{\partial \zonea}\frac{\partial \normmsq{v-V_2^*(v)}}{\partial n_1}(\normm{\na_{i,t}(v)}+\kappa)^2\rho_t^i(v)dS\right],
			\end{aligned}
		\end{equation}
		\revised{the first term in the right hand side of \eqref{eq:311} becomes $\frac{\sigma^2}{2}\int_{\mathbb{R}^d}\Delta\normmsq{v-V^*(v)}(\normm{\na_{i,t}(v)}+\kappa)^2d\rho_t^i(v)$; the second term vanishes due to $\partial\zonea=\partial \zoneb, n_1=-n_2$, $\normm{v-V^*_1(v)}=\normmsq{v-V^*_2(v)}$ for any $v\in \partial \zonea$; in the third term of the right hand side of \eqref{eq:311}, the partial derivatives can be directly computed. So combining all three terms, we have}
		\begin{equation}
		\begin{aligned}
			&\frac{\sigma^2}{2}\int_{\mathbb{R}^d}\normmsq{v-V^*(v)}\Delta((\normm{\na_{i,t}(v)}+\kappa)^2\rho_t^i(v))dv\\
			&=\frac{\sigma^2}{2}\int_{\mathbb{R}^d}\Delta\normmsq{v-V^*(v)}(\normm{\na_{i,t}(v)}+\kappa)^2d\rho_t^i(v)\\
		&\quad-{\sigma^2}\left[\int_{\partial \zonea}\inner{(I-\nabla V^*_1(v)(v-V^*_1(v))^{\top})}{n_1}(\normm{\na_{i,t}(v)}+\kappa)^2\rho_t^i(v)dS\right.\\
		&\quad\left.-\int_{\partial\zonea}\inner{(I-\nabla V^*_2(v)(v-V^*_2(v))^{\top})}{n_1}(\normm{\na_{i,t}(v)}+\kappa)^2\rho_t^i(v)dS\right]\\
		&=\frac{\sigma^2}{2}\int_{\mathbb{R}^d}\Delta\normmsq{v-V^*(v)}(\normm{\na_{i,t}(v)}+\kappa)^2d\rho_t^i(v)\\
		&-\sigma^2\int_{\partial \zonea}\inner{V_2^*(v)-V^*_1(v)}{n_1}(\normm{\na_{i,t}(v)}+\kappa)^2\rho_t^i(v)dS\\
		&\leq \frac{\sigma^2}{2}\int_{\mathbb{R}^d}\Delta\normmsq{v-V^*(v)}(\normm{\na_{i,t}(v)}+\kappa)^2d\rho_t^i(v),
		\end{aligned}
		\end{equation}
  where we used $\nabla V^*_i(v)(v-V^*_i(v))=0$ for $v\in \mathbb{R}^d,i=1,2$~(see \eqref{eq:17}). The last inequality can be obtained as follows: we know $\normmsq{v-V^*_1(v)}-\normmsq{v-V^*_2(v)}=0$ for any $v\in \partial\zonea$, so we know $\partial\zonea$ is a contour of \revised{the}
		function $K(v):=\normmsq{v-V^*_1(v)}-\normmsq{v-V^*_2(v)}$ and $K(v+\epsilon n_1)>0$  for sufficiently small $\epsilon>0$, so we know
		$n_1$ is the gradient direction of the function $K(\cdot)$ at a point $v$, that is $\nabla K(v)=a n_1$ for some \revised{non-negative} $a$, so we have
		\begin{equation}
			\nabla K(v)=2(I-\nabla V^*_1(v))(v-V^*_1(v))-2(I-\nabla V^*_2(v))(v-V^*_2(v))=2(V^*_2(v)-V^*_1(v))= an_1,
		\end{equation}
		and so $\inner{V^*_2(v)-V^*_1(v)}{n_1}\geq 0$.
  
  By combining now the estimates for the diffusion and drift terms, we have 
  \begin{equation}
    \begin{aligned}
        &\frac{d}{dt}\int\normmsq{v-V^*(v)}d\rho_t^{i}(v)\\
        &\leq-\lambda\int \inner{\nabla \normmsq{v-V^*(v)}}{\na_{i,t}(v)}d\rho_t^{i}(v)+\frac{\sigma^2}{2}\int\Delta\normmsq{v-V^*(v)}(\normm{\na_{i,t}(v)}+\kappa)^2d\rho_t^{i}(v),
    \end{aligned}
\end{equation}
in the integral form, we have for any $t_1\leq t_2 \in [0,T^*]$, that
\begin{equation}\label{eq:103103}
    \begin{aligned}
        &\int\normmsq{v-V^*(v)}d\rho_{t_2}^{i}(v)-\int\normmsq{v-V^*(v)}d\rho^i_{t_1}(v)\\
        &\leq -\lambda\int_{t_1}^{t_2}\int \inner{\nabla \normmsq{v-V^*(v)}}{\na_{i,s}(v)}d\rho_s^{i}(v)ds\\
        &\quad+\frac{\sigma^2}{2}\int_{t_1}^{t_2}\int\Delta\normmsq{v-V^*(v)}(\normm{\na_{i,s}(v)}+\kappa)^2d\rho_s^{i}(v)ds.
    \end{aligned}
\end{equation}
Next, we will show that 
\begin{equation*}
    \normmsq{v-V^*(v)}-\normmsq{w-V^*(w)}\leq \Cnineteen(1+\normm{w}+\normm{v})\normm{v-w},
\end{equation*}
for some constant $\Cnineteen$ only depending on $\cV$: when $v,w$ are both in $\zonea$ or $\zoneb$, we have 
$\normmsq{V^*(v)-V^*(w)}\leq \normmsq{v-w}$ \revised{due to the convexity of $\zonea,\zoneb$,} then 
\begin{equation}\label{eq:tmpp1}
    \begin{aligned}
        \normmsq{v-V^*(v)}-\normmsq{w-V^*(w)}&=\inner{v-V^*(v)-w+V^*(w)}{v-V^*(v)+w-V^*(w)}\\
        &\leq \normm{v-V^*(v)+w-V^*(w)}\normm{v-V^*(v)-w+V^*(w)}\\
        &\leq \Ceightteen(1+\normm{w}+\normm{v})(\normm{v-w}+\normm{V^*(v)-V^*(w)})\\
        &\leq 2\Ceightteen(1+\normm{w}+\normm{v})\normm{v-w},
    \end{aligned}
\end{equation}
where $\Ceightteen$ only depends on $\cV$;
when $v\in\zonea$ and $w\in\zoneb$~(the other case $v\in\zoneb$ and $w\in\zonea$ can be considered similarly), assume $u(v,w)\in\partial\zonea$~(for simplicity, we will denote it as $u$ in the next) and satisfies $\normm{v-w}=\normm{v-u}+\normm{u-w}$, then, \revised{with slight abuse of notation in the use of $V^*(u)$}, we have
\begin{equation}
    \begin{aligned}
    \normmsq{v-V^*(v)}-\normmsq{w-V^*(w)}&=\normmsq{v-V^*(v)}-\normmsq{u-V^*(u)}+\normmsq{u-V^*(u)}-\normm{w-V^*(w)}\\
    &\leq 2\Ceightteen(1+\normm{v}+\normm{u})\normm{v-u}+ 2\Ceightteen(1+\normm{w}+\normm{u})\normm{w-u}\\
    &\leq4\Ceightteen(1+\normm{v}+\normm{w})(\normm{v-u}+\normm{u-w})\\
    &\leq 4\Ceightteen(1+\normm{v}+\normm{w})\normm{v-w},
    \end{aligned}
\end{equation}
the first inequality is due to \eqref{eq:tmpp1} and the second inequality is due to $\normm{u}\leq \max\{\normm{v},\normm{w}\}$; and so combine the above analsyis together, we have, for any $v,w\in\mathbb{R}^d$, that
\begin{equation}\label{eq:lala}
    \normmsq{v-V^*(v)}-\normmsq{w-V^*(w)}\leq \Cnineteen(1+\normm{w}+\normm{v})\normm{v-w},
\end{equation}
where $\Cnineteen$ depends on $\cV$.

Following the proof of Lemma~\ref{lem:444}, we can show that $\lim_{i\to\infty}W_2^2(\rho_s^i,\rho_s^{\kappa,R})=0$, for all $s\in [0,T^*]$. Let $\pi^i_s$ be the optimal { transport} coupling between $\rho_s^i$ and $\rho_s^{\kappa,R}$, then we have 
\begin{equation}\label{eq:104104}
    \begin{aligned}
       &\Big|\int\normmsq{v-V^*(v)}d\rho_s^{i}(v)-\int\normmsq{w-V^*(w)}d\rho_s^{\kappa,R}(w)\Big|\\
       &\leq \Big|\iint\left[\normmsq{v-V^*(v)}-\normmsq{w-V^*(w)}\right]d\pi^i_s(v,w)\Big|\\
      &\leq \Big|\iint \Cnineteen(1+\normm{w}+\normm{v})\normm{v-w}d\pi^i_s(v,w)\Big|\\
      &\leq \Cnineteen\left(1+\sqrt{\int\normmsq{w}d\rho_s^{\kappa,R}}+\sqrt{\int\normmsq{v}d\rho_s^i(v)}\right)W_2(\rho^i_s,\rho_s^{\kappa,R})\\
       &\leq \Ctwenty W_2(\rho^i_s,\rho_s^{\kappa,R})\to 0\quad \text{as $i\to\infty$},
    \end{aligned}
\end{equation}
in the above, the constant $\Ctwenty$ only depends on the second momentum of $\rho_s^i, \rho_s^{\kappa,R}$ and $\cV$ which are all bounded. So we have shown the left hand side of $\eqref{eq:103103}$ converges to \begin{equation}
    \int\normmsq{v-V^*(v)}d\rho_{t_2}^{\kappa,R}(v)-\int\normmsq{v-V^*(v)}d\rho^{\kappa,R}_{t_1}(v).
\end{equation} By the explicit formulas for $\nabla \normmsq{v-V^*(v)}$ and $\Delta\normmsq{v-V^*(v)}$~(see the proof of Lemma~\ref{lem:66}), the fact that $\lim_{i\to\infty}\normm{\na_{i,s}(v)}=\normm{v-v_{\alpha}(\rho_s^{\kappa,R},\cP(v))}$, the $1$-Lipschitz continuity of $\Phi_i$, and a similar analysis as in \eqref{eq:104104}, we can also show, for any $s\in [0,T^*]$, that
\begin{equation}
    \begin{aligned}
        &\lim_{i\to\infty}\int \inner{\nabla \normmsq{v-V^*(v)}}{\na_{i,s}(v)}d\rho_s^{i}(v)\\
        &=\int \inner{\nabla \normmsq{v-V^*(v)}}{v-v_{\alpha}(\rho_s^{\kappa,R},\cP(v))}d\rho_s^{\kappa,R}(v)
    \end{aligned}
\end{equation}
and
\begin{equation}
    \begin{aligned}
        &\lim_{i\to\infty}\int\Delta\normmsq{v-V^*(v)}(\normm{\na_{i,s}(v)}+\kappa)^2d\rho_s^{i}(v)\\
        &=\int\Delta\normmsq{v-V^*(v)}(\normm{v-v_{\alpha}(\rho_s^{\kappa,R},\cP(v))}+\kappa)^2d\rho_s^{\kappa,R}(v)
    \end{aligned}
\end{equation}
and  both terms are uniformly bounded for $s\in[0,T^*]$. Thus by \revised{the} dominated convergence theorem, we have that the right hand side of \eqref{eq:103103} converges to 
\begin{equation}
    \begin{aligned}
        &-\lambda\int_{t_1}^{t_2}\int \inner{\nabla \normmsq{v-V^*(v)}}{v-v_{\alpha}(\rho_s^{\kappa,R},\cP(v))}d\rho_s^{\kappa,R}(v)ds\\
        &\quad+\frac{\sigma^2}{2}\int_{t_1}^{t_2}\int\Delta\normmsq{v-V^*(v)}(\normm{v-v_{\alpha}(\rho_s^{\kappa,R},\cP(v))}+\kappa)^2d\rho_s^{\kappa,R}(v)ds.
    \end{aligned}
\end{equation}
So combine the above analysis and we have, for any $t_1\leq t_2 \in [0,T^*]$, that
\begin{equation}
    \begin{aligned}
        &\int\normmsq{v-V^*(v)}d\rho_{t_2}^{\kappa,R}(v)-\int\normmsq{v-V^*(v)}d\rho^{\kappa,R}_{t_1}(v)\\
        &\leq -\lambda\int_{t_1}^{t_2}\int \inner{\nabla \normmsq{v-V^*(v)}}{v-v_{\alpha}(\rho_s^{\kappa,R},\cP(v))}d\rho_s^{\kappa,R}(v)ds\\
        &\quad+\frac{\sigma^2}{2}\int_{t_1}^{t_2}\int\Delta\normmsq{v-V^*(v)}(\normm{v-v_{\alpha}(\rho_s^{\kappa,R},\cP(v))}+\kappa)^2d\rho_s^{\kappa,R}(v)ds,
    \end{aligned}
\end{equation}
which is equivalent to 
 \begin{equation}
    \begin{aligned}
        &\frac{d}{dt}\int\normmsq{v-V^*(v)}d\rho_t^{\kappa,R}(v)\\
        &\leq-\lambda\int \inner{\nabla \normmsq{v-V^*(v)}}{v-v_{\alpha}(\rho^{\kappa,R}_t,\cP(v))}d\rho_t^{\kappa,R}(v)\\
        &\quad+\frac{\sigma^2}{2}\int\Delta\normmsq{v-V^*(v)}(\normm{v-v_{\alpha}(\rho^{\kappa,R}_t,\cP(v))}+\kappa)^2d\rho_t^{\kappa,R}(v).
    \end{aligned}
\end{equation}
In the following, similar to the proof of Lemma \ref{lem:66}, we finally have
\begin{equation}
    \begin{aligned}
        &\frac{d}{dt}\int\normmsq{v-V^*(v)}d\rho_t^{\kappa,R}(v)\\
         &\leq-\lambda\int \inner{\nabla \normmsq{v-V^*(v)}}{v-v_{\alpha}(\rho_t^{\kappa,R},\cP(v))}d\rho_t^{\kappa,R}(v)\\
         &\quad+\frac{\sigma^2}{2}\int\Delta\normmsq{v-V^*(v)}(\|v-v_{\alpha}(\rho_t^{\kappa,R},\cP(v))\|+\kappa)^2d\rho_t^{\kappa,R}(v)\\
        &=-2\lambda\int (v-V^*(v))(I-\nabla V^*(v))(v-V^*(v))^{\top}d\rho_t^{\kappa,R}(v)\\
        &\quad -2\lambda\int (V^*(v)-v_{\alpha}(\rho_t^{\kappa,R},\cP(v)))(I-\nabla V^*(v)(v-V^*(v))d\rho_t^{\kappa,R}(v)\\
        &\quad +\sigma^2\int \operatorname{tr}(I-\nabla V^*(v))(\|v-v_{\alpha}(\rho_t^{\kappa,R},\cP(v))\|+\kappa)^2d\rho_t^{\kappa,R}(v)\\
        &\leq -2\lambda\int\normmsq{v-V^*(v)}d\rho_t^{\kappa,R}(v)\\
        &\quad+2\lambda\sqrt{\int\normmsq{v-V^*(v)}d\rho_t^{\kappa,R}(v)}\sqrt{\int \normmsq{V^*(v)-v_{\alpha}(\rho_t^{\kappa,R},\cP(v))}d\rho_t^{\kappa,R}(v)}\\
        &\quad+2\sigma^2d\int\normmsq{v-V^*(v)}d\rho_t^{\kappa,R}(v)+2\sigma^2d\int \left (\normm{V^*(v)-v_{\alpha}(\rho_t^{\kappa,R},\cP(v))}+\kappa \right )^2d\rho_t^{\kappa,R}(v)\\
        &\leq-(\lambda-2\sigma^2d)\int\normmsq{v-V^*(v)}d\rho_t^{\kappa,R}(v)\\
        &\quad+(\lambda+2\sigma^2d)\int \left(\normm{V^*(v)-v_{\alpha}(\rho_t^{\kappa,R},\cP(v))}+\kappa \right )^2d\rho_t^{\kappa,R}(v).
    \end{aligned}
\end{equation} 
\end{proof}

\revised{\begin{remark}When ${\mm}>2$, the integration should be considered in each zone $\mathcal{Z}_i$, the boundary is separated into $\mathcal{Z}_i\cap\mathcal{Z}_j,i\not=j$,  thus the surface integration should be considered in each $\mathcal{Z}_i\cap\mathcal{Z}_j$~($=\mathcal{Z}_j\cap\mathcal{Z}_i$) and $n_i(v)$ for a point $v$ on $\mathcal{Z}_i\cap\mathcal{Z}_j$ equals $-n_j(v)$ for the same point on $\mathcal{Z}_j\cap\mathcal{Z}_i$. For example, \eqref{eq:9595} should be 
		\begin{equation}
		\begin{aligned}
		&\lambda\int_{\mathcal{Z}_1}\normmsq{v-V^*_1(v)}\operatorname{div}(\chi_{i,t}(v)\rho_t^i(v))dv\\
		&=-\lambda\int_{\zonea}\inner{\nabla \normmsq{v-V_1^*(v)}}{\chi_{i,t}(v)}d\rho_t^i(v)\\
		&\quad +\sum_{j=2}^{{\mm}}\int_{\mathcal{Z}_1\cap\mathcal{Z}_j}\inner{\normmsq{v-V^*_1(v)}\chi_{i,t}(v)\rho_t^i(v)}{n_1}dS.
		\end{aligned}
		\end{equation}
 
\end{remark}}

In the following, we will analyze the properties of \revised{the} function $A(w,v):= \frac{1}{\varkappa}f(w)(f(v)+\Cone)+\normmsq{v-w}$, here $\Cone\geq\sml \operatorname{diam}(\cV)^p-\underbar{f}$. 

\revised{\begin{lemma}
	\label{lem:999}	
	Under Assumption \ref{asp:11} and Assumption \ref{asp:44}, let $\Cone\geq \sml\operatorname{diam}(\cV)^p-\underbar{f}$, then we have 
	\begin{equation}\label{eq:8877}
		A(w,v)-A(V_1^*(v),v)\geq \normmsq{w-V_1^*(v)}-2(\frac{2^p\varkappa}{\sml^2})^{\frac{1}{p-1}},\quad \text{for all } w\in \zonea,
	\end{equation}
	and 
	\begin{equation}
		A(w,v)-A(V_2^*(v),v)\geq \normmsq{w-V_2^*(v)}-2(\frac{2^p\varkappa}{\sml^2})^{\frac{1}{p-1}},\quad \text{for all } w\in \zoneb.
	\end{equation}
\end{lemma}
}
\begin{proof}
\revised{The proof is similar to the proof of Lemma \ref{lem:2}. We will assume $v\in\zonea$ and so $V^*(v)=V^*_1(v)$.
	For simplicity, we will denote $a:=V_2^*(v)$.
	Let $\left\{e_i\right\}_{i=1}^d$ be some basis of $\mathbb{R}^d$, such that  $v-a=v-V_2^*(v)=te_d$ for some $t\geq 0$ and we denote $w=\sum_{i=1}^dw_ie_i$ for any $w\in\mathbb{R}^d$ such that $w+a\in\zoneb$, we know $\inner{v-a}{w}=(v_d-a_d)w_d$.
	When $w_d\geq 0$, we have
	\begin{equation}
		\begin{aligned}
			&A(a+w,v)-A(a,v)\\
			&=\frac{1}{\varkappa}(f(a+w)-f(a))(f(v)+\Cone)+\normmsq{a+w-v}-\normmsq{a-v}\\
			&\geq\frac{\sml}{\varkappa}\dist{a+w,\cV_2}^p({\sml}\normm{v-V_1^*(v)}^p+\Cone+\underbar{f})+\normmsq{w}-2|a_d-v_d|w_d\\
			&\geq \frac{\sml}{\varkappa}w_d^p({\sml}2^{1-p}\normm{v-a}^p-{\sml}\normm{{V_2^*(v)}-V_1^*(v)}^p+\Cone+\underbar{f})+\normmsq{w}-2|a_d-v_d|w_d\\
			&\geq \normmsq{w}+\frac{\sml^2}{\varkappa}2^{1-p}\normm{v-a}^pw_d^p-2|a_d-v_d|w_d\\
			&\geq \normmsq{w}-2(\frac{2^p\varkappa}{\sml^2})^{\frac{1}{p-1}},
		\end{aligned}
	\end{equation}
 in the above, the second inequality is due to $|c-b|^p+b^p\geq 2^{1-p}c^p$ for $c,b>0$, the third inequality is due to $\Cone\geq {\sml}\operatorname{diam}(\cV)^p-\underbar{f}$ and so $\Cone+\underbar{f}-{\sml}\normm{{V_1^*(v)}-V_2^*(v)}^p\geq 0$, the derivation of the fourth inequality is the same as the one in \eqref{ieq:29}.} When $w_d<0$, since we already have $f(a+w)-f(a)\geq 0$, we obtain
	\begin{equation}
		\begin{aligned}
			A(a+w,v)-A(a,v)&\geq  \normmsq{a+w-v}-\normmsq{a-v}\\
			&=\normmsq{w}-2|a_d-v_d|w_d\\
			&\geq \normmsq{w}.
		\end{aligned}
	\end{equation}
	So, for any $w\in \zoneb$, we have
	\begin{equation}
		A(w,v)-A(V_2^*(v),v)\geq \normmsq{w-a}-2(\frac{2^p\varkappa}{\sml^2})^{\frac{1}{p-1}}.
	\end{equation}
	
	Similarly, denote $a':=V_1^*(v)$.
	Let $\left\{e'_i\right\}_{i=1}^d$ be some basis of $\mathbb{R}^d$, such that $v-a'=v-V_1^*(v)=t'e'_d$ for some $t'\geq 0$ and we denote $w'=\sum_{i=1}^dw'_ie'_i$ for any $w\in\mathbb{R}^d$ such that $w'+a'\in\zonea$, we know $\inner{v-a'}{w}=(v_d-a'_d)w_d$.
	When $w'_d\geq 0$, we have
	\begin{equation}
		\begin{aligned}
			&A(a'+w',v)-A(a',v)\\
			&=\frac{1}{\varkappa}(f(a'+w')-f(a'))(f(v)+\Cone)+\normmsq{a'+w'-v}-\normmsq{a'-v}\\
			&\geq \frac{\sml}{\varkappa}\dist{a'+w',\cV_1}^p({\sml}\normm{v-V_1^*(v)}^p+\Cone+\underbar{f})+\normmsq{w'}-2|a'_d-v_d|w'_d\\
			&\geq \frac{\sml}{\varkappa}{w'_d}^p({\sml}\normm{v-a'}^p+\Cone+\underbar{f})+\normmsq{w'}-2|a'_d-v_d|w'_d\\
			&\geq \normmsq{w'}+\frac{\sml^2}{\varkappa}\normm{v-a'}^p{w'_d}^p-2|a'_d-v_d|w'_d\\
			&\geq \normmsq{w'}-2(\frac{2\varkappa}{\sml^2})^{\frac{1}{p-1}}\\
   &\geq \normmsq{w'}-2(\frac{2^p\varkappa}{\sml^2})^{\frac{1}{p-1}}
		\end{aligned}
	\end{equation}
	When $w'_d<0$, we have
	\begin{equation}
		\begin{aligned}
			A(a'+w',v)-A(a',v)&\geq  \normmsq{a'+w'-v}-\normmsq{a'-v}\\
			&=\normmsq{w'}-2|a'_d-v_d|w'_d\\
			&\geq \normmsq{w'}.
		\end{aligned}
	\end{equation}
	So, for any $w\in \zonea$,we have
	\begin{equation}
		A(w,v)-A(V_1^*(v),v)\geq \normmsq{w-a}-2(\frac{2^p\varkappa}{\sml^2})^{\frac{1}{p-1}}.
	\end{equation}
	
	When $v\in\zoneb$, the analysis is the same, so all in all we finished the proof.

\end{proof}
In the following, we will adopt the notations of constants from Section \ref{sec:LTA1}.
With the above lemma and the fact that
\begin{alignat}{2}
A_{1,r}&:=\sup_{w\in B_r(V^*_1(v))}A(w,v)-A(V^*_1(v),v)\leq \Celeven(1+\normmsq{v-V^*(v)})r,\quad r\in (0,1),\label{eq:1185}\\
A_{2,r}&:=\sup_{w\in B_r(V^*_2(v))}A(w,v)-A(V^*_2(v),v)\leq \Celeven(1+\normmsq{v-V^*(v)})r,\quad r\in (0,1),\label{eq:1186}
\end{alignat}
where $\Celeven=\Celeven(L,\varkappa,\operatorname{diam}(\cV))$ by Assumption~\ref{asp:22}, then using Lemma~\ref{lem:9090}, we have
\begin{equation}\label{ieq:198}
	\begin{aligned}
		&\normmsq{v_{\alpha}(\rho_t^{\kappa,R},\cP(v))-V^*(v)}\\&\leq \left(2q^{1/2}+2\Celeven r^{1/2}(1+\normm{v-V^*(v)})+\mathrm{diam}(\cV)E^{\kappa,R}_{t,r,p,q,\sml}(\alpha,\varkappa)\right.\\
  &\quad\left.+E^{\kappa,R}_{t,r,p,q,\sml}(\alpha,\varkappa)\int\normm{v-V^*(v)}d\rho_t^{\kappa,R}+\mathrm{diam}(\cV)\right)^2\\
  &\leq  \Ctwelve\left[r\normmsq{v-V^*(v)}+r+(E^{\kappa,R}_{t,r,p,q,\sml}(\alpha,\varkappa))^2\int \normmsq{v-V^*(v)}d\rho_t^{\kappa,R}(v)\right.\\
  &\quad\left.+(E^{\kappa,R}_{t,r,p,q,\sml}(\alpha,\varkappa))^2+1\right]+8q,
\end{aligned}
\end{equation}
where $E^{\kappa,R}_{t,r,p,q,\sml}(\alpha,\varkappa):=\frac{e^{-\alpha (q-2(2^p\varkappa/\sml^2)^{1/(p-1)})}}{\inf_{w\in\cV}\int_{B_r(w)}\phi_r^{\tau}(u-w)\rho_t^{\kappa,R}(u)}$ and $\Ctwelve$ only depends $L,\operatorname{diam}(\cV)$ and $\varkappa$.
So with Lemma \ref{lem:6666}, we have 
\begin{equation}\label{ieq:144}
	\begin{aligned}
		&\frac{d}{dt}\int\normmsq{v-V^*(v)}d\rho_t^{\kappa,R}(v)\\
		&\leq-(\lambda-2\sigma^2d)\int\normmsq{v-V^*(v)}d\rho_t^{\kappa,R}(v)+(\lambda+2\sigma^2d)\int(\normm{V^*(v)-v_{\alpha}(\rho_t,v)}+\kappa)^2d\rho^{\kappa,R}_t(v)\\
		&\leq -(\lambda-2\sigma^2d)\int\normmsq{v-V^*(v)}d\rho_t^{\kappa,R}(v)+(\lambda+2\sigma^2d)\Ctwelve\left[r\int\normmsq{v-V^*(v)}d\rho_t^{\kappa,R}+r\right.\\
		&\quad\left.+(E^{\kappa,R}_{t,r,p,q,\sml}(\alpha,\varkappa))^2\int \normmsq{v-V^*(v)}d\rho_t^{\kappa,R}(v)+(E^{\kappa,R}_{t,r,p,q,\sml}(\alpha,\varkappa))^2+1+\kappa^2\right]+8(\lambda+\sigma^2d)q\\
		&\leq -\left[(\lambda-2\sigma^2d)-(\lambda+2\sigma^2d)\Ctwelve r\right]\int\normmsq{v-V^*(v)}d\rho_t^{\kappa,R}(v)+(\lambda+2\sigma^2d)\Ctwelve\left[r\right.\\
		&\quad\left.+(E^{\kappa,R}_{t,r,p,q,\sml}(\alpha,\varkappa))^2\int \normmsq{v-V^*(v)}d\rho_t^{\kappa,R}(v)+(E^{\kappa,R}_{t,r,p,q,\sml}(\alpha,\varkappa))^2+1+\kappa^2\right] +8(\lambda+\sigma^2d)q\\
		&\leq -\Cexp\mathcal{V}^{\kappa,R}_t+\tilde{H}^{\kappa,R}_{t,r,p,q,\sml}(\alpha,\varkappa,\mathcal{V}^{\kappa,R}_t),
	\end{aligned}
\end{equation}
where $\Cexp:=(\lambda-2\sigma^2d)-(\lambda+2\sigma^2d)\Ctwelve r$~(as in Section~\ref{sec:LTA1}, we need to choose $r,\lambda,\sigma$ such that $\Cexp>0$), $\mathcal{V}^{\kappa,R}_t:=\int\normmsq{v-V^*(v)}d\rho_t^{\kappa,R}(v)$ and 
\begin{equation}
\begin{aligned}
\tilde{H}^{\kappa,R}_{t,r,p,q,\sml}(\alpha,\varkappa,\mathcal{V}^{\kappa,R}_t)&:=(\lambda+2\sigma^2d)\Ctwelve\Big[r+(E^{\kappa,R}_{t,r,p,q,\sml}(\alpha,\varkappa))^2\int \normmsq{v-V^*(v)}d\rho_t^{\kappa,R}(v)\\
&\quad+(E^{\kappa,R}_{t,r,p,q,\sml}(\alpha,\varkappa))^2+1+\kappa^2\Big]+8(\lambda+2\sigma^2d)q.
\end{aligned}
\end{equation}

With the above preparation, following the proofs of Lemma \ref{lem:44}, Lemma \ref{lem:3} and Theorem \ref{prop:22},  we can show  the following results.

\begin{lemma}\label{lem:10}	
	Under Assumption \ref{asp:11}, Assumption \ref{asp:22} and Assumption \ref{asp:44}, for any $T^*$, by choosing $\varkappa$ small enough depending on $\lambda,\sigma,d,\sml$ and $\alpha$ \revised{large enough depending on $\rho_0,\lambda,\sigma,d,L,\varkappa,T^*,$ $\operatorname{diam}(\cV),$ but independent of $R$},
	we have \begin{equation}
		\normmsq{V^*(v)-v_{\alpha}(\rho_t^{\kappa,R},\cP(v))}\leq \Cctr\normmsq{v-V^*(v)}+\Csmall,\quad\forall t\in [0,T^*]
	\end{equation}
	when $R\in (0,\infty),\kappa>0$, where $\Cctr,\Csmall$ depend on $\kappa,\operatorname{diam}(\cV)$ but are independent of $R$.
\end{lemma}

\begin{lemma}\label{lem:11} Under Assumption \ref{asp:11}, Assumption \ref{asp:22} and Assumption \ref{asp:44}, for any $T^*$, by choosing $\varkappa$ small enough depending on $\lambda,\sigma,d,\sml$ and $\alpha$ \revised{large enough depending on $\lambda,\sigma,d,L,\varkappa,T^*,$ $\operatorname{diam}(\cV),$ but independent of $R$}, we have
	\begin{equation}
		\Ep{|{\bold{V}^{\kappa,R}_t}|^4}\leq \Cd,\quad t\in [0,T^*],
	\end{equation}when $R\in (0,\infty),\kappa>0$, where $\bold{V}^{\kappa,R}_t:=\sup_{s\in[0,t]}\normm{V^{\kappa,R}_s}$ and $\Cd$ only depends on $T^*,\lambda,\sigma,d,\kappa,$
 $\operatorname{diam}(\cV),\rho_0.$
\end{lemma}
{ We are ready now to state the main existence result whose proof follows precisely the same limiting argument as in Theorem \ref{prop:22}. As in Remark \ref{rem:uniq}, the uniqueness remains an open problem and it does not easily follow from standard stability arguments.}

\begin{theorem}\label{prop:2222}Under Assumption \ref{asp:11},Assumption \ref{asp:22}, and Assumption \ref{asp:44}, for any $T^*,\kappa>0,$ by choosing $\varkappa>0$ small enough and $\alpha>0$ large enough, the SDE~\eqref{eq:n1} has a strong solution and its law $\rho^{\kappa}_t$ satisfies the related Fokker--Planck equation \eqref{eq:n3} in the weak sense.
\end{theorem}

\subsection{Large time asymptotics}

{ Because of the necessity of using the Lemma \ref{lem:9090} which is a variant of the quantitative Laplace principle, the large time asymptotics for the solution $\rho_t$ requires a specific adaptation with respect to the analysis in Section \ref{sec:LTA1}. Let us start with the following observation.
}

\begin{remark}
	Let $R\to\infty$ in Lemma \ref{lem:6666}, we also have 
	\begin{equation}
		\begin{aligned}
			\frac{d}{dt}\int\normmsq{v-V^*(v)}d\rho^{\kappa}_t(v)
			&\leq-(\lambda-2\sigma^2d)\int\normmsq{v-V^*(v)}d\rho_t^{\kappa}(v)\\
			&\quad+(\lambda+2\sigma^2d)\int\left(\normm{V^*(v)-v_{\alpha}(\rho_t^\kappa,v)}+\kappa\right)^2d\rho_t^{\kappa}(v).
		\end{aligned}
	\end{equation}
\end{remark}
In the following, we  show that the functional 
\begin{equation}
    {\cv^\kappa_t:= \int\dist{v,\cV}^2 d\rho_{t}^{\kappa}= \int\normmsq{v-V^*(v)}d\rho^{\kappa}_t(v)},
\end{equation}
is exponentially decreasing and for any error bound $\epsilon>0$, by choosing proper $\lambda,\sigma>0$, $\kappa,\varkappa>0$ small enough, and $\alpha>0$ large enough, we can achieve 
\begin{equation}
	\cv^\kappa_{T^*}=\int\normmsq{v-V^*(v)}d\rho_{T^*}^{\kappa}\leq 2\epsilon,
\end{equation}
for some $T^*$ defined later. First, we need to analyze the dynamics of $\cv^\kappa_t$.

\begin{lemma}\label{lem:12} For any $T^*,\kappa>0$,
	by choosing $\varkappa$ small enough depending on $\lambda,\sigma,d,\sml,\rho_0,T^*,$\\$\operatorname{diam}(\cV),\kappa,L$; $r$ small enough depending on $\rho_0,\kappa,\lambda,\sigma,d,L,\operatorname{diam}(\cV),T^*,\varkappa$; $\alpha$ large enough depending on $r,\varkappa,\kappa,\lambda,\sigma,d,L,T^*,\rho_0,\operatorname{diam}(\cV),$ we have 
	\begin{equation}
		\frac{d}{dt}\mathcal{V}^\kappa_t\leq -\Cexp\mathcal{V}^\kappa_t+H^{\kappa}_{t,r,p,q,\sml}(\alpha,\varkappa,\mathcal{V}^\kappa_t),\quad t\in [0,T^*],
	\end{equation} 
 where 
  $H^{\kappa}_{t,r,p,q,\sml}(\alpha,\varkappa,\mathcal{V}^\kappa_t):=\Cerr\left[r+(E^{\kappa}_{t,r,p,q,\sml}(\alpha,\varkappa))^2+(E^{\kappa}_{t,r,p,q,\sml}(\alpha,\varkappa))^2\mathcal{V}^\kappa_t+\kappa^2\right]+8(\lambda+2\sigma^2d)q,$ $\mathcal{V}_t^{\kappa}:=\int \normmsq{v-V^*(v)}d\rho_t^{\kappa}$,
   $E^{\kappa}_{t,r,p,q,\sml}(\alpha,\varkappa):=\frac{e^{-\alpha (q-2(2^p\varkappa/\sml^2)^{1/(p-1)})}}{\inf_{w\in\cV}\int_{B_r(w)}\phi_r^{\tau}(u-w)\rho_t^{\kappa}(u)}$, and $\Cexp,\Cerr>0$ only depend on $\lambda,\sigma,d,L,\varkappa,$ $\operatorname{diam}(\cV).$
\end{lemma}
\begin{proof}
	We separate $\mathbb{R}^d$ into three regions: $B_{\R}(0)^c, B_{\R}(0)\cap(\partial\zonea+B_{\rr}(0)),B_{\R}(0)\cap(\partial\zonea+B_{\rr}(0))^c$,  where $\R,\rr>0$ will be defined later.
	
	On $B_{\R}(0)^c$, by \eqref{ieq:198}, we have
	\begin{equation}
		\begin{aligned}
			&\int_{B_{\R}(0)^c}(\normm{V^*(v)-v_{\alpha}(\rho_t^{\kappa},v)}+\kappa)^2d\rho_t^{\kappa}(v)\\
			&\leq  \Ctwelve r\int_{B_{\R}(0)^c} \normmsq{v-V^*(v)}d\rho_t^{\kappa}(v)+
			\Ctwelve\int_{B_{\R}(0)^c} \left ( 1+\kappa^2 \right )d\rho_t^{\kappa}\\
			&\quad+\Ctwelve\int_{B_{\R}(0)^c}\left[r+E^2_{t,r,p,q,\sml}(\alpha,\varkappa)+E^2_{t,r,p,q,\sml}(\alpha,\varkappa)\mathcal{V}_t^{\kappa}\right]d\rho_t^{\kappa}+8\int_{B_{\R}(0)^c}qd\rho_t^\kappa,
		\end{aligned}
	\end{equation}
 here $\Ctwelve$ only depends on $L,\varkappa,\operatorname{diam}(\cV)$. By choosing $\R\geq 20\sqrt{\Ctwelve(1+\kappa^2)}$, we have
	$\normmsq{v-V^*(v)}$\\$\geq 100\Ctwelve\left(1+\kappa^2\right)$ for any $v \in B_{\R}(0)^c$, so
	\begin{equation}
		\Ctwelve\int_{B_{\R}(0)^c}\left (1+\kappa^2 \right )d\rho_t^{\kappa}\leq \frac{1}{100}\int_{B_{\R}(0)^c}\normmsq{v-V^*(v)}d\rho_t^{\kappa}(v),
	\end{equation}
	and  
	\begin{equation}
		\begin{aligned}
			&\int_{B_{\R}(0)^c}(\normm{V^*(v)-v_{\alpha}(\rho_t^{\kappa},v)}+\kappa)^2d\rho_t^{\kappa}(v)\\
			&\leq \left(\Ctwelve r+\frac{1}{100}\right )\int_{B_{\R}(0)^c} \normmsq{v-V^*(v)}d\rho_t^{\kappa}(v)\\
			&\quad+\Ctwelve\int_{B_{\R}(0)^c}\left[r+E^2_{t,r,p,q,\sml}(\alpha,\varkappa)+E^2_{t,r,p,q,\sml}(\alpha,\varkappa)\mathcal{V}_t^{\kappa}\right]d\rho_t^{\kappa}+8\int_{B_{\R}(0)^c}qd\rho_t^\kappa,
		\end{aligned}
	\end{equation}
	
	On $B_{\R}(0)\cap(\partial\zonea+B_{\rr}(0))$, we also have 
	\begin{equation}
		\begin{aligned}
			&\int_{B_{\R}(0)\cap(\partial\zonea+B_{\rr}(0))}(\normm{V^*(v)-v_{\alpha}(\rho_t^{\kappa},v)}+\kappa)^2d\rho_t^{\kappa}(v)\\
			&\leq\Ctwelve r\int_{B_{\R}(0)\cap(\partial\zonea+B_{\rr}(0))} \normmsq{v-V^*(v)}d\rho_t^{\kappa}(v)+\Ctwelve\int_{B_{\R}(0)\cap(\partial\zonea+B_{\rr}(0))}\left  ( 1+\kappa^2 \right )d\rho_t^{\kappa}\\
			&\quad+\Ctwelve\int_{B_{\R}(0)\cap(\partial\zonea+B_{\rr}(0))}\left[r+E^2_{t,r,p,q,\sml}(\alpha,\varkappa)+E^2_{t,r,p,q,\sml}(\alpha,\varkappa)\mathcal{V}_t^{\kappa}\right]d\rho_t^{\kappa}\\
   &\quad+8\int_{B_{\R}(0)\cap(\partial\zonea+B_{\rr}(0))}qd\rho_t^\kappa,
		\end{aligned}
	\end{equation}we will show in Lemma \ref{lem:13} in the Appendix that for any $T^*$, we can choose $\rr$ small enough depends $\rho_0,\lambda,\sigma,d,\kappa,\tilde{R},\operatorname{diam}(\cV),T^*$ such that 
	\begin{equation}
		\int_{B_{\R}(0)\cap(\partial\zonea+B_{\rr}(0))} \left ( 1+\kappa^2 \right )d\rho_t^{\kappa}\leq \kappa^2,\quad \text{for all } t\in [0,T^*],
	\end{equation}
	then we have
	\begin{equation}
		\begin{aligned}
			&\int_{B_{\R}(0)\cap(\partial\zonea+B_{\rr}(0))}(\normm{V^*(v)-v_{\alpha}(\rho_t^{\kappa},v)}+\kappa)^2d\rho_t^{\kappa}(v)\\
			&\leq \Ctwelve r\int_{B_{\R}(0)\cap(\partial\zonea+B_{\rr}(0))} \normmsq{v-V^*(v)}d\rho_t^{\kappa}(v)+\Ctwelve\kappa^2\\
			&\quad+\Ctwelve\int_{B_{\R}(0)\cap(\partial\zonea+B_{\rr}(0))}\left[r+E^2_{t,r,p,q,\sml}(\alpha,\varkappa)+E^2_{t,r,p,q,\sml}(\alpha,\varkappa)\mathcal{V}_t^\kappa\right]d\rho_t^{\kappa}\\
   &\quad +8\int_{B_{\R}(0)\cap(\partial\zonea+B_{\rr}(0))}qd\rho_t^\kappa,
		\end{aligned}
	\end{equation}
	 this $\rr$ is our choice for $\rr$.
	
	On $B_{\R}(0)\cap(\partial\zonea+B_{\rr}(0))^c$: for any $v\in B_{\R}(0)\cap(\partial\zonea+B_{\rr}(0))^c$, we have
\begin{equation}
    A(w,v)-A(V^*_1(v),v)\geq \normmsq{w-V^*_1(v)}-2(\frac{2^p\varkappa}{\sml^2})^{\frac{1}{p-1}}
\end{equation}
and 
\begin{equation}
     \min_{w\in\zonea}A(w,v)-A(V^*_1(v),v)\leq 0,
\end{equation}
 so $\min_{w\in\zonea}A(w,v)\in[A(V^*_1(v),v)-2(2^p\varkappa/\sml^2)^{1/(p-1)},A(V^*_1(v),v)]$, similarly, we have \\$\min_{w\in\zoneb}A(w,v)\in[A(V^*_2(v),v)-2(2^p\varkappa/\sml^2)^{1/(p-1)},A(V^*_2(v),v)]$. And so 
	\begin{equation}
		\begin{aligned}
		    &|{\min_{w\in\zonea}A(w,v)-\min_{w\in\zoneb}A(w,v)}|\\
      &\geq|{\operatorname{dist}(v,\cV_1)^2-\operatorname{dist}(v,\cV_2)^2}|-2(\frac{2^p\varkappa}{\sml^2})^{\frac{1}{p-1}}\\
      &\geq {C_{\mathrm{gap}}}>0,
		\end{aligned}
	\end{equation}
	for some constant  ${C_{\mathrm{gap}}}>0$ that only depends on $\operatorname{diam}(\cV),\rr,\R,\varkappa$, if we 
 choose $\varkappa$ that depends on $\tilde{R},\sml,\rr$
 such that $2(2^p\varkappa/\sml^2)^{1/(p-1)}\ll\rr$. Then by Lemma \ref{lem:qlp}, we can derive
	\begin{equation}
		\normmsq{V^*(v)-v_{\alpha}(\rho_t^{\kappa},v)}\leq 2q+\Ctwentyone r+\Ctwentyone(E^{\kappa}_{t,r,p,q,\sml}(\alpha,\varkappa))^2\mathcal{V}_t^{\kappa},
	\end{equation}
 for $r$ small depending on $\rr$ and $q$ small with $q>2(2^p\varkappa/\sml^2)^{1/(p-1)}$, we also need to choose $q,r$ such that they satisfy $q-2(2^p\varkappa/\sml^2)^{1/(p-1)}+\sup_{v\in B_{\R}(0)\cap(\partial\zonea+B_{\rr}(0))^c}\sup_{w\in B_r(V^*(v))}A(w,v)-A(V^*(v),v)\leq c_{\rm gap}$. In the above, $\Ctwentyone$ only depends on $L,\varkappa,\tilde{R},\operatorname{diam}(\cV)$. So 
	\begin{equation}
	\begin{aligned}
		&\int_{B_{\R}(0)\cap(\partial\zonea+B_{\rr}(0))^c}(\normm{V^*(v)-v_{\alpha}(\rho_t^{\kappa},v)}+\kappa)^2d\rho_t^{\kappa}(v)\\
		&\leq 2\Ctwentyone\int_{B_{\R}(0)\cap(\partial\zonea+B_{\rr}(0))^c}\left[r+(E^{\kappa}_{t,r,p,q,\sml}(\alpha,\varkappa))^2\mathcal{V}^\kappa_t+\kappa^2\right]d\rho_t^{\kappa}\\
  &\quad+8\int_{B_{\R}(0)\cap(\partial\zonea+B_{\rr}(0))^c}qd\rho_t^\kappa.
	\end{aligned}
	\end{equation}
	
	Combine all three parts, we have
	\begin{equation}
		\begin{aligned}
			&\frac{d}{dt}\int\normmsq{v-V^*(v)}d\rho_t^{\kappa}(v)\\
			&\leq-(\lambda-2\sigma^2d)\int\normmsq{v-V^*(v)}d\rho_t^{\kappa}(v)+(\lambda+2\sigma^2d)\int(\normm{V^*(v)-v_{\alpha}(\rho_t,v)}+\kappa)^2d\rho^{\kappa}_t(v)\\
			&\leq -\underbrace{\left[(\lambda-2\sigma^2d)-(\lambda+2\sigma^2d)(\Ctwelve r+\frac{1}{100})\right]}_{\Cexp}\int\normmsq{v-V^*(v)}d\rho_t^{\kappa}\\
			&\quad+\underbrace{4(\lambda+2\sigma^2d)\max\{\Ctwelve,\Ctwentyone\}}_{\Cerr}\left[r+(E^{\kappa}_{t,r,p,q,\sml}(\alpha,\varkappa))^2+(E^{\kappa}_{t,r,p,q,\sml}(\alpha,\varkappa))^2\mathcal{V}_t^{\kappa}+\kappa^2\right]\\
   &\quad +8(\lambda+2\sigma^2d)q,
		\end{aligned}
	\end{equation}
	where we need choose proper $\sigma,\lambda,r$ such that $\Cexp>0$. 
\end{proof}

In the end, we prove the main convergence result of this section.

\begin{theorem}\label{prop:66}Under Assumption \ref{asp:11},Assumption \ref{asp:22},Assumption \ref{asp:44}, for any error bound $\epsilon>0$, choose $\kappa,\varkappa$ small enough and denote 
	$T^*:=\frac{2}{\Cexp}\log(\frac{\cv_0^{\kappa}\vee (2\epsilon)}{2\epsilon})$, where constants $\Cexp,\Cerr$ are from Lemma \ref{lem:12}; Then for \revised{the} SDE \eqref{eq:n1} with $\rho_0^{\kappa}=\rho_0\in\mathcal{P}_4(\mathbb{R}^d)$, we can choose $\alpha$ big enough and with these $\varkappa,\alpha$, we have $\mathcal{V}^{\kappa}_{T^*}\leq2\epsilon$, and on $[0,T^*]$, we have $\mathcal{V}^{\kappa}_{t}\leq (\Cexp\epsilon t+\mathcal{V}^{\kappa}_0)e^{-\Cexp t}$.
\end{theorem}

\begin{proof}
	By Lemma \ref{lem:10} with $R\to\infty$, we have 
	\begin{equation}
		\sup_{t\in[0,T^*]}\sup_{v\in \cV+B_r(0)}\normm{V^*(v)-v_{\alpha}(\rho_t^{\kappa},v)}<\infty,
	\end{equation}
	and so by Lemma \ref{lem:5l}, we have
	\begin{equation}
		\inf_{t\in [0,T^*]}{\inf_{w\in\cV}\int_{B_r(w)}\phi_r^{\tau}(u-w)\rho_t^{\kappa}(u)}>0.
	\end{equation}
	By Lemma \ref{lem:11} with $R\to\infty$, we have
	\begin{equation}
		\sup_{t\in [0,T^*]}\mathcal{V}^{\kappa}_t<\infty.
	\end{equation}
	By Lemma \ref{lem:12} and the choice of $\kappa$~(for example $\kappa\in (0,\sqrt{{\Cexp\epsilon}/{(2\Cerr)}})$), we have
	\begin{equation}
		\frac{d}{dt}\mathcal{V}^{\kappa}_t\leq -\Cexp\mathcal{V}^{\kappa}_t+H^{\kappa}_{t,r,p,q,\sml}(\alpha,\varkappa,\mathcal{V}^{\kappa}_t).
	\end{equation}
	Due to the definition, we have $\mathcal{V}^{\kappa}_t=\int\normmsq{v-V^*(v)}d\rho^{\kappa}_t(v),$\\ $H^{\kappa}_{t,r,p,q,\sml}(\alpha,\varkappa,\mathcal{V}^{\kappa}_t)\leq \Cerr\left[r+(E^{\kappa}_{t,r,p,q,\sml}(\alpha,\varkappa))^2+(E^{\kappa}_{t,r,p,q,\sml}(\alpha,\varkappa))^2\mathcal{V}^{\kappa}_t\right]+\frac{\Cexp\epsilon}{2}+8(\lambda+2\sigma^2d)q$ and $E^{\kappa}_{t,r,p,q,\sml}(\alpha,\varkappa)=\frac{e^{-\alpha (q-2(2^p\varkappa/\sml^2)^{1/(p-1)})}}{\inf_{w\in\cV}\int_{B_r(w)}\phi_r^{\tau}(u-w)\rho_t^{\kappa}(u)}$,  so we can choose $\varkappa,q,r$ small enough and $\alpha$ large enough such that
	\begin{equation}
		H_{t,r,p,q}^{\kappa}(\alpha,\varkappa,\mathcal{V}^{\kappa}_t)\leq {\Cexp\epsilon},
	\end{equation}
	then on $[0,T^*]$, we have
	\begin{equation}
		\frac{d}{dt}\mathcal{V}^{\kappa}_t\leq-\Cexp\mathcal{V}^{\kappa}_t+\Cexp\epsilon
	\end{equation}
then by {Gr\"onwall's inequality}, we have
  \begin{equation}
\mathcal{V}^{\kappa}_{t}\leq (\Cexp\epsilon t+\mathcal{V}^{\kappa}_0)e^{-\Cexp t} \text{  for all } t \in [0,T^*],
  \end{equation}
and
	\begin{equation}
		\mathcal{V}^{\kappa}_{T^*}\leq 2\epsilon.
	\end{equation}
\end{proof}

\section{Conclusions and outlook}\label{sec:4}
In this paper \revised{we designed a new variant of CBO type dynamic that can find} the global minimizers of a potentially \revised{non-convex}
and \revised{non-smooth} objective function. {With Theorem \ref{thm:main0} we establish the existence of a solution to the related SDE and Fokker--Planck equation and we prove its contraction to global minimizers of the objective function.

Theorem \ref{thm:main0}, together with standard convergence of discrete time numerical approximations \cite{platen1999introduction} and a quantitative mean-field limit such as \eqref{eq:qmfl} \cite{B1-fornasier2021global,fornasier2023consensus1,gerber2023propagation} would ensure the convergence in probability of the fully time discrete and finite particle scheme
\begin{alignat}{3} \label{eq:alg}
	{V_{k+1}^i - V_{k}^i} &= - \Delta t\lambda\left( V_{k}^i - \conspoint{\empmeasure{k},V^i_{k}}\right)\nonumber\\
	&\quad+ \sqrt{\Delta t}\sigma \left(\N{ V_{k}^i-\conspoint{\empmeasure{k},V^i_{k}}}_2+\kappa\right) B_{k}^i, \\
	{V_0^i} \sim \rho_0 &\quad \text{for all } i =1,\ldots,N, \nonumber
\end{alignat}
in the same spirit of \eqref{B1-ownresult} and \eqref{B1-ownresult2} as done in details in \cite{B1-fornasier2021global,fornasier2023consensus1}. The only difference is that, instead of $v^*$ one considers $V^*(v)$ in the error estimates. We sketch the  adapted estimates of convergence in the Appendix and we outline the open problems that remain for their full validity, opting not to pursue a detailed analysis here. In particular, while the present paper is fully dedicated to the theoretical analysis of \eqref{eq:n1}-\eqref{eq:n3}, an extensive numerical analysis and experimentation of \eqref{eq:alg} will be subject of a subsequent work.}

\section*{Appendix}
\paragraph{The quantitative Laplace principle and variations.}
The next quantitative Laplace principle is adapted from \cite[Proposition 21]{B1-fornasier2021global} and \cite[Proposition E.2]{riedl2023gradient}.


\revised{\begin{lemma}\label{lem:qlp}
	Let $\beta>0,g \in \mathcal{C}\left(\mathbb{R}^d\right)$, $v^*\in\mathbb{R}^d$ and  $\rho\in\mathcal{P}(\mathbb{R}^d)$ with $v^*\in \mathrm{supp}(\rho)$, denote $g^*:=g(v^*)$.  
 For any $r>0$ define $g_r:=\sup _{v \in B_r\left(v^*\right)} g-g^*$. Then, for fixed $\alpha>0$: if 
 there exist {$\eta,\nu>0$ such that $g$ satisfies} $g(v)-g^*\geq (\eta\normm{v-v^*})^{\frac{1}{\nu}}-\beta$, then for any $r>0, q>\beta$, we have
	\begin{equation}\label{eq:141141}
		\left\|v_{\alpha}(\rho)-v^*\right\|_2 \leq \frac{\left(q+g_r\right)^\nu}{\eta}+\frac{\exp (-\alpha (q-\beta))}{\rho\left(B_r\left(v^*\right)\right)} \int\left\|v-v^*\right\|_2 d \rho(v),
	\end{equation}
	where 
	\begin{equation}
	    v_{\alpha}(\rho):=\frac{\int we^{-\alpha g(w)}d\rho(w)}{\int e^{-\alpha g(w)}d\rho(w)};
	\end{equation}
	if there exist $g_{\infty}, R_0, \eta>0$ and $\nu \in(0, \infty)$ such that
	\begin{equation}\label{eq:1553}
	\begin{aligned}
		g(v)-g^*&\geq (\eta\normm{v-v^*})^{\frac{1}{\nu}}-\beta \quad \text { for all } v \in B_{R_0}\left(v^*\right), \\
		g(v)-g^*& >g_{\infty} \quad \text { for all } v \in\left(B_{R_0}\left(v^*\right)\right)^c,
	\end{aligned}    
	\end{equation}
	then for any $r \in\left(0, R_0\right]$ and $q>\beta$ such that $q-\beta+g_r \leq g_{\infty}$, we have \eqref{eq:141141}.
\end{lemma}
\begin{proof}
     We have
    \begin{equation}\label{ieq:1776}
        \begin{aligned}
        \normm{v_{\alpha}(\rho)-v^*}&\leq \int_{\mathbb{R}^d}\normm{w-v^*}\frac{e^{-\alpha g(w)}}{\int_{\mathbb{R}^d}e^{-\alpha g(w)}d\rho(w)}d\rho(w)\\
        &\leq \int_{B_{\tilde{r}}(v^*)}\normm{w-v^*}\frac{e^{-\alpha g(w)}}{\int_{\mathbb{R}^d} e^{-\alpha g(w)}d\rho(w)}d\rho(w)\\
        &\quad  +\int_{(B_{\tilde{r}}(v^*))^c}\normm{w-v^*}\frac{e^{-\alpha g(w)}}{\int_{\mathbb{R}^d} e^{-\alpha g(w)}d\rho(w)}d\rho(w)\\
        &\leq \tilde{r}+\frac{e^{-\alpha (\inf_{w\in (B_{\tilde{r}}(v^*))^c}g(w)-g_r-g^*)}}{\rho(B_r(v^*))}\int_{\mathbb{R}^d}\normm{w-v^*}d\rho(w),
        \end{aligned}
    \end{equation}
    where $\tilde{r}\geq r>0$ will be determined later and in the third inequality, we used $\int_{\mathbb{R}^d}e^{-\alpha g(w)}d\rho(w)\geq e^{-\alpha(g_r+g^*)}\rho(B_r(v^*))$. Now for any $q>\beta$, we choose $\tilde{r}=(q+g_r)^{\nu}/\eta$, then by the inverse growth condition, we have
    \begin{equation}\label{ieq:1766}
        \tilde{r}=\frac{(q+g_r)^{\nu}}{\eta}\geq \frac{(\beta+g_r)^{\nu}}{\eta}\geq \frac{\left((\eta r)^{\frac{1}{\nu}}-\beta+\beta\right)^{\nu}}{\eta}=r,
    \end{equation}
    and 
    \begin{equation}
        \inf_{w\in (B_{\tilde{r}}(v^*))^c}g(w)-g_r-g^*=\inf_{w\in (B_{\tilde{r}}(v^*))^c}g(w)-g^*-g_r\geq(\eta\tilde{r})^{\frac{1}{\nu}}-\beta-g_r= q-\beta.
    \end{equation}
Inserting these into \eqref{ieq:1776}, we finished the proof for the first half of the lemma.

For the second half, we also have \eqref{ieq:1776} and choose $\tilde{r}=(q+g_r)^{\nu}/\eta$, then 
\eqref{ieq:1766} also hold for any $r\in (0,R_0]$ and 
\begin{equation}
        \inf_{w\in (B_{\tilde{r}}(v^*))^c}g(w)-g_r-g^*={\inf_{w\in (B_{\tilde{r}}(v^*))^c}g(w)-g^*}-g_r\geq\min\{(\eta\tilde{r})^{\frac{1}{\nu}}-\beta,g_{\infty}\}-g_r= q-\beta,
    \end{equation}
    which holds true since we choose $q-\beta+g_r\leq g_{\infty}$.
Finally, inserting these into \eqref{ieq:1776}, we also have the second half of the lemma.
\end{proof}
}

{ The above variant of quantitative Laplace principle can be adapted to the assumptions of Section \ref{sec:2} and yields the following auxiliary result.}

\revised{\begin{lemma}\label{lem:9090}
    Suppose $\beta>0$ and $\mathbb{R}^d$ is separated into two regions $\zonea,\zoneb$. Let $g\in\mathcal{C}(\mathbb{R}^d)$ and satisfy the following properties:
    \begin{alignat}{2}
        g(v)-g(v_1^*)&\geq (\eta\normm{v-v_1^*})^{\frac{1}{\nu}}-\beta\quad \text{for some $v_1^*\in\zonea$ and $\forall v\in\zonea$}\label{eq:153153},\\
         g(v)-g(v_2^*)&\geq (\eta\normm{v-v_2^*})^{\frac{1}{\nu}}-\beta\quad \text{for some $v_2^*\in\zoneb$ and $\forall v\in\zoneb$}.
    \end{alignat}
    For any $r>0$ define $g_r:=\max\{\sup_{v\in B_r(v_1^*)}g-g(v_1^*),\sup_{v\in B_r(v_2^*)}g-g(v_2^*)\}$.
    Then for $\rho\in\mathcal{P}(\mathbb{R}^d)$ with $v_1^*\in\mathrm{supp}(\rho)$ and $v_2^*\in\mathrm{supp}(\rho)$ and $\alpha>0$, we have 
\begin{equation}
    \begin{aligned}
        \normm{v_{\alpha}(\rho)-v^*_i}&\leq \normm{v^*_1-v_2^*}+\frac{2(q+g_{r})^{\nu}}{\eta}+\frac{e^{-\alpha (q-\beta)}}{{\inf_{w\in\{v^*_1,v^*_2\}}\int_{B_r(w)}\phi_r^{\tau}(u-w)\rho(u)}}\int\normm{w-v_2^*}d\rho(w)\\
    &\quad +\normm{v^*_1-v_2^*}\frac{e^{-\alpha (q-\beta)}}{{\inf_{w\in\{v^*_1,v^*_2\}}\int_{B_r(w)}\phi_r^{\tau}(u-w)\rho(u)}},\quad\text{for $i=1,2$},
    \end{aligned}
\end{equation}
  for any $r>0,q>\beta$ such that $ \partial B_{q+g_r}(v_1^*)\cap \overline{\zonea}\neq\emptyset$ and $\partial B_{q+g_r}(v_2^*)\cap \overline{\zoneb}\neq\emptyset$ ,  where $v_{\alpha}(\rho)$ is the same as in Lemma~\ref{lem:qlp}.
\end{lemma}
\begin{proof}
    We only prove the case $i=2$, the case $i=1$ is similar. In this case, we have
\begin{equation}
    \begin{aligned}
        \normm{v_{\alpha}(\rho)-v_2^{*}}&\leq \int_{\mathbb{R}^d}\normm{w-v_2^{*}}\frac{e^{-\alpha g(w)}}{\int_{\mathbb{R}^d}e^{-\alpha g(w)}d\rho(w)}d\rho(w)\\
        &= \underbrace{\int_{\zonea}\normm{w-v^*_2}\frac{e^{-\alpha g(w)}}{\int_{\mathbb{R}^d}e^{-\alpha g(w)}d\rho(w)}d\rho(w)}_{I}\\
        &\quad+\underbrace{\int_{\zoneb}\normm{w-v_2^*}\frac{e^{-\alpha g(w)}}{\int_{\mathbb{R}^d}e^{-\alpha g(w)}d\rho(w)}d\rho(w)}_{II}.
    \end{aligned}
\end{equation}
For term $I$, we have
\begin{equation}
    \begin{aligned}
        I&=\int_{\zonea}\normm{w-v_2^*}\frac{e^{-\alpha g(w)}}{\int_{\mathbb{R}^d}e^{-\alpha g(w)}d\rho(w)}d\rho(w)\\
        &\leq \normm{v^*_1-v_2^*}+\int_{\zonea}\normm{w-v_1^*}\frac{e^{-\alpha g(w)}}{\int_{\mathbb{R}^d}e^{-\alpha g(w)}d\rho(w)}d\rho(w)\\
        &\leq \normm{v^*_1-v_2^*}+\int_{\zonea\cap B_{\tilde{r}}(v_1^*)}\normm{w-v_1^*}\frac{e^{-\alpha g(w)}}{\int_{\mathbb{R}^d}e^{-\alpha g(w)}d\rho(w)}d\rho(w)\\
         &\quad+\int_{\zonea\backslash B_{\tilde{r}}(v_1^*)}\normm{w-v_1^*}\frac{e^{-\alpha g(w)}}{\int_{\mathbb{R}^d}e^{-\alpha g(w)}d\rho(w)}d\rho(w)\\  
         &\leq \normm{v^*_1-v_2^*}+\tilde{r}+\frac{e^{\alpha (g_{1,r}+g(v_1^*))}}{\rho(B_r(v_1^*))}\int_{\zonea\backslash B_{\tilde{r}}(v_1^*)}\normm{w-v_1^*}e^{-\alpha g(w)}d\rho(w)\\ 
         &\leq \normm{v^*_1-v_2^*}+\tilde{r}+\frac{e^{-\alpha \left(\inf_{w\in \zonea\backslash B_{\tilde{r}}(v_1^*)}g(w)-g_{1,r}-g(v_1^*)\right)}}{\rho(B_r(v_1^*))}\int_{\zonea\backslash B_{\tilde{r}}(v_1^*)}\normm{w-v_1^*}d\rho(w),
    \end{aligned}
\end{equation}
where $g_{1,r}:=\sup_{w\in B_r(v_1^*)}g(w)-g(v_1^*)$, $0<r\leq \tilde{r}$ will be determined later and in the third inequality, we used $\int_{\mathbb{R}^d}e^{-\alpha g(w)}d\rho(w)\geq e^{-\alpha (g_{1,r}+g(v_1^*))}\rho(B_r(v_1^*))$. Now for any $q>\beta$, we choose $\tilde{r}=(q+g_{1,r})^{\nu}/\eta$, then we have 
\begin{equation}
   \begin{aligned}
        \tilde{r}&=\frac{(q+g_{1,r})^{\nu}}{\eta}\geq \frac{(g_{1,r}+\beta)^{\nu}}{\eta}=\frac{\left(\sup_{w\in B_r(v_1^*)}g(w)-g(v_1^*)+\beta\right)^{\nu}}{\eta}\\
        &\geq \sup_{w\in B_r(v_1^*)}\normm{w-v_1^*}=r,
   \end{aligned}
\end{equation}
 the last inequality is due to \eqref{eq:153153}. Further using \eqref{eq:153153} and the definition of $\tilde{r}$, we have
\begin{equation}
    \inf_{w\in \zonea\backslash B_{\tilde{r}}(v_1^*)}g(w)-g_{1,r}-g(v_1^*)\geq (\eta\tilde{r})^{\frac{1}{\nu}}-\beta-g_{1,r}=q-\beta,
\end{equation}
so, by combining all of the above, we have
\begin{equation}
    \begin{aligned}
        I&\leq \normm{v^*_1-v_2^*}+\frac{(q+g_{1,r})^{\nu}}{\eta}+\frac{e^{-\alpha (q-\beta)}}{\rho(B_r(v_1^*))}\int_{\zonea\backslash B_{\tilde{r}}(v_1^*)}\normm{w-v_1^*}d\rho(w)\\
    &\leq \normm{v^*_1-v_2^*}+\frac{(q+g_{1,r})^{\nu}}{\eta}+\frac{e^{-\alpha (q-\beta)}}{\rho(B_r(v_1^*))}\int_{\zonea}\normm{w-v_2^*}d\rho(w)\\
    &\quad +\normm{v^*_1-v_2^*}\frac{e^{-\alpha (q-\beta)}}{\rho(B_r(v_1^*))},
    \end{aligned}
\end{equation}
for any $r,q>\beta$.

For term $II$, we have
\begin{equation}
   \begin{aligned}
        II&=\int_{\zoneb}\normm{w-v_2^*}\frac{e^{-\alpha g(w)}}{\int_{\mathbb{R}^d}e^{-\alpha g(w)}d\rho(w)}d\rho(w)\\
        &\leq \int_{\zoneb\cap B_{\tilde{r}}(v^*_2)}\normm{w-v_2^*}\frac{e^{-\alpha g(w)}}{\int_{\mathbb{R}^d}e^{-\alpha g(w)}d\rho(w)}d\rho(w)\\
         &\quad+\int_{\zoneb\backslash B_{\tilde{r}}(v^*_2)}\normm{w-v_2^*}\frac{e^{-\alpha g(w)}}{\int_{\mathbb{R}^d}e^{-\alpha g(w)}d\rho(w)}d\rho(w),
   \end{aligned} 
\end{equation}
by a similar analysis as in the term $I$, we have
\begin{equation}
    II\leq \frac{(q+g_{2,r})^{\nu}}{\eta}+\frac{e^{-\alpha (q-\beta)}}{\rho(B_r(v^*_2))}\int_{\zoneb}\normm{w-v_2^*}d\rho(w),
\end{equation}
for any $r>0,q>\beta$, where $g_{2,r}:=\sup_{w\in B_r(v^*_2)}g(w)-g(v^*_2)$.

Combine term $I$ and $II$, we have
\begin{equation}
    \begin{aligned}
        \normm{v_{\alpha}(\rho)-v^*_2}&\leq \normm{v^*_1-v_2^*}+\frac{(q+g_{1,r})^{\nu}}{\eta}+\frac{(q+g_{2,r})^{\nu}}{\eta}\\
        &\quad+\frac{e^{-\alpha (q-\beta)}}{\min\{\rho(B_r(v_1^*)),\rho(B_r(v^*_2))\}}\int\normm{w-v_2^*}d\rho(w)\\
    &\quad +\normm{v^*_1-v_2^*}\frac{e^{-\alpha (q-\beta)}}{\min\{\rho(B_r(v_1^*)),\rho(B_r(v^*_2))\}},
    \end{aligned}
\end{equation}
finally, using 
\begin{alignat}{1}
    \min\{\rho(B_r(v^*_1),\rho(B_r(v^*_2)\}&\geq {\inf_{w\in\{v^*_1,v^*_2\}}\int_{B_r(w)}\phi_r^{\tau}(u-w)\revised{d}\rho(u)},
\end{alignat}
we proved the lemma.
\end{proof}
}

\paragraph{On the continuity of the law on sets.}
The next lemma shows that $\rho_t^{\kappa}(B_R(0)\cap(\partial \zonea+B_r(0)))$ is continuous in terms of $t$, then by a further analysis, we can find $r$ small enough such that $\rho_t^{\kappa}(B_R(0)\cap(\partial \zonea+B_r(0)))\leq\epsilon'$, for any $\epsilon'>0$.
\begin{lemma}\label{lem:13}
	For any $T^*,\epsilon'>0$, we can choose $r$ small enough, \revised{depending} on $\operatorname{diam}(\cV),\lambda,\sigma,d,$\\
 $R,T^*,\epsilon',\rho_0$, $\varkappa$ small enough depending on $\lambda,\sigma,d,\sml,$
 $\alpha$ large enough depending on $\lambda,\sigma,d,T^*,L,\kappa,\varkappa,$\\$\operatorname{diam}(\cV),\epsilon'$ such that $\rho_t^{\kappa}(B_R(0)\cap(\partial \zonea+B_r(0)))\leq\epsilon'$, for any $t\in [0,T^*]$. 
\end{lemma}
\begin{proof}

	We can find a { smooth} function $\phi_{r}(v)$~(see Remark \ref{rmk:8}) which will vanish outside of $(\partial \zonea+B_{2r}(0))\cap B_{2R}(0)$ and has range $[C_{\mathrm{s}},C_{\mathrm{h}}]$~($C_{\mathrm{h}}\geq C_{\mathrm{s}}>0$) inside of $(\partial \zonea+B_{r}(0))\cap B_{R}(0)$ { independent of $r>0$}, then we have $\rho^{\kappa}_t((\partial \zonea+B_{r}(0))\cap B_{R}(0))\leq \Ctwentytwo\int \phi_{r}(v)d\rho^{\kappa}_t(v)$ for some constant $\Ctwentytwo>0$ { independent of $r>0$}. So to prove the lemma, it is enough to show that for any $\epsilon$, we have $\sup_{t\in [0,T^*]}\int \phi_{r}(v)d\rho^{\kappa}_t(v)\leq \epsilon$ for $r$ small enough.
	
	Firstly, we have that  $(\normm{v-v_{\alpha}(\rho^{\kappa}_t,v)}+\kappa)^{\frac{d}{d+1}}\rho^{\kappa}_t(v)$ is
	absolutely continuous with respect to the Lebesgue measure, see \cite[Theorem 6.3.1]{bogachev2022fokker}. Moreover, since
	$(\normm{v-v_{\alpha}(\rho^{\kappa}_t,v)}+\kappa)^{\frac{d}{d+1}}\geq \kappa^{\frac{d}{d+1}}>0$, we have that
	$\rho^{\kappa}_t$ is  absolutely continuous for any $t\in [0,T^*]$, and so by monotone convergence theorem and the properties that $\phi_r(v)$ is monotone increasing in terms of $r$, $\phi_0(v)=0$~(see Remark \ref{rmk:8}), we have for any $t\in [0,T^*]$ that $\int\phi_{r}(v)d\rho^{\kappa}_t(v)\to 0$ as $r\to 0$.  
	
	Secondly, define $r_{\epsilon}^*(t):=\sup\left\{r>0: \int \phi_r(v)d\rho^{\kappa}_t(v)\leq \epsilon\right\}$, then we have $r_{\epsilon}^*(t)>0$ for any $t\in [0,T^*]$. In the following, we want to show $r_{\epsilon}^*(t)\geq r_\epsilon>0$ for all $t\in [0,T^*]$, for some constant $r_\epsilon>0$. If not, without loss of generality, there will be a increasing sequence $\left\{t_i\right\}_{i=1}^{\infty}$, such that $\lim_{i\to\infty}t_i=t_0\in [0,T^*]$ and $\lim_{i\to\infty}r^*_{\epsilon}(t_i)=0$. For simplicity, we will denote $r_0:=r^*_{\epsilon}(t_0)$ and we know $r_0>0$.
	
	We have for any $r,t$ that
	\begin{equation}\label{eq:36}
		\begin{aligned}
			\frac{d}{dt}\int\phi_{r}(v)d\rho^{\kappa}_t(v)&=\int -\lambda\inner{v-v_{\alpha}(\rho^{\kappa}_t,v)}{\nabla \phi_{r}(v)}+\frac{\sigma^2}{2}\normmsq{v-v_{\alpha}(\rho^{\kappa}_t,v)}\Delta \phi_{r}(v)d\rho^{\kappa}_t(v)\\&\quad\in [-\Ctwentythree,\Ctwentythree],
		\end{aligned}
	\end{equation}
	for some $\Ctwentythree\in (0,\infty)$ depending on $r,R,d,\lambda,\sigma$, since $\phi_{r}(v)$ vanishes outside $B_{2r}(0)$, and inside $B_{2R}(0)$ we have $\nabla\phi_{r}(v)$ and $\Delta\phi_{r}(v)$ bounded. Based on the above equation, we have
	\begin{equation}
		\int\phi_{r_0}(v)d\rho^{\kappa}_{t_i}(v)\leq \int \phi_{r_0}(v)d\rho^{\kappa}_{t_0}(v)+\Ctwentythree\left(t_0-t_i\right)=\epsilon+\Ctwentythree\left(t_0-t_i\right).
	\end{equation}
	On the other hand, we have for any $\theta>0$ that
	\begin{equation}
		\begin{aligned}
			\int \phi_{r_0-\theta}(v)d\rho^{\kappa}_{t_i}(v)=\int\phi_{r_0}(v)d\rho^{\kappa}_{t_i}(v)+\underbrace{\int\left[\phi_{r_0-\theta}(v)-\phi_{r_0}(v)\right]d\rho^{\kappa}_{t_i}(v)}_{(I)},
		\end{aligned}
	\end{equation}
	while we know that the term $(I)$ is strictly negative, because $\phi_r$ is monotonically increasing with respect to $r$.
 {By Lemma \ref{lem:5l} we have \begin{equation}
     \rho^{\kappa}_t(B_{\rr}(v))\geq e^{-q'T^*}\int \phi_{\rr}^{\tau}(z-v)d\rho^{\kappa}_0(z)>0,\end{equation}
 for any $t\in [0,T^*]$ and $\rr>0,v\in(\partial\zonea\pm (2r-\frac{\theta}{2}))\cap B_R(0)$, here we use $\partial\zonea\pm (2r-\frac{\theta}{2})$ to denote the set $\cup_{v\in\partial\zonea}v\pm (2r-\frac{\theta}{2}))N(v)$ where $N$ is the Gauss map. The constant $q'$ is from Lemma \ref{lem:5l}. 
 
 Then due to $\phi_{r_0-\theta}(v)-\phi_{r_0}(v)\in [-\Ctwentyfour,-\Ctwentyfive],\quad \forall v\in (\partial\zonea\pm(2r-\frac{\theta}{2}))\cap B_{R/2}(0)$, where $\Ctwentyfour,\Ctwentyfive>0$ only depend on $r_0,\theta$, we have $(I)\leq - \Ctwentysix e^{-q'T^*}\int \phi_{\frac{\theta}{4}}^{\tau}(z-v)d\rho^{\kappa}_0(z)<0$ for any $v'\in (\partial\zonea\pm(2r-\frac{\theta}{2}))\cap B_{R/2}(0)$, where $\Ctwentysix$ depends on $r_0,\theta$.} And so we have
	\begin{equation}
		\begin{aligned}
			\int \phi_{r_0-\theta}(v)d\rho^{\kappa}_{t_i}(v)&\leq \int\phi_{r_0}(v)d\rho^{\kappa}_{t_i}(v)- \Ctwentysix e^{-q'T^*}\int \phi_{\frac{\theta}{4}}^{\tau}(z-v')d\rho^{\kappa}_0(z)\\
			&\leq \epsilon+\Ctwentythree\left(t_0-t_i\right)- \Ctwentysix e^{-q'T^*}\int \phi_{\frac{\theta}{4}}^{\tau}(z-v')d\rho^{\kappa}_0(z),
		\end{aligned}
	\end{equation}
 for any $v'\in (\partial\zonea\pm(2r-\frac{\theta}{2}))\cap B_{R/2}(0)$.
 
	Thirdly, based on the above analysis,  we can choose $\theta = \frac{1}{2}r_0$, and $i_0$ large enough so that for any $i\geq i_0$, we have $\Ctwentythree\left(t_0-t_i\right)- \Ctwentysix e^{-q'T^*}\int \phi_{\frac{\theta}{4}}^{\tau}(z-v')d\rho^{\kappa}_0(z)< 0$ and $r^*_{\epsilon}(t_i)\leq \frac{1}{3}r_0$. However in this case, we have
	\begin{equation}
		\int \phi_{\frac{r_0}{2}}(v)d\rho^{\kappa}_{t_i}(v)< \epsilon,
	\end{equation}
	for any $i\geq i_0$, and this means $r^*_{\epsilon}(t_i)>\frac{1}{2}r_0$ when $i\geq i_0$, which is a contradiction. So we have proved that $r_\epsilon>0$ for any $t\in [0,T^*]$. 
	
	Lastly, we want to show $r_{\epsilon}$ will not vanish when $\alpha\to\infty$. Since $\lim_{\alpha\to\infty}v_{\alpha}(\rho,v)= V^*(v)$, for almost every $v\in\mathbb{R}^d$ and $\rho$ with $\operatorname{supp}(\rho)=\mathbb{R}^d$,  the Fokker--Planck equation will converge to 
	\begin{equation}
		{\partial_ t}\rho_t=\lambda\operatorname{div}\left((v-V^*(v))\rho_t\right)+\frac{\sigma^2}{2}\Delta\left((\normm{v-V^*(v)}+\kappa)^2\rho_t\right),
	\end{equation}
	and $\rho_t$ is absolutely continuous~(also due to \cite[Theorem 6.3.1]{bogachev2022fokker}); hence, in view of the arguments above for this $\rho_t$, we have $r_\epsilon>0$ when $\alpha=\infty$, and therefore $r_{\epsilon}$ cannot vanish as $\alpha\to\infty$.

\end{proof}
\begin{remark}\label{rmk:8}
	We can construct $\phi_{r}(v)$ as follows:  $	\phi_{r}^{\tau}(v)$ is \revised{defined} as in  \eqref{eq:first} and it is easy to construct a barrier function $\Lambda_{R}(v)$ which is a smooth function on $\mathbb{R}^d$ such that it equals $1$ inside $B_{R}(0)$ and equals $0$ outside $B_{2R}(0)$. Based on $\phi_{r}^{\tau}(s)$ and $\Lambda_{R}(v)$, we can define $\phi_{r}(v):=\phi_{2r}^{\tau}(\frac{\operatorname{dist}^2(v,\partial \zonea)}{2r})\Lambda_{R}(v)$, then we know with fixed $v$, $\phi_{r}(v)$ is monotone increasing in terms of $r$.
\end{remark}

\paragraph{On the convergence of the fully discrete finite particle iterations \eqref{eq:alg}.}
We conclude this Appendix here with the estimates that explain the rationale of the convergence of the fully discrete finite particle algorithm \eqref{eq:alg}. In the following $\overline V_t^i$ represent i.i.d. realizations of a solution of \eqref{eq:n3} for $i=1,\dots,N$, and $V_t^i$ satisfies 
\begin{equation}\label{eq:empSDE}
d {V}^i_t
		=-\lambda\left(V^i_t-v_\alpha(\hat \rho_t^N, {V}^i_t )\right)dt+\sigma \left(\|V^i_t-v_\alpha(\widehat \rho^N_t,{V}^i_t)\|_2+\kappa\right) d B_t^i,\quad i=1,\dots,N.
\end{equation}
{
Denote $V_t:=((V_t^1)^\top,\ldots,(V_t^N)^\top)^{\top}$, then the above emprical SDE can be written in a compact way
\begin{equation}
    dV_t=G(V_t)dt+H(V_t)dB_t,
\end{equation}
where $G=((G^1)^\top,\ldots,(G^N)^\top)^\top,H=\operatorname{Diag}(H^1,\ldots,H^N)$ and $G(V_t)^i=-\lambda(V_t^i-v_{\alpha}(\hat{\rho}^N_t,V_t^i),$\\$H^i(V_t):=\sigma \left(\|V^i_t-v_\alpha(\widehat \rho^N_t,{V}^i_t)\|_2+\kappa\right)I_d$. Then we can show $H,G$ are locally Lipschitz and $H,G$ grow at most linearly,
since $v_{\alpha}(\hat{\rho}^N_t,V_t^i)$ is a convex combination between $\{V_t^i\}_{i=1}^N$ 
and so $\normm{v_{\alpha}(\hat{\rho}^N_t,V_t^i)}\leq \max_{i\in[N]}\{\normm{V_t^i}\}\leq \normm{V_t}$. Hence, we have that the above empirical SDE has a unique strong solution. 




}

Once again we consider the set of bounded processes
\begin{equation}
    \Omega_M= \left \{\sup_{t\in[0,T]} \frac{1}{N}\sum_{i=1}^N \max\left\{\N{V_t^i}_2^4,\N{\overbar{V}_t^i}_2^4\right\} \leq M\right \},
\end{equation} for any $M>0$, as we did in \eqref{eq:bndp}.
In analogy to the chain of inequalities \eqref{B1-ownresult}, we should study instead $\frac{1}{N}\sum_{i=1}^N\Ep{\normmsq{F(V_{K}^i)}}$, where $F(v)=v-V^*(v)$, then we have 
\begin{equation}
   \begin{aligned}
        \frac{1}{N}\sum_{i=1}^N\Ep{\normmsq{F(V_{K}^i)}\Bigg | \Omega_M}&=\underbrace{\frac{1}{N}\sum_{i=1}^N\Ep{\normmsq{F(\overline{}{V}_{T^*}^i)}\Bigg | \Omega_M}}_{I}\\
        &+\underbrace{\frac{1}{N}\sum_{i=1}^N\Ep{\normmsq{F({V}_{T^*}^i)}\Bigg | \Omega_M}-\frac{1}{N}\sum_{i=1}^N\Ep{\normmsq{F(\overline{V}_{T^*}^i)}\Bigg | \Omega_M}}_{II}\\
        &+\underbrace{\frac{1}{N}\sum_{i=1}^N\Ep{\normmsq{F({V}_{K}^i)}\Bigg | \Omega_M}-\frac{1}{N}\sum_{i=1}^N\Ep{\normmsq{F({V}_{T^*}^i)}\Bigg | \Omega_M}}_{III},
   \end{aligned}
\end{equation}
where $K\Delta t=T^*$.
We have $I\leq \epsilon$ by the main result of this paper, Theorem \ref{thm:main0}. From the estimate \eqref{eq:lala} and the analysis of \eqref{eq:104104}, we know that, for any transport coupling $\pi$ between $\rho_1$ and $\rho_2$, we have 
\begin{equation}
    \Big|\int \normm{F(v)}^2d\rho_1(v)-\int\normm{F(w)}^2d\rho_2(w)\Big|\leq C\sqrt{\iint\normmsq{v-w}d\pi(v,w)},
\end{equation}
so for term $II$, we obtain 
\begin{equation}\label{eq:qmfl2}
    II\leq C\Ep{\sqrt{\frac{1}{N}\sum_{i=1}^N{\normmsq{\overline{V}_{T^*}^i-V_{T^*}^i}}}\Bigg | \Omega_M}\leq C\sqrt{\Ep{{\frac{1}{N}\sum_{i=1}^N{\normmsq{\overline{V}_{T^*}^i-V_{T^*}^i}}}\Bigg | \Omega_M}}\leq C_{\rm MFA}\frac{1}{\sqrt{N}},
\end{equation}
where the last inequality is due to a similar analysis as in \cite[Proposition 3.11]{B1-fornasier2021global} and $C_{\rm MFA}>0$ may depend exponentially on $M>0$.

Similarly, and using standard results from \cite{platen1999introduction}, for term $III$, we have
\begin{equation}
    \begin{aligned}
        &III\leq C\sqrt{\Ep{\frac{1}{N}\sum_{i=1}^N{\normmsq{{V}_{{K}}^i-V_{T^*}^i}}\Bigg | \Omega_M}}\leq C_{\rm NA}{(\Delta t)^s},
    \end{aligned}
\end{equation}
for some index $s>0$. Then by using the same arguments as in \eqref{B1-ownresult2}, we conclude 
\begin{equation}
     \mathbb P\Bigg( \frac{1}{N} \sum_{i=1}^N\normmsq{V_{K}^i-V^*(V_K^i)} \leq \varepsilon\Bigg) \geq 1 - \left [ \varepsilon^{-1} (C_{\mathrm{NA}}\Delta t+C_{\mathrm{MFA}} N^{-1}+\epsilon) -\mathbb P (\Omega_M^c) \right ].  
\end{equation}
{Unfortunately, a uniform bound with respect to $N$ of the moments of the solution of \eqref{eq:empSDE} remains elusive, because for $v_\alpha(\rho,v)$ in estimates like \eqref{eq:2626} the constant $b_2$ depends on $\|v\|_2^2$ (compare also \cite[Lemma 3.3]{carrillo2018analytical}). This gap does not allow us yet to provide an explicit estimate of $\mathbb P  (\Omega_M^c)$, uniformly with respect to $N$ as done, e.g., in \cite[Lemma 3.10]{B1-fornasier2021global}.}

\section*{Acknowledgements and Competing Interests}

%

This work was initiated while LS visited the Chair of Applied
Numerical Analysis of the Technical University of Munich in June 2023. LS thanks the hospitality of the Chair of Applied
Numerical Analysis of the Technical University of Munich, Konstantin Riedl for discussion and the support of Prof. Peter Richt\'arik. MS acknowledges the partial support of the Munich Center for Machine Learning.

\bibliography{bib}
\bibliographystyle{abbrv}
\end{document}